\documentclass[a4paper,reqno,11pt]{amsart}

\usepackage{hyperref}
\usepackage{amsmath}
\usepackage{amssymb}
\usepackage{cases}
\usepackage{enumitem}
\allowdisplaybreaks[4]
\usepackage{tikz-cd}
\usepackage{inputenc}

\input xy
\xyoption{all}

\AtBeginDocument{%
   \def\MR#1{}
}

\newtheorem{thm}{Theorem}[section]
\newtheorem{prop}[thm]{Proposition}

\newtheorem{lem}[thm]{Lemma}
\newtheorem{q}[thm]{Question}

\newtheorem{cor}[thm]{Corollary}

\newtheorem{claim}[thm]{Claim}

\theoremstyle{definition}
\newtheorem{definition}[thm]{Definition}
\newtheorem{example}[thm]{Example}

\theoremstyle{remark}
\newtheorem{remark}[thm]{Remark}

\numberwithin{equation}{section}

\newcommand{\bQ}{\mathbb{Q}}

\newcommand{\rounddown}[1]{\lfloor{#1}\rfloor}

\newcommand\FF{{\mathcal{F}}}
\newcommand\HH{{\mathcal{H}}}
\newcommand\GG{{\mathcal{G}}}

\newcommand\Sing{{\text{\rm Sing}}}

\newcommand\Nklt{{\rm{Nklt}}}
\newcommand\nklt{{\rm{Nklt}}}

\newcommand\Nlc{{\rm{Nlc}}}

\newcommand\Supp{{\rm{Supp}}}

\newcommand{\bb}[1]{\mathbb{#1}}

\newcommand{\cal}[1]{\mathcal{#1}}
\newcommand{\ff}[1]{\mathfrak{#1}}

\begin{document}

\title
[Applications of the MMP for co-rank 1 foliations on 3-folds]
{Local and global applications of the Minimal Model Program for co-rank one foliations on threefolds}
\date{\today}

\author{Calum Spicer}
\address{Department of Mathematics, Imperial College London, London SW7 2AZ, United Kingdom.}
\email{calum.spicer@imperial.ac.uk}

\author{Roberto Svaldi}
\address{EPFL, SB MATH-GE, MA B1 497 (B\^{a}timent MA), Station 8, CH-1015 Lausanne, Switzerland.}
\email{roberto.svaldi@epfl.ch}

\begin{abstract}
We provide several applications of the minimal model program to the local and global study of co-rank one foliations on threefolds.
Locally, we prove a singular variant of Malgrange's theorem, a classification of terminal foliation singularities
and the existence of separatrices for log canonical singularities.
Globally, we prove termination of flips, a connectedness theorem on lc centres, a non-vanshing theorem 
and some hyperbolicity properties of foliations.
\end{abstract}

\maketitle
\tableofcontents

\section*{Introduction}
Our primary goal in this paper is to use techniques and ideas from the foliated minimal model program (MMP) to deduce some
structural and dynamical results for foliation singularities.  Along the way of proving these results we further develop the MMP and explore
some applications of these developments to global properties of foliations.

\subsection*{Local results}

To every foliation singularity the MMP associates a numerical invariant, the discrepancy, which measures how 
the canonical class of the foliation changes under blow ups.

Our local results explore to what extent this numerical invariant
characterizes the structural and dynamical behavior of the foliation singularity.  
Here we are mostly interested in three classes of foliation singularities which are defined according to the behavior of the discrepancy,
namely, terminal, canonical and log canonical, see Definition \ref{defn_sing_of_fmmp}.  Terminal singularities can be viewed as 
the mildest class of singularities
of the MMP whereas log canonical are the most severe. For a two quick illustrations of this principle, 
terminal foliations on smooth surfaces are smooth foliations, and simple singularities (see Definition \ref{can.form.simple}) are canonical,
we refer the reader to Section \ref{ss_comparing_classes_of_sing} for a further 
discussion on the relations and parallels between the classes of singularities of the MMP
and other classes of singularities.

The singularities of the MMP are birational generalizations of nice classes of foliation singularities and are natural from the perspective of the geometry of foliations.
For example, canonical foliation singularities  play a central role in the study of hyperbolicity properties of surfaces by McQuillan, see \cite{McQuillan98, McQuillan08},
and the classification of foliations with trivial canonical bundle, see \cite{LPT11, Druel18}. 

We remark that simple singularities (which are roughly analogous to the singularities of smooth normal crossings pairs)
are not preserved by the operations of the MMP, since the underlying variety may become
singular in the course of the MMP.  In other words,
if one seeks to improve the global geometry of the foliation (by making $K_{\cal F}$ more positive) one loses some control
on the local geometry of the foliation.  As a result, canonical singularities
should be viewed as a compromise.  They are flexible enough to allow for the operations of the MMP, but mild enough to have
many of the same desirable properties as simple singularities.

A fundamental result in the study of singular foliations on smooth varieties is a theorem of Malgrange~\cite{Malgrange76} asserting that the classical Frobenius integrability criterion holds even in the presence of foliated singularities, provided that the codimension of the singular locus of the foliation is at least 3.
We prove a version of Malgrange's theorem on singular threefolds.

\begin{thm}
[= Theorem \ref{t_Malgrange}]
\label{Malg.thm.intro}
Let $P \in X$ be a germ of an isolated (analytically) $\bb Q$-factorial threefold singularity with a co-rank 1 foliation $\cal F$. 
Suppose that 
$\cal F$ has an isolated canonical singularity at $P$.

Then $\cal F$ admits a holomorphic first integral.
\end{thm}

The above statement is close to optimal, cf.~\cite[Examples~1.1-1.3]{CLN08}.

As a consequence of the above theorem, we obtain the following strong classification result.

\begin{thm}[= Theorem \ref{t_3fold_term}]
\label{3foldtermintro}
Let $P \in X$ be a germ of normal threefold with a co-rank 1 foliation $\cal F$ with terminal singularities.
Then $\cal F$ admits a holomorphic first integral.

Moreover, up to a $\bb Z/n\bb Z \times \bb Z/m\bb Z$-cover, $\cal F$ admits a holomorphic first integral $\phi\colon (P\in X) \rightarrow (0 \in \bb C)$, where $\phi^{-1}(0)$ is a Du Val surface singularity and $\phi^{-1}(t)$ is smooth for $t \neq 0$.  In particular, $X$ is terminal.
Moreover, $(X, \cal F)$ fits into the finite list of families contained in Proposition \ref{prop_cart_case}.
\end{thm}

We remark that in Theorem \ref{3foldtermintro} we make no assumption on the singularities of the underlying space other than normality.

We also remark that Theorems \ref{Malg.thm.intro} and \ref{3foldtermintro} 
should be viewed as analogues of the classification of terminal and canonical singularities
on threefolds.  These classification results have been crucial in understanding the global geometry of threefolds, as well as
the moduli space of surfaces, and we expect that the above results to play a corresponding role in the study of
the global geometry of foliations of threefolds and moduli of surface foliations.

We next  prove a result on the existence of separatrices of log canonical foliation singularities. 
Loosely speaking a separatrix may be thought of as a local solution to the differential equation defining the foliation, see Definition 
\ref{defn_sep} for a precise definition.
It is an interesting and challenging problem to decide when a foliation singularity admits a separatrix.
The existence of a (converging) separatrix is an essential element in the study of foliations singularities as 
it provides a way to ``organize" the dynamics around the foliation singularity.
While separatrices do not necessarily exist for a general foliation singularity, 
we prove their existence for log canonical singularities.

\begin{thm}
[= Theorem \ref{t_lc_sep}]
\label{t_lc_sep_intro}
Let $\cal F$ be a germ of a log canonical foliation singularity on $0 \in \bb C^3$.  Then $\cal F$ admits a separatrix.
\end{thm}

Our strategy of proof actually provides a more general version of this result allowing the underlying analytic germ to be singular.

In \cite{CC92} it is shown that non-dicritical foliation singularities always admit separatrices, confirming a conjecture of R. Thom.  
Log canonical singularities are in general dicritical and so \cite{CC92} does not apply to prove existence of separatrices for this class of singularities.
Theorem \ref{t_lc_sep_intro} is also closely related to a local analogue of a conjecture of Brunella that has been formulated in \cite{CRVS15} and explored in \cite{CRV15}.

We remark that results analogous to Theorem \ref{Malg.thm.intro} and Theorem \ref{t_lc_sep_intro} have been shown in \cite{CLN08}
and \cite{MolinaSamper19}, respectively, under differing assumptions on the singularities of the foliation and variety.
An advantage of our statements is that they hold for very natural classes of singularities which are satisfied by a wide range of foliations.

Classically, the technique of inversion of adjunction has proven crucial for understanding singularities by providing a precise relation between the singularities
of a variety and the singularities of a divisor in the variety. 
We prove a foliated analogue of this result which should prove equally useful in the study of foliation singularities.

\begin{thm}[= Theorem \ref{t_inversion_of_adjunction}]
Let $X$ be a $\mathbb{Q}$-factorial threefold and let $\FF$ be a co-rank one foliation.
Consider a prime divisor $S$ and an effective $\mathbb{Q}$-divisor $\Delta$ on $X$ which does not contain $S$ in its support.  
Let $\nu\colon S^{\nu}\to S$ be the normalization and let $\cal G$ be the restricted foliation to $S^\nu$
and write $\nu^*(K_{\cal F}+\Delta) =K_{\cal G}+\Theta$.  
Suppose that 
\begin{itemize}
    \item if $S$ is transverse to $\FF$, then $(\GG, \Theta)$ is lc;
    \item if $S$ is $\FF$-invariant, then $(S^\nu, \Theta)$ is lc.
\end{itemize}
Then $(\FF, \epsilon(S)S+ \Delta)$ is lc in a neighborhood of $S$.
\end{thm}

We refer the reader to \S \ref{inv.adj.sect} for a discussion of this result and its relationship to the adjunction formula for lc pairs, cf. \S \ref{adj.sect}.

We now take a moment to explain some of the key ideas of the proofs in the above statements. 
Indeed, our central innovation is the systematic use of F-dlt modifications to study foliation singularities, 
see Theorem \ref{t_existencefdlt}
for a recollection on the definition and existence of F-dlt modifications which were first shown to exist in \cite{CS18}.  
An F-dlt modification (which is a foliated analogue
of a classical dlt modification) is a special kind of partial resolution which extracts divisors for which the global properties of the foliation restricted to these divisors strongly reflect the local properties of the foliation singularity.
In particular, for dicritical singularities, an F-dlt modification will always extract one exceptional geometric valuation transverse to the foliation.

To prove Theorem \ref{t_lc_sep_intro} we extract, by way of an F-dlt modification, 
an exceptional divisor which is transverse to the foliation (in other words, a dicritical component of the singularity).
Showing the existence of a separatrix is then reduced to producing a global invariant algebraic divisor 
for the restricted foliation on this exceptional divisor.  An adjunction calculation shows
that this restricted foliation has trivial first chern class, and so the existence of an invariant algebraic divisor
is a consequence of the classification of foliations with trivial first chern class.

To prove Theorem \ref{Malg.thm.intro} we provide a precise bound on the singularities of $X$ (we show
that $X$ is klt) by controlling the geometry of the invariant divisors on an F-dlt modification.  
We then show that the singularities of $X$ are mild enough to allow us to prove the existence
of a holomorphic Godbillon-Vey sequence associated to the foliation, \S \ref{godb.sect}, and we may then conclude roughly along the lines
of Malgrange's original proof.


\subsection*{Global results}

In \cite{Spicer17} and \cite{CS18} much the of minimal model program for rank 2 foliations on threefolds was completed, including a cone and contraction theorem, existence of flips and special termination.  
However, the termination of flips was not proven.  
In this paper, we termination of flips, thereby completing the statement of the MMP for F-dlt pairs.
We refer the reader to Definition \ref{defn_fdlt} for the definition of 
F-dlt singularities, but we emphasize here that they 
are a very large and natural class of foliated singularities; for example, they include pairs 
$(X, \FF)$ such that $X$ is smooth, $\cal F$ has simple singularities.

\begin{thm}[= Theorem \ref{t_term}]
Let $X$ be a $\bb Q$-factorial quasi-projective threefold.
Let $(\cal F, \Delta)$ be an F-dlt pair.
Then starting at $(\cal F, \Delta)$ there is no infinite sequence
of flips.
\end{thm}

A direct consequence of termination and the work in \cite{CS18} is the following non-vanishing theorem.

\begin{thm}[= Theorem \ref{t_nonvan}]
Let $\mathcal F$ be a co-rank one foliation on a normal projective $\bb Q$-factorial threefold $X$. 
Let $\Delta$ be a $\bb Q$-divisor such that $(\mathcal F,\Delta)$ is a F-dlt pair. 
Let $A\ge 0$ and $B\ge 0$ be $\bb Q$-divisors such that $\Delta= A+B$ and $A$ is ample. Assume that $K_{\mathcal F}+\Delta$ is 
pseudo-effective 

Then $K_{\mathcal F}+\Delta \sim_{\bb Q} D \geq 0$.
\end{thm}

We then turn our attention to the study of the non-klt centres
of foliations.  One of our central results in this direction is the proof of a foliated analogue of the connectedness of non-klt centres.

\begin{thm}[= Theorem \ref{conn.f-dlt.thm}]
Let $X$ be a projective $\mathbb{Q}$-factorial threefold and let $\FF$ be a rank $2$ foliation on $X$.
Let $(\FF, \Delta)$ be an $F$-dlt pair on $X$.
Assume that $-(K_\FF+\Delta)$ is nef and big.
Then $\nklt(\FF, \Delta)$ is connected.
\end{thm}

One of the fundamental ideas of the foliated MMP is that the negativity of foliated log pairs $(\FF, \Delta)$ 
with mild singularities is governed by the presence of rational curves, see, for example, \cite{Spicer17}.
As a final application we prove a foliated version of the main result of \cite{Svaldi14}
which relates the hyperbolicity of a foliation to an analysis of the log canonical singularities
of a foliation. 
Given a foliated pair $(\FF, \Delta)$ and an lc center $S$ we will denote by $\bar{S} \subset S$ the locally closed subvariety obtained by removing from $S$ the lc centers of $(\FF, \Delta)$ strictly contained in $S$.

\begin{thm}[= Theorem \ref{hyperb.thm}, Foliated Mori hyperbolicity] 
\label{hyperb.thm.intro}
Let $(\FF, \Delta)$ be a foliated log canonical pair on a normal projective threefold $X$. 
Assume that 
  \begin{itemize}
   \item $X$ is klt,
   \item there is no non-constant morphism $f : \mathbb{A}^1 \to X \setminus \nklt(\FF, \Delta)$ which is tangent to $\cal F$, and
   \item for any stratum $S$ of $\nklt(\FF, \Delta)$ there is no non-constant morphism $f : \mathbb{A}^1 \to \bar{S}$ which is tangent to $\FF$.
  \end{itemize}
Then $K_\FF + \Delta$ is nef.
\end{thm}

Finally, we remark that the central idea in the proof of our connectedness and hyperbolicty 
results is a refinement of the technique of F-dlt modifications, see Theorem \ref{t_sp_dlt_mod}, and a careful analysis
of the properties of F-dlt modifications through adjunction.

\subsection*{Acknowledgements}
We would like to thank Paolo Cascini, Michael McQuillan and Jorge V. Pereira for many valuable conversations, suggestions and comments.
We also thank the referees for many suggestions and improvements to the exposition of the paper.
\newline
This project was started during a visit of CS and RS to FRIAS; we would like to thank FRIAS
and Prof. Stefan Kebekus for their hospitality and the financial support during the course of our visit.
Most of this work was completed during several visits of RS to Imperial College, London. 
RS would like to thank Imperial College for the hospitality and the nice working environment. 
He would also like to thank University of Cambridge and Churchill College, Cambridge where he was a research fellow when part of this work was completed.
\newline
RS was partially supported from the European Union's Horizon 2020 research and innovation programme under the Marie Sk\l{}odowska-Curie grant agreement No. 842071.

\section{Preliminaries}

\subsection*{Notations and conventions}
By the term variety, we will always mean an integral separated 
scheme over an algebraically closed field $k$. 
Unless otherwise stated, it will be understood that $k = \mathbb{C}$.
\newline
Unless otherwise specified, we adopt the same notations and conventions as in~\cite{KM98}.

A \emph{contraction} is a projective morphism $f\colon X \to Y$ of quasi-projective varieties with $f_\ast \mathcal{O}_X = \mathcal{O}_Z$.
If $X$ is normal, then so is $Z$ and the fibers of $f$ are connected.
A proper birational map $f \colon X \dashrightarrow Y$ of normal quasi-projective varieties is a \emph{birational contraction} if $f^{-1}$ does not contract any divisor.

Given a Weil $\mathbb{R}$-divisor $D$ and a prime divisor $E$ on a normal variety $X$, we will denote by $\mu_E D$ the coefficient of $E$ in D.
If $D$ is an Weil $\mathbb{R}$-divisor on $X$ then for any $c \in \mathbb{R}$ we define 
\[
D^{*c} := \sum_{\mu_{E}D \; \ast \; c} \mu_{E}D \; E,
\]
where $\ast$ is any of $=, \geq, \leq, >, <$.
We define the {\it round down} $\lfloor D \rfloor$ of $D$ to be $\sum_{i=1}^n \lfloor \mu_{P_i} D \rfloor P_i$.
The {\it round up} $\lceil D \rceil$ of $D$ is defined analogously.
The {\it fractional part} $\{D\}$ of $D$ is defined as $\{D\} := D - \lfloor D \rfloor$.

The \textit{support} $\mathrm{Supp}(D)$ of an $\mathbb{R}$-divisor $D$ is the union of the prime divisors appearing in $D$ with non-zero coefficient, $\mathrm{Supp}(D):=\bigcup_{\mu_{E} D\neq 0}E$.

\subsection{Recollection on foliations}
\label{ss_recollectionsonfol}

We refer the reader to \cite{Brunella00} for basic notions in foliation theory. 

A {\bf foliation} on a normal variety $X$ is a coherent subsheaf $\cal F \subset T_X$ such that
\begin{enumerate}
\item $\cal F$ is saturated, i.e. $T_X/\cal F$ is torsion free, and

\item $\cal F$ is closed under Lie bracket.
\end{enumerate}

The {\bf rank} of $\cal F$ is its rank as a sheaf.  Its {\bf co-rank} is its co-rank as a subsheaf
of $T_X$.

Let $X$ be a normal variety and let $\cal F$ be a rank $r$ foliation
on $X$.  
A {\bf canonical divisor} of $\cal F$ is a divisor $K_{\cal F}$
such that $\cal O_X(-K_{\cal F}) \cong \det (\cal F)$.
We define the {\bf normal sheaf} of 
$\cal F$ as $\mathcal N_{\cal F}:= (T_X/\cal F)^{**}$. The {\bf conormal sheaf} $\mathcal N_{\cal F}^*$ of $\cal F$ is the dual of $\cal N_{\cal F}$.
If $\cal F$ is a foliation of co-rank one then, by abuse of notation, we denote by $N^*_{\cal F}$ a divisor associated to $\cal N^*_{\cal F}$. 

We can associate to $\cal F$ a morphism
\[\phi\colon \Omega^{[r]}_X \rightarrow \cal O_X(K_{\cal F})\]
defined by taking the double dual of the  $r$-wedge product of the map $\Omega^{[1]}_X\to \cal F^*$, induced by the inclusion
$\cal F \subset T_X$.  This yields a map
\[
\phi'\colon (\Omega^{[r]}_X \otimes \cal O_X(-K_{\cal F}))^{**} \rightarrow \cal O_X\]
and we define the {\bf singular locus} of $\cal F$,   denoted by $\text{sing}(\cal F)$, to be the cosupport of the image of $\phi'$.

Given a dominant rational map $f\colon Y \dashrightarrow X$ and a foliation $\cal F$ on $X$ we may
pullback $\cal F$ to a foliation on $Y$, denoted $f^{-1}\cal F$.

\begin{remark}
\label{rem_cover_smooth}
If  $q\colon X' \rightarrow X$ is a quasi-\'etale cover
and $\cal F' = q^{-1}\cal F$ then $K_{\cal F'}=q^*K_{\cal F}$ and 
\cite[Proposition 5.13]{Druel18} implies that $\cal F'$ is non-singular if and only if
$\cal F$ is.  In particular, it is not the case that we always have $\text{sing}(X) \subset \text{sing}(\cal F)$.
\end{remark}

\subsection{Invariant subvarieties}
\label{s_invariant}

Let $X$ be a normal variety and let $\cal F$ be a rank $r$ foliation on $X$. 
Let $S\subset X$ be a subvariety. Then  $S$ is said to be  {\bf $\cal F$-invariant}, or {\bf invariant by }$\cal F$, if for any open subset $U\subset X$ and any section $\partial \in H^0(U,\cal F)$, we have that 
\[ \partial (\mathcal I_{S\cap U})\subset \mathcal I_{S\cap U},
\]
where $\mathcal I_{S\cap U}$ denotes the ideal sheaf of $S\cap U$. 
If  $D\subset X$ is a prime divisor then we define $\epsilon(D) = 1$ if $D$ is not $\cal F$-invariant and $\epsilon(D) = 0$ if it is $\cal F$-invariant. 

\subsection{Foliation singularities}
Frequently in birational geometry it is useful to consider pairs $(X, \Delta)$ where $X$ is a normal variety, and $\Delta$ is a $\bb Q$-Weil divisor such that $K_X+\Delta$ is $\bb Q$-Cartier.  
We refer the reader to~\cite[{\S~2}]{KM98} for the relevant definitions and notations on the singularities of pairs.
We will use the following definition that we recall as it is non-standard.
\begin{definition}
\label{def:potent.klt}
A normal variety $X$ is said to be potentially lc (resp. potentially klt) if there exists an effective $\mathbb R$-divisor $D$ on $X$ such that the log pair $(X, D)$ has lc (resp. klt singularities).
\end{definition}

It is possible to define singularities for pairs also in the foliated world, in analogy with the classical case of pairs.
\begin{definition}
A {\bf foliated pair} $(\cal F, \Delta)$ is a pair of a foliation and a $\bb Q$-Weil 
($\bb R$-Weil) divisor
such that $K_{\cal F}+\Delta$ is $\bb Q$-Cartier ($\bb R$-Cartier).
\end{definition}

Note also that we are typically interested only in the cases when $\Delta \geq 0$,
although it simplifies some computations to allow $\Delta$ to have negative coefficients.

Given a birational morphism $\pi: \widetilde{X} \rightarrow X$ 
and a foliated pair $(\cal F, \Delta)$ on $X$ 
let $\widetilde{\cal F}$ be the pulled back foliation on $\tilde{X}$ and $\pi_*^{-1}\Delta$ be the strict
transform.
We can write
\[K_{\widetilde{\cal F}}+\pi_*^{-1}\Delta=
\pi^*(K_{\cal F}+\Delta)+ \sum a(E_i, \cal F, \Delta)E_i.\]

\begin{definition}
\label{defn_sing_of_fmmp}
We say that $(\cal F, \Delta)$ is {\bf terminal, canonical, log terminal, log canonical} if
$a(E_i, \cal F, \Delta) >0$, $\geq 0$,
$> -\epsilon(E_i)$, $\geq -\epsilon(E_i)$, respectively, where
$\epsilon(D) = 0$ if $D$ is invariant and 1 otherwise and where $\pi$
varies across all birational morphisms.

If $(\cal F, \Delta)$ is log terminal and $\lfloor \Delta \rfloor = 0$
we say that $(\cal F, \Delta)$ is {\bf foliated klt}.
\end{definition}

Notice that these notions are well defined, i.e., $\epsilon(E)$ and $a(E, \cal F, \Delta)$
are independent of $\pi$.  We say $a(E, \cal F, \Delta)$ is the discrepancy of $E$ (with respect to $(\cal F, \Delta)$),
or the foliated discrepancy.

Observe that in the case where $\cal F = T_X$ no exceptional divisor
is invariant, i.e., $\epsilon(E)=1$, and so this definition recovers the usual
definitions of (log) terminal, (log) canonical.

We remark that we will be using the terminal, etc. classification to refer to both the singularities of the foliation
and the singularities of the underlying variety.  If necessary we will use the term foliation terminal, etc. to emphasize the fact
that we are talking about the singularities of the foliation rather than the variety.

\begin{definition}
\label{nonklt.loc.def}
Let $(\cal F, \Delta)$ be foliated log pair.
\begin{enumerate}
\item 
We say that $W \subset X$ is a log canonical centre (lc centre) of $(\cal F, \Delta)$ provided $(\cal F, \Delta)$ is log canonical at the generic point of $W$ and there exists some divisor $E$ of discrepancy $-\epsilon(E)$ on some birational model over $X$ such that $E$ dominates $W$.
\item 
The nonklt locus $\nklt(\FF, \Delta)$ of $(\cal F, \Delta)$ is the union of the centres of all divisors divisorial valuations $E$ of discrepancy $\leq -\varepsilon(E)$.
\item
The nonlc locus $\Nlc(\FF, \Delta)$ of $(\cal F, \Delta)$ is the union of the centers of all divisorial valuations $E$ of discrepancy $<-\epsilon(E)$.
\end{enumerate}
\end{definition}

\begin{remark}
\label{nonklt.loc.def.rmk}
\begin{enumerate}
\item 
If $\epsilon(E_i) = 0$ for all exceptional divisors $E_i$
over a centre $W \subset X$, the notions of log canonical and canonical centre coincide for $W$.  
In this case, we will refer to canonical centres as log canonical centres.
\item 
Any $\cal F$-invariant divisor $D$ is an lc centre
of $(\cal F, \Delta)$ since $D$ shows up in $\Delta$ with coefficient at least $0 = \epsilon(D)$.
\\
Moreover, a direct computation shows that any strata of a simple singularity is an lc centre.
\end{enumerate}
\end{remark}

We have the following nice characterization due to \cite[Corollary I.2.2.]{McQuillan08}:

\begin{prop}
\label{prop_term_surf_class}
Let $0 \in X$ be a normal surface germ with a terminal foliation $\cal F$ of rank one.  

Then there exists a cyclic cover
$\sigma\colon  Y \rightarrow X$ such that $Y$ is a smooth surface and  $\sigma^{-1}\cal F$ is a smooth foliation.

In particular, by Remark \ref{rem_cover_smooth}, we have $0 \notin \text{sing}(\cal F)$.
\end{prop}

We emphasize that the above shows that even if $0 \in \text{sing}(X)$ it may be the case
that $0 \notin \text{sing}(\cal F)$.

We will also make use of the class of simple foliation singularities,
see \cite[Appendix: About simple singularities]{Cano}.

\begin{definition}
\label{can.form.simple}
We say that $p \in X$ with $X$ smooth is a {\bf simple singularity} for $\cal F$
provided in formal coordinates $x_1, ..., x_n$ around $p,$ $N^*_{\cal F}$ is generated by a $1$-form which is 
in one of the following two forms, where $1 \leq r \leq n$.

\begin{enumerate}
\item There are $\lambda_i \in \bb C^*$ such that
\begin{equation}
\label{can.form.simple.eqn1}
\omega = (x_1\cdot ... \cdot x_r)(\sum_{i = 1}^r \lambda_i \frac{dx_i}{x_i})
\end{equation}
and if $\sum a_i\lambda_i = 0$ for some non-negative integers $a_i$
then $a_i = 0$ for all $i$.

\item There is an integer $k \leq r$ such that
\begin{equation}
\label{can.form.simple.eqn2}
\omega=(x_1\cdot ... \cdot x_r)(\sum_{i = 1}^kp_i\frac{dx_i}{x_i} + 
\psi(x_1^{p_1}...x_k^{p_k})\sum_{i = 2}^r \lambda_i\frac{dx_i}{x_i})
\end{equation}
where $p_i$ are positive integers, without a common factor, $\psi(s)$
is a series which is not a unit, and $\lambda_i \in \bb C$
and if $\sum a_i\lambda_i = 0$ for some non-negative integers $a_i$
then $a_i = 0$ for all $i$.

\end{enumerate}
\end{definition}

By Cano, \cite{Cano}, every foliation on a smooth threefold admits a resolution
by blow ups centred in the singular locus of the foliation
such that the transformed
foliation has only simple singularities.

We recall the definition of non-dicritical foliation singularities,
see \cite[\S 2]{CM92}.

\begin{definition}
\label{defn_ndc}
Given a foliated pair $(X, \cal F)$ we say that 
$\cal F$ has {\bf dicritical} singularities if for 
some $P \in X$ there exists a germ of a surface $P \in S$
such that the restricted foliation $\cal F\vert_S$ has
infinitely many invariant curves passing through $\text{sing}(\cal F) \cap S$.

Otherwise, we say that $\cal F$ has {\bf non-dicritical} singularities.
\end{definition}

\begin{remark}
Observe that non-dicriticality implies that if $W$ is $\cal F$ invariant, then
$\pi^{-1}(W)$ is $\cal F'$ invariant.
\end{remark}

We remark that the above definition is equivalent on threefolds to the characterization appearing in \cite{CS18}, thanks to the existence
of resolution of singularities.
Namely, $\cal F$ has non-dicritical singularities if for any sequence of
blow ups $\pi:(X', \cal F') \rightarrow (X, \cal F)$
and any closed $q \in X$ we have $\pi^{-1}(q)$ is tangent to
the foliation.

\begin{definition}
\label{defn_sep}
Given a germ $0 \in X$ with a foliation $\cal F$
such that $0$ is a singular point for $\cal F$
we call a (formal) hypersurface germ $0 \in S$ a {\bf (formal) separatrix}
if it is invariant under $\cal F$.
\end{definition}

Note that away from the singular locus of $\cal F$
a separatrix is in fact a leaf.  Furthermore being non-dicritical 
implies that there are only finitely
many separatrices through a singular point. The converse of this statement is false.

\begin{definition}
\label{definitionlogsmooth}
Given a normal variety $X$, a co-rank one foliation $\cal F$ and a foliated pair $(\cal F, \Delta)$ we say that $(\cal F, \Delta)$
is {\bf foliated log smooth} provided the following hold:
\begin{enumerate}
\item $(X, \Delta)$ is log smooth.

\item $\cal F$ has simple singularities. 

\item If $S$ is the support of the non-invariant components of $\Delta$ then
for any $p \in X$ if $\Sigma_1, ..., \Sigma_k$ are the separatrices of $\cal F$ at $p$ (formal or otherwise),
then $S \cup \Sigma_1 \cup... \cup \Sigma_k$ is normal
crossings at $p$. 
\end{enumerate}

Given a normal variety $X$, a co-rank one foliation $\cal F$ and a foliated pair $(\cal F, \Delta)$ a 
{\bf foliated log resolution} is a high enough model 
$\pi\colon (Y, \cal G) \rightarrow (X, \cal F)$ so that $(Y, \pi_*^{-1}\Delta+\sum_i E_i)$ is foliated log smooth where 
the $E_i$ are all the $\pi$-exceptional divisors.

Such a resolution on threefolds is known to exist by \cite{Cano}
\end{definition}

We recall the class of F-dlt singularities introduced in \cite[Definition 3.6]{CS18}.

\begin{definition}
\label{defn_fdlt}
Let $X$ be a normal variety and let $\cal F$ be a co-rank one foliation on $X$.
Suppose that $K_{\cal F}+\Delta$ is $\bb Q$-Cartier.

We say $(\cal F, \Delta)$ is {\bf foliated divisorial log terminal (F-dlt)}
if 
\begin{enumerate}
\item each irreducible component of $\Delta$ is transverse to $\mathcal F$ and has coefficient at most 1, and
\item there exists a foliated log resolution $(Y, \cal G)$ of $(\cal F, \Delta)$ which only
extracts divisors $E$ of discrepancy $>-\epsilon(E)$.
\end{enumerate}
\end{definition}

In the case of surfaces F-dlt singularities have a particularly simple characterization.

\begin{lem}
\label{lem_Fdlt_surface}
Let $X$ be a normal surface and let $\cal F$ be a co-rank one foliation on $X$.
Suppose that $K_{\cal F}$ is $\bb Q$-Cartier and that $\cal F$ is F-dlt.
Then for all $P \in X$ one of the following holds:
\begin{enumerate}
\item $\cal F$ is terminal at $P$; or

\item $X$ is smooth at $P$ and $\cal F$ has a simple singularity at $P$.
\end{enumerate}
In particular, if $K_{\cal F}$ is Cartier then $X$ is smooth.
\end{lem}
\begin{proof}
Our dichotomy is a direct consequence of \cite[Lemma 3.8]{CS18}.
Our second claim follows from Proposition \ref{prop_term_surf_class} by observing that if $P \in X$ is terminal for $\cal F$
and $K_{\cal F}$ is Cartier near $P$ then $X$ is smooth.
\end{proof}

\subsection{Pulling back 1-forms}

In \S~\ref{sect.malgr}, we will need the following result.

\begin{prop}
\label{c_pullbackvanishing}
Let $P \in X$ be an isolated potentially klt singularity and let $\mu \colon \tilde{X} \rightarrow X$ be
a resolution of singularities of $X$ and let $E$ be an irreducible $\mu$-exceptional divisor.  
Let $\omega \in \Omega^{[1]}_X$
and let $\widetilde{\omega} := d_{\textrm{refl}}\mu(\omega)$ and let $\widetilde{\omega}_E$ be the restriction
of $\widetilde{\omega}$ to $E$.  Then $\widetilde{\omega}_E = 0$. 
\end{prop}

\begin{proof}
This is a straightforward consequence of the existence of pull-back for differential forms on potentially klt varieties, cf.~\cite[Theorem~1.2]{Kebekus13}
\end{proof}

\subsection{Singularities of $X$ vs. singularities of $\cal F$}
The following is \cite[Theorem 11.3]{CS18}.  Because we will refer to it frequently we include it here.

\begin{thm}
\label{t_canimpliesnondicritical}
Let $(\cal F, \Delta)$ be a foliated pair on a quasi-projective threefold $X$.
Assume that either
\begin{enumerate}

\item $(\cal F, \Delta)$ is F-dlt or

\item $(\cal F, \Delta)$ is canonical.

\end{enumerate}

Then $\cal F$ has non-dicritical singularities.
Furthermore, if $(\cal F, \Delta)$ is F-dlt and $K_X$ is $\bb Q$-Cartier then $X$ is klt.
\end{thm}

We also have the following comparison of singularities result, which is a slight modification of \cite[Lemma 3.16]{CS18}.

\begin{lem}
\label{lem_fdlt_implies_dlt}
Let $X$ be a $\mathbb Q$-factorial threefold and let $\cal F$ be a co-rank one foliation. 
Suppose that $(\cal F, \Delta)$ is F-dlt.
Then $(X, \Delta)$ is dlt.
\end{lem}
\begin{proof}
Let $\pi\colon X' \rightarrow X$ be a foliated log resolution of $(\cal F, \Delta)$
which only extracts divisors of foliation discrepancy $>-\epsilon(E)$.
Observe that a foliated log resolution $\pi\colon X' \rightarrow X$ of $(\cal F, \Delta)$ 
is a log resolution of $(X, \Delta)$.  
By Theorem \ref{t_canimpliesnondicritical}, $\cal F$ has non-dicritical singularities, thus, we may apply \cite[Lemma 8.14]{Spicer17} to conclude that $\pi$ only extracts divisors of discrepancy $>-1$ with respect to $K_X+\Delta$, as required.
\end{proof}

\subsection{Extending separatrices} 
We recall the following extension of separatrices result.

\begin{lem}\label{l_formalseparatrix}
Let $X$ be a normal quasi-projective threefold.  Let $\cal F$ be a co-rank one foliation on $X$
with non-dicritical singularities. Let $V \subset X$ be a subvariety tangent to $\cal F$,
let $q \in V$ be any point
and let $S_q$ be a separatrix at $q$.

Then there exists an analytic open neighborhood $U$ of $V$ in $X$ and an analytic divisor $S$ on $U$
which contains $S_q$ near $q$. 
\end{lem}
\begin{proof}
This is proven in  \cite[Lemma 3.5]{CS18} (see also \cite[\S 5.1]{Spicer17}).
We remark that this is a slight extension of the techniques and ideas utilized in \cite[\S IV]{CC92}.
\end{proof}

\subsection{Special termination}

We recall the following theorem, \cite[Theorem 7.1]{CS18}:

\begin{thm}[Special Termination]\label{t_spter}
Let $X$ be a $\bb Q$-factorial quasi-projective threefold.
Let $(\cal F, \Delta)$ be an F-dlt pair. 
Suppose $(\cal F_i, \Delta_i)$ is an infinite sequence of $(K_{\cal F_i}+\Delta_i)$-flips.
Then after finitely many flips, the flipping and flipped locus is disjoint from the lc centres
of $(\cal F_i, \Delta_i)$.
In particular, $(\cal F_i, \Delta_i)$ is log terminal in a neighborhood of each flipping curve.
\end{thm}

\subsection{MMP with scaling} \label{scaling.ssect}

A version of the MMP with scaling was proven in \cite[\S 10]{CS18}, however, for our purposes we will need the MMP with scaling in a slightly different form than presented there.  Here we briefly explain the necessary adjustments.

We recall the following lemma proven in \cite[Lemma 3.27]{CS18}

\begin{lem}\label{l_A+B}
Let $X$ be a normal projective $\bb Q$-factorial threefold and let $\mathcal F$ be a co-rank one foliation on $X$. 
Let $\Delta=A+B$ be a $\bb Q$-divisor such that $(\mathcal F,\Delta)$ is a F-dlt pair, $A\ge 0$ is an ample $\bb Q$-divisor and $B\ge 0$. Let $\varphi\colon X\dashrightarrow X'$ be a sequence of steps 
of the $(K_{\mathcal F}+\Delta)$-MMP and let $\mathcal F'$ be the induced foliation on $X'$. 

Then, there exist $\bb Q$-divisors $A'\ge 0$ and $C'\ge 0$ on $X'$ such that 
\begin{enumerate}
\item $\varphi_* A \sim_{\bb Q}  A'+C'$,
\item $A'$ is ample, and 
\item if $\Delta':=A'+C'+\varphi_*B$ then $\Delta'\sim_{\bb Q} \varphi_*\Delta$ and $(\mathcal F',\Delta')$ is F-dlt. 
  \end{enumerate}
\end{lem}

\subsubsection{Running the MMP with scaling}
\label{MMP_scaling_sect}

Let $X$ be a projective $\bb Q$-factorial threefold and let $\cal F$ be a co-rank 1 foliation on $X$.
Let $\Delta = A+B$ be a $\bb Q$-divisor where $A \geq 0$ is ample and $B \geq 0$ so that $(\cal F, \Delta)$ is a F-dlt pair.
Let $H$ be a divisor on $X$
so that $K_{\cal F}+\Delta+H$ is nef.  In practice we will often take $H$ to be some sufficiently ample divisor on $X$.

Let $\lambda = \inf \{t >0 : K_{\cal F}+\Delta+tH \text{ is nef}\}$.  By \cite[Lemma 9.2]{CS18}
there exists a $K_{\cal F}+\Delta$-negative extremal ray $R$ such that $(K_{\cal F}+\Delta+\lambda H)\cdot R = 0$.
Let $\phi\colon X \dashrightarrow X'$ be the contraction or flip associated to $R$.  

If $\phi$ is a fibre type contraction
the MMP terminates and so we may assume that $\phi$ is a divisorial contraction or a flip.
Let $\cal F'$ be the strict transform of $\cal F$, let $\Delta' = \phi_*\Delta$ and let $H' = \phi_*H$.
By Lemma \ref{l_A+B} we may find $\Theta\sim_{\bb Q}\Delta'$ so that $(\cal F', \Theta)$ is F-dlt
and $\Theta = A'+B'$ where $A' \geq 0$ is ample and $B' \geq 0$.  Thus we are free to again apply 
\cite[Lemma 9.2]{CS18}.
to $K_{\cal F'}+\Delta'+\lambda H'$ and letting
$\lambda' = \inf \{t >0 : K_{\cal F'}+\Delta'+tH' \text{ is nef}\}$ we see that
$\lambda' \leq \lambda$ and there exists 
a $K_{\cal F'}+\Delta'$-negative extremal ray $R'$ with $(K_{\cal F'}+\Delta'+\lambda'H')\cdot R' = 0$.  We are therefore
free to continue this process.

Setting $X_0:=X$,  $\cal F_0 := \cal F$, $\Delta_0 := \Delta$ and $H_0:=H$ we may produce
a sequence $\phi_i\colon X_i \dashrightarrow X_{i+1}$
of $K_{\cal F_i}+\Delta_i$ divisorial contractions and flips contracting an extremal ray $R_i$ and rational numbers $\lambda_i$
such that $K_{\cal F_i}+\Delta_i+\lambda_iH_i$ is nef and
 $(K_{\cal F_i}+\Delta_i+\lambda_iH_i)\cdot R_i = 0$, where $\cal F_i, \Delta_i, H_i$ are the strict transforms
of $\cal F_0, \Delta_0, H_0$.

Moreover we have that $\lambda_i \geq \lambda_{i+1}$, and that $R_i\cdot H_i >0$ for all $i$.  Assuming the relevant termination
of flips we see that this MMP terminates in either a Mori fibre space or a model where $K_{\cal F_i}+\Delta_i$ is nef.
We call this process the MMP of $(\cal F, \Delta)$ (or $K_{\cal F}+\Delta$) with scaling of $H$.

\subsection{(Pre) simple vs. (Log) canonical}
\label{ss_comparing_classes_of_sing}

We now briefly discuss some relations and parallels between the classes of singularities
defined by the MMP and some of the other classes of singularities described above.

Intuitively, for a given singularity, the smaller its discrepancy the more severe the singularity is.  
So terminal singularities are, in this sense, the mildest kind of singularities appearing in the MMP.
Indeed, terminal singularities on smooth surfaces are in fact smooth foliated points, although this equivalence fails in higher dimensions,
as we will see in Theorem \ref{t_3fold_term}.

We observe that simple singularities are both non-dicritical and canonical,
however canonical singularities are in general not simple, as the following example shows.

\begin{example}
\label{ex_can_not_sim}
The foliation on $\bb C^2$ defined by the vector field
$x\frac{\partial}{\partial x} + (x+y)\frac{\partial}{\partial y}$ has canonical singularities (which may be verified since
 a single blow up resolves the singularities of the foliation to simple singularities, and this blow up has discrepancy $=0$).  
 However, since both its eigenvalues are
 positive integers it is not a simple singularity.
 \end{example}
 
We remark that the above example also shows that canonical singularities are not in general F-dlt singularities.
On the other hand, Theorem~\ref{t_canimpliesnondicritical} shows that canonical and F-dlt singularities are non-dicritical.

Consider a germ of a vector field $\partial$ on $\bb C^2$ and 
suppose that $\partial$ is singular at $0$ and let $\ff m$ be the maximal ideal ideal at $0$.  
We get an induced linear map $\partial\colon \mathfrak{m}/\mathfrak{m}^2 \rightarrow \mathfrak{m}/\mathfrak{m}^2$
which is non-nilpotent if and only if the foliation generated by $\partial$ is log canonical, see \cite[Fact I.ii.4]{mp13}.
To our knowledge there is no similar criterion for characterizing log canonical foliations of rank $\geq 2$.

We refer to \cite[Definition 3]{Cano}
for the definition of presimple singularities.  The difference between simple and presimple is (roughly) the additional requirement of
a non-resonance condition on the eigenvalues of the foliation.  For instance, 
$x\frac{\partial}{\partial x}+y\frac{\partial}{\partial y}$ defines a foliation on $\bb C^2$ with
a presimple but not
simple singularity.  A single blow up of this foliation has discrepancy $=-1$ and resolves this foliation to 
a smooth foliation which shows that this foliation has a log canonical but not canonical singularity.

With this example in mind it might be useful to view the relation between simple and presimple singularities as analogous to
the relation between canonical and log canonical singularities.
However, we do not know if every presimple singularity is log canonical.  It also does not seem to be the case
that a canonical singularity is a log canonical singularity which satisfies a particular resonance condition, in light of 
Example \ref{ex_can_not_sim}.

We observe that on a smooth threefold a log canonical singularity which is non-dicritical is necessarily 
canonical.  Indeed, by \cite{Cano} and our assumption of non-dicriticality we may find a resolution
of singularities of $\cal F$ call it $\pi\colon X' \rightarrow X$, which only extracts $\cal F' := \pi^{-1}\cal F$ invariant divisors.
If $E$ is any $\pi$-exceptional divisor then $a(E, \cal F) \geq -\epsilon(E)$ by log canonicity and so $a(E, \cal F) \geq 0$
since $E$ is invariant and we may conclude that $\cal F$ has in fact canonical singularities.

We recall that the class of simple singularities is stable under blow ups contained in strata of the singular locus, however,
it is important to realize that a canonical singularity may not remain canonical after a blow up in the singular locus.
In fact, it is a subtle problem to decide when the blow up of a canonical singularity remains canonical.

We emphasize that in contrast to (pre)simple singularities the notion of (log) canonical singularities 
makes sense on singular varieties.  Take for instance the foliation on $\bb C^2$ generated by the vector field 
$\partial = x\frac{\partial}{\partial x} - y \frac{\partial}{\partial y}$.  This defines a canonical foliation singularity.  If we let $X = \bb C^2/(x, y) \sim (-x, -y)$
then $X$ is singular and note that $\partial$ descends to a vector field on $X$ which still defines a canonical singularity.

\section{Termination}
\label{term_sect}

Our goal in this section is to show the following:

\begin{thm}[Termination]
\label{t_term}
Let $X$ be a $\bb Q$-factorial quasi-projective threefold and
let $(\cal F, \Delta)$ be a F-dlt pair.
Then starting at $(\cal F, \Delta)$ there is no infinite sequence
of flips.
\end{thm}

Together with the existence of flips, \cite[Theorem 6.4]{CS18}, and divisorial contractions, \cite[Theorem 6.7]{CS18}, this has the following immediate corollary
(whose proof is identical to the proof for the corresponding statement for varieties).

\begin{cor}
\label{cor_MMP}
Let $X$ be a projective $\bb Q$-factorial threefold and let $(\cal F, \Delta)$ be an F-dlt pair.
Then there is birational contraction $f:X \dashrightarrow Y$ (which may be factored as a sequence of flips and divisorial contractions)
such that if $\cal G$ is the transformed foliation then either
\begin{enumerate}
\item $K_{\cal G}+f_*\Delta$ is nef; or

\item there is a fibration $g:Y \rightarrow Z$ such that $-(K_{\cal G}+f_*\Delta)$ is $g$-ample and the fibres
of $g$ are tangent to $\cal G$.
\end{enumerate}
We call such a contraction a $(\cal F, \Delta)$ or $K_{\cal F}+\Delta$-MMP.
\end{cor}

We will also frequently need to run the relative MMP.  The relative MMP can be deduced from the absolute MMP via
standard arguments, see for instance \cite[\S 3.6-7]{KM98}.

\begin{cor}
\label{cor_Rel_MMP}
Let $X$ be a $\bb Q$-factorial quasi-projective threefold and let $(\cal F, \Delta)$ be an F-dlt pair.
Let $p:X \rightarrow S$ be a surjective projective morphism.  Then there is a birational contraction
$f:X \dashrightarrow Y/S$ (which may be factored as a sequence of flips and divisorial contractions)
 such that if $\cal G$ is the transformed foliation  
and if $q:Y \rightarrow S$ is the structure map either
\begin{enumerate}
\item $K_{\cal G}+f_*\Delta$ is $q$-nef; or

\item there is a fibration $g:Y \rightarrow Z/S$ such that $-(K_{\cal G}+f_*\Delta)$ is $g$-ample and the fibres
of $g$ are tangent to $\cal G$.
\end{enumerate}
\end{cor}

We call the contraction $f \colon X \dashrightarrow Y/S$ constrcuted in the above corollary a $(\cal F, \Delta)$-MMP or $(K_{\cal F}+\Delta)$-MMP over $S$. 
In case (1) of the above statement, we then say that $Y$ (or, alternatively, $(\GG, f_\ast \Delta)$) is a minimal model of $X$ (or, alternatively, $(\FF, \Delta)$) over $S$;
in case (2), instead, we say that that $Y$ (or, alternatively, $(\GG, f_\ast \Delta)$) is a Mori fibre space for $X$ (or, alternatively, $(\FF, \Delta)$) over $S$.

Corollary~\ref{cor_Rel_MMP} immediately implies 
the following extension of the existence of F-dlt modification to the relative setting, cf.~\cite[Theorem 8.1]{CS18}.

\begin{thm}[Existence of F-dlt modifications]
\label{t_existencefdlt}
Let $\cal F$ be a co-rank one foliation on a 
normal quasi-projective threefold $X$.
Let $(\cal F, \Delta)$ be 
foliated pair.

Then there exists a birational morphism $\pi\colon Y \rightarrow X$ which only extracts divisors $E$ of foliation discrepancy $=-\epsilon(E)$ such that if we write $K_{\cal G}+\Gamma+F = \pi^*(K_{\cal F}+\Delta)$, where $\Gamma = \pi_*^{-1}\Delta+\sum_i\epsilon(E_i)E_i$ and where $E_i$ are the $\pi$-exceptional divisors,
then $F$ is an effective $\pi$-exceptional $\mathbb R$-divisor  and $(\cal G, \Gamma)$ is F-dlt.

Furthermore, we may choose $(Y, \cal G)$ so that 
\begin{enumerate}

\item 
if $W$ is an lc centre of $(\cal G, \Gamma)$ then $W$ is contained in a codimension one lc centre of $(\cal G, \Gamma)$,

\item 
$Y$ is $\bb Q$-factorial, and

\item 
$Y$ is klt.
\end{enumerate}
We call such a modification an F-dlt modification.
\end{thm}

\begin{proof}
The proof is analogous to that of~\cite[Theorem 8.1]{CS18}.
In particular, it suffices to consider a log resolution $\pi_1 \colon Y_1 \to X$ of $(\FF, \Delta)$ and then run the $K_{\FF_1}+\Gamma$-MMP relatively over $X$, where $\FF_1$ is the transform of $\FF$ on $Y_1$ and $\Gamma:=\pi_{1\ast}^{-1}\Delta+ \sum\varepsilon(E)E$, where the sum is taken over the prime $\pi_1$-exceptional divisors.
The relatively minimal model produced by this MMP is the desired modification $(\GG, \Gamma)$.
\end{proof}

\begin{remark}
\label{rmk:nonklt.F-dlt.mod}
We use the notation of Theorem~\ref{t_existencefdlt}. 
\begin{enumerate}
\item 
The equality $K_{\cal G}+\Gamma+F = \pi^*(K_{\cal F}+\Delta)$ implies that $\nklt(\FF, \Delta)=\pi(\nklt(\GG, \Gamma+
F))$ and $\Nlc(\FF, \Delta)=\pi(\Nlc(\GG, \Gamma+F))$.
Moreover, as $(\GG, \Gamma)$ is F-dlt, then $\Nlc(\GG, \Gamma+F))={\rm Supp}(F)$.
\item 
If the foliated log pair $(\mathcal F, \Delta)$ is lc, then the previous part of the remark implies that $F=0$ and $K_\mathcal{G}+\Gamma=\pi^\ast(K_\FF+\Delta)$.
Moreover, property $(1)$ in the statement of Theorem~\ref{t_existencefdlt} implies that $\nklt(\GG, \Gamma)$ is the union of all codimension $1$ subvarieties contained in it.
Hence, $\nklt(\GG, \Gamma)= I \cup \Supp(\lfloor \Gamma \rfloor)$, where $I$ is the union of the $\GG$-invariant divisors.
\end{enumerate}
\end{remark}

As a consequence of the existence of the MMP we 
have the following non-vanishing theorem.

\begin{thm}\label{t_nonvan}
Let $\mathcal F$ be a co-rank one foliation on a normal projective $\bb Q$-factorial threefold $X$. 
Let $\Delta$ be a $\bb Q$-divisor such that $(\mathcal F,\Delta)$ is a F-dlt pair. 
Let $A\ge 0$ and $B\ge 0$ be $\bb Q$-divisors such that $\Delta= A+B$ and $A$ is ample. Assume that $K_{\mathcal F}+\Delta$ is 
pseudo-effective 

Then $K_{\mathcal F}+\Delta \sim_{\bb Q} D \geq 0$.
\end{thm}
\begin{proof}
We run a $K_{\cal F}+\Delta$-MMP.  By \cite[Theorem 6.4 and Theorem 6.7]{CS18} 
all the required divisorial contractions and flips exist.  By Theorem \ref{t_term} there is no infinite sequence of flips
and so this MMP terminates, call it $\phi\colon X \dashrightarrow X'$.
Let $\cal F'$ be the transform of $\cal F$.
By Lemma \ref{l_A+B} 
we may find an ample divisor $A'$ and $B' \geq 0$ such
that $\phi_*\Delta \sim_{\bb Q} A'+B'$ and $(\cal F', A'+B')$ is F-dlt. Set $\Delta' = A'+B'$

Thus, we may apply \cite[Theorem 9.4]{CS18} to conclude that $K_{\cal F'}+\Delta'$ is semi-ample and
so there exists $0 \leq D \sim_{\bb Q} K_{\cal F'}+\Delta'$.
For all $m$ sufficiently divisible we have
\begin{align*}
H^0(X, \cal O(m(K_{\cal F}+\Delta))) =& H^0(X', \cal O(m(K_{\cal F'}+\phi_*\Delta))) \\
=& H^0(X', \cal O(m(K_{\cal F'}+\Delta')))
\end{align*}
and our result follows.
\end{proof}

\subsection{Singular Bott partial connections}
We recall Bott's partial connection.
Let $\cal F$ be a smooth foliation on a complex manifold $X$.
We can define a partial connection on $N_{\cal F}$ locally by
\[\nabla: N_{\cal F} \rightarrow \Omega^1_{\cal F} \otimes N_{\cal F}\]
\[w \mapsto \sum \omega_i \otimes q([\partial_i, \widetilde{w}])\]
where $\widetilde{w}$ is any local lift of $w$ to $T_X$ and $\omega_i$ are local generators
of $\Omega^1_{\cal F}$, $\partial_i$ are dual generators of $\cal F$  and $q:T_X \rightarrow N_{\cal F}$ is the quotient map.  
One can check that these local connections patch to give a global connection.

\begin{lem}
\label{l_connection}
Let $\cal F$ be a rank $r$ foliation on a complex analytic variety $X$.
Let $S \subset X$ be a local complete intersection subvariety of $X$
of dimension $r$ and suppose that $S$ is $\cal F$-invariant.

Let $Z = \text{sing}(X) \cup \text{sing}(\cal F)$.  Suppose that $Z\cap S$ is codimension at least $2$ in $S$.
Then there is a connection 
\[\nabla:N_{S/X} \rightarrow \Omega^{[1]}_S \otimes N_{S/X}\]
where $\Omega^{[1]}_S = (\Omega^1_S)^{**}$ is the sheaf of reflexive differential on $S$.
\end{lem}
\begin{proof}
Let $X^\circ = X - Z$ and $S^\circ = S-(Z\cap S)$.
Notice that $N_{\cal F}\vert_{S^\circ} = N_{S^\circ/X^\circ}$ and that $\Omega^1_{\cal F}\vert_{S^\circ} = \Omega^1_{S^\circ}$.
Thus, if we restrict Bott's partial connection on $X^\circ$ to $S^\circ$ we get a connection
\[\nabla^\circ: N_{S^\circ/X^\circ} \rightarrow \Omega^1_{S^\circ} \otimes N_{S^\circ/X^\circ}.\]

Let $i:S^\circ \rightarrow S$ be the inclusion.  Since $N_{S/X}$ is locally free we have an isomorphism
$i_*(\Omega^1_{S^\circ} \otimes N_{S^\circ/X^\circ}) = \Omega_S^{[1]} \otimes N_{S/X}$
by the push-pull formula, \cite[Exercise II.5.1(d)]{Hartshorne77}.
Thus we get a map
\[i_*\nabla^\circ: N_{S/X} \rightarrow \Omega_S^{[1]} \otimes N_{S/X}\]
and by observing that $i_*\nabla^\circ$ satisfies the Leibniz condition (since it does so away
from a set of codimension at least 2) we see that this is our desired connection.
\end{proof}

\subsection{Proof of Theorem \ref{t_term}}

\begin{lem}
\label{l_zerointersection}
Let $X$ be a normal complex analytic threefold and let $(\cal F, \Delta)$
be a log terminal co-rank 1 foliation on $X$.
Let $C \subset X$ be a compact curve tangent to $\cal F$.  Let $S$ be a germ
of an invariant surface containing $C$.  Suppose that $K_X$, $K_{\cal F}$ and $S$ are $\bb Q$-Cartier.  Then
\[S\cdot C = 0.\]
\end{lem}

\begin{proof}
Let $H \subset X$ be a sufficiently ample divisor meeting $C$ transversely
and choose $H$ to be sufficiently general so that
$(\cal F, \Delta+(1-\epsilon)H)$ is log terminal for all $\epsilon>0$.

We may then find a Galois cover $\pi:X' \rightarrow X$ ramified over
$H$ and $\text{sing}(X)$  such that if we write $S' = \pi^{-1}(S)$
and $\cal F' = \pi^{-1}\cal F$ 
then $S'$ and $K_{\cal F'}$ are both Cartier.

Write $\Delta' = \pi^*\Delta$ and $C' = \pi^{-1}(C)$.
We claim that $(\cal F', \Delta')$ is log terminal.
Indeed, let $r$ be the ramification index along $H$.  We have by foliated
Riemann-Hurwitz
that $K_{\cal F'}+\Delta' = \pi^*(K_{\cal F}+\Delta+\frac{r-1}{r}H)$.
Since $(\cal F, \Delta+\frac{r-1}{r}H)$ is log terminal
we see that the same is true of $(\cal F', \Delta')$, see \cite[Proposition 5.20]{KM98} or the proof
of \cite[Corollary 3.9]{Spicer17}.

Since $K_{\cal F'}$ is Cartier and $(\cal F', \Delta')$ is log terminal 
if we let $Z = \text{sing}(X') \cup \text{sing}(\cal F')$ then we have that $Z\cap S'$ is codimension at least 2 in $S'$,
see \cite[Corollary I.2.2]{McQuillan08}.
By construction $S'$ is Cartier and so we may apply Lemma \ref{l_connection} to produce a connection
\[\nabla:N_{S'/X'} \rightarrow \Omega^{[1]}_{S'}\otimes N_{S'/X'}.\]

Let $n:B \rightarrow C'$ be the normalization of an irreducible component of $C'$.
By \cite[Lemma 3.16]{CS18} we see that $(X', \Delta'+S')$ is plt and so by (usual) adjunction, 
\cite[Theorem 5.50]{KM98}, $S'$ is klt. 
By \cite[Theorem 4.3]{GKKP11} there exists a non-zero
morphism $d_{refl}n:n^*\Omega^{[1]}_{S'} \rightarrow \Omega^1_B$,
and so we may pull back $\nabla$ to get a connection by composing 
\[n^*\cal O(S') \xrightarrow{n^*\nabla} n^*\Omega^{[1]}_{S'} \otimes n^*\cal O(S') 
\xrightarrow{d_{refl}n\otimes \text{id}} \Omega^1_B \otimes n^*\cal O(S').\]

In particular, since $n^*\cal O(S')$ admits a holomorphic connection, 
it is flat which implies $0 = S'\cdot n(B) = m(S\cdot C)$.
\end{proof}

\begin{proof}[Proof of Theorem \ref{t_term}]
By Special Termination, \cite[Theorem 7.1]{CS18}, it suffices to show that any sequence of log terminal flips terminates.
Let \[\phi\colon (X_i, \cal F_i, \Delta_i) \dashrightarrow (X_{i+1}, \cal F_{i+1}, \Delta_{i+1})\]
be one such flip and let $C \subset X_i$ be an
irreducible component of $\text{exc}(\phi)$.
Let $f\colon X_i \rightarrow Z$ denote the base of the flip.

Since $C$ is tangent to the foliation
any divisor $E$ dominating $C$ on some model of $X_i$ is invariant, i.e., $\epsilon(E) =0$.
Since $(\cal F_i, \Delta_i)$ is log terminal we have $a(E, \cal F_i, \Delta_i) >\epsilon(E) =0$,
i.e., $\cal F_i$ is terminal at the generic point of $C$.

By \cite[Lemma 3.14]{CS18}
by taking $U$ to be a sufficiently small analytic neighborhood of
$z = f(C)$ we may find a unique $\cal F_i$-invariant divisor
on $X_{i, U} := f^{-1}(U)$ containing $C$.  Call this divisor $S$.

Since $X_{i, U}$ is klt and projective over $U$ we may find a small $\bb Q$-factorialization of $X_{i, U}$ denoted
$g\colon \overline{X_{i, U}} \rightarrow X_{i, U}$.  Let $\overline{\cal F_i}$ be the transformed foliation, write
$K_{\overline{\cal F_i}}+\overline{\Delta_i} = g^*(K_{\cal F_i}+\Delta_i)$,
let $\overline{S}$ be the strict transform of $S$ and let $\overline{C}$ be the strict transform
of $C$.
Since $g$ is small, we see that $(\overline{\cal F_i}, \overline{\Delta_i})$ is still log terminal.

Let $P \in \overline{C}$ be a point and let $T$ be a germ of a $\overline{\cal F_i}$-invariant divisor at $P$.
We claim that $T = \overline{S}$ (as germs).  Indeed, suppose otherwise.  Since $\overline{S}$ is $\bb Q$-Cartier we know
that $T \cap \overline{S}$ contains a $1$-dimensional component $\Sigma$.
Since $\Sigma$ is the intersection of two invariant divisors we see that $\Sigma \subset \text{sing}(\cal F)$ 
and $\Sigma$ is tangent to the foliation, in particular, we see that 
$\overline{\cal F_i}$ is terminal at the generic point of $\Sigma$.
This, however, is a contradiction of Proposition \ref{prop_term_surf_class} which implies that terminal
foliation singularities are non-singular in codimension $2$.

By Lemma \ref{l_zerointersection} 
we see that \[\overline{S} \cdot \overline{C} = 0.\]
On the other hand by \cite[Corollary 3.20]{CS18} and the observation in the previous paragraph
we see that the collection of $\overline{\cal F_i}$-invariant divisors meeting $\overline{C}$ is exactly $\overline{S}$ itself
and so
\[(K_{\overline{\cal F_i}}+\overline{\Delta_i})\cdot \overline{C} = (K_{\overline{X_i}}+\overline{\Delta_i}+\overline{S})\cdot \overline{C}.\]
Since $K_{\overline{X_i}}+\overline{\Delta_i} =g^*(K_{X_i}+\Delta_i)$ putting these equalities together yields
\[0 > (K_{\overline{X_i}}+\overline{\Delta_i})\cdot \overline{C} = (K_{X_i}+\Delta_i)\cdot C.\]
Thus, each $K_{\cal F_i}+\Delta_i$-flip is in fact a $K_{X_i}+\Delta_i$-flip.

By \cite[Lemma 3.16]{CS18} $(X_i, \Delta_i)$ is log terminal and so our result follows by termination for threefold log
terminal flips, see for example \cite[Theorem 6.17]{KM98}.
\end{proof}

To finish we present an example of a foliation flip, another example may be found in \cite[Example 9.1]{Spicer17}.
\begin{example}
\label{example_flip}
Let $b\colon Y \rightarrow \bb C^2$ be the blow up at the origin with exceptional curve $C$ and let $p\colon \widetilde{X} \rightarrow Y$ be the total space of the line bundle $\cal O_Y(C)$.  
Observe that $\widetilde{X}$ contains a single projective curve, which we will continue to denote by $C$.
Let $\cal G$ be the foliation on $Y$ given by the transform of the foliation generated by $\frac{\partial}{\partial x_1}$ on $\bb C^2$ (with coordinates $(x_1, x_2)$)  and let $\widetilde{\cal G} = p^{-1}\cal F$.  
Set $S = p^{-1}(C)$ and let $D_i = p^{-1}(b_*^{-1}\{x_i = 0\})$.
\\
It is straightforward to check that $K_{\widetilde{\cal G}}\cdot C = 0$ and that $\widetilde{\cal G}$ is smooth at the generic
point of $C$.  Moreover, observe that $S$ and $D_2$ are $\widetilde{\cal G}$-invariant whereas $D_1$ is not.
\\
Consider the map $\sigma\colon \bb C^2 \rightarrow \bb C^2$ given by $(x_1, x_2) \mapsto (-x_1, x_2)$ and observe
that $\sigma$ lifts to a map $\tau\colon \widetilde{X} \rightarrow \widetilde{X}$.  
Let $X:=\widetilde{X}/\langle \tau \rangle$ and observe that $\pi\colon \widetilde{X} \rightarrow X$
is ramified to order 2 along $S$ and $D_1$.  Observe moreover that $\tau$ preserves $\widetilde{\cal G}$ and so it descends
to a foliation $\cal F$ on $X$.  A foliated Riemann-Hurwitz computation shows that
$K_{\widetilde{\cal G}} = \pi^*K_{\cal F}+D_1$.  In particular, if we let $\Sigma = \pi(C)$ we see that
$K_{\cal F}\cdot \Sigma <0$ and so $\Sigma$ is a $K_{\cal F}$-flipping curve.

Notice that $\Sigma$ meets $\text{sing}(\cal F)$ at a single point which is a $\bb Z/2$ quotient singularity.
\end{example}

\section{Connectedness}
\label{conn.sect}

\subsection{Connectedness of the nonklt locus for foliated pairs}
The aim of this section is to prove the following connectedness statement which constitutes one of the pillars in the analysis of the birational structure of foliated singularities.
The analogue result in the non-foliated case has a long history and is rather classical; recently,~\cite{Bir20, FS20} fully settled the Connectedness Principle in full generality for pairs.

\begin{thm}
\label{conn.f-dlt.thm}
Let $f \colon X \to Y$ be a contraction of normal quasi-projective varieties, with $X$ a $\mathbb{Q}$-factorial threefold. 
Let $\FF$ be a co-rank $1$ foliation on $X$.
Let $(\FF, \Delta)$ be a foliated log pair with $\Delta= \sum a_i D_i$.
Assume that $-(K_{\FF}+\Delta)$ is $f$-nef and $f$-big and that $(\FF, \Delta')$ is F-dlt, where $\Delta':= \sum_{a_i < \epsilon(D_i)} a_i D_i + \sum_{a_j \geq \epsilon(D_j)} \epsilon(D_j)D_j$.
Then $\nklt(\FF, \Delta)$ is connected in a neighborhood of every fiber of $f$.
\end{thm}

Theorem \ref{conn.f-dlt.thm} immediately implies the following more general result.

\begin{thm}
\label{conn.f-lc.thm}
Let $f \colon X \to Y$ be a contraction of normal quasi-projective varieties.
Let $\FF$ be a co-rank $1$ foliation on $X$.
Let $(\FF, \Delta)$ be a foliated log pair.
Assume that $-(K_{\FF}+\Delta)$ is $f$-nef and $f$-big.
Then $\nklt(\FF, \Delta)$ is connected in a neighborhood of every fiber of $f$.
\end{thm}

\begin{proof}
It suffices to consider $g \colon X' \to X$ a F-dlt modification, $K_\GG+\Delta_{X'}= g^\ast(K_\FF +\Delta)$ and apply Theorem \ref{conn.f-dlt.thm} to the pair $(\GG, \Delta_{X'})$ and the map $f \circ g \colon X' \to Y$.
\end{proof}

We will prove Theorem \ref{conn.f-dlt.thm} in the course of this section by proving different cases that fit together to provide a argument for it.

Before proving the theorem we indicate a quick application of
Theorem \ref{conn.f-dlt.thm} to the geometry of (weak) Fano foliations, see also \cite{AD13}.

We will denote by $\text{sing}^*(\cal F)$ the union of all codimension 2 components of $\text{sing}(\cal F)$.

\begin{cor}
\label{cor_sample_app_conn}
Let $X$ be a smooth projective threefold and let $\cal F$ be a co-rank $1$ foliation on $X$.
Assume that $-K_{\cal F}$ is big and nef. 
Then either
\begin{enumerate}
\item $\cal F$ has an algebraic leaf; or

\item $\text{sing}^*(\cal F)$ is connected.
\end{enumerate}
\end{cor}
\begin{proof}
Let us observe that $\nklt(\cal F) = \text{sing}^*(\cal F)\cup I\cup Z$ where $I$ is the union of all the $\cal F$-invariant divisors and $Z$ is a finite collection of points.
\\
We take an F-dlt modification $\mu\colon \overline{X}\rightarrow X$ of $\cal F$ and let $\overline{\mathcal F}$ be the induced foliation on $\overline X$, which exists by Theorem~\ref{t_existencefdlt}.
Writing $K_{\overline{\mathcal F}} +\Delta = \mu^*K_{\cal F}$, 
then $\mu(\nklt(\overline{\cal F}, \Delta))=\nklt(\cal F)$, see Remark~\ref{rmk:nonklt.F-dlt.mod}.
If $\FF$ has no algebraic leaves then $I=\emptyset$.
We conclude applying Theorem \ref{conn.f-dlt.thm} that $\nklt(\overline{\cal F}, \Delta)$ is connected and our result follows.
\end{proof}

\begin{remark}
The corresponding statement to Theorem \ref{conn.f-dlt.thm} for rank 1 foliations is an essentially trivial consequence
of the arguments in \cite{BMc01}.
\end{remark}

We now turn to the proof of Theorem \ref{conn.f-dlt.thm}.

We will work in the following setting.
We denote by $f \colon X \to Y$ a contraction of normal quasi projective varieties, with $X$ a $\mathbb{Q}$-factorial threefold. 
Recall that $f$ being a contraction means that $f$ is surjective, projective with $f_*\cal O_X = \cal O_Y$.

We assume the existence of a co-rank $1$ foliation $\FF$ on $X$ and of a foliated log pair $(\FF, \Delta)$ 
with $\Delta= \sum a_i D_i$.
We will denote by 
\begin{eqnarray*}
\nonumber
H:= & -(K_\FF+\Delta), & \\
\nonumber 
\Delta':= & \sum_{a_i < \epsilon(D_i)} a_i D_i + \sum_{a_j \geq \epsilon(D_j)} \epsilon(D_j)D_j,\\ 
\nonumber
\Delta'' := & \Delta - \Delta', \\
\nonumber
F := & {\rm Supp}(\Delta^{\geq 1}).
\end{eqnarray*}

We remark that we allow $\Delta$ to have $\cal F$-invariant components, however
$\Delta'$ will have no $\cal F$-invariant components.

We start by addressing the birational case.
\begin{lem}
\label{connct.lemma.bir.case.rho=1}
With the notation above, we assume that $f$ is birational, $\rho(X/Y)=1$ and $(\FF, \Delta')$ is F-dlt.
Suppose moreover that every lc centre of $(\FF, \Delta)$ is contained in a codimension 1 lc centre of $(\FF, \Delta)$.
If $-(K_X+\Delta)$ is $f$-ample then $\Nklt(\FF, \Delta)$ is connected in a neighborhood of any fibre of $f$.
\end{lem}

Let us recall that, as observed in Remark~\ref{rmk:nonklt.F-dlt.mod}, $\nklt(\FF, \Delta)= \nklt(\FF, \Delta') = {\rm Supp}(F+ I)$, where $I$ is the sum of the $\FF$-invariant divisors.

\begin{proof}
We assume that $\nklt(\FF, \Delta)$ is disconnected in a neighborhood of some fiber of $f$ and we will show that such assumption leads to a contradiction.
Let us observe that we may assume that $\Delta''\geq 0$ is $f$-ample,
otherwise $\text{exc}(f) \subset\Supp(\Delta'')$ and there is nothing to prove. 
Thus $-(K_{\cal F}+\Delta') = -(K_{\cal F}+\Delta)+\Delta''$ is $f$-ample.
By \cite[Lemma 8.10]{Spicer17} we see that $f$ only contracts curves tangent to $\cal F$.

{\bf Case 1}. \emph{The morphism $f$ is a divisorial contraction}. 
\newline
Suppose that $f$ contracts a divisor $E$.  Observe that since $\rho(X/Y) = 1$ we have that $E$ is irreducible.  
If $E$ is invariant then it is an lc centre and so there is nothing to prove.  Thus we may assume that $E$
is not invariant.
If $f(E)$ is a point and if $B$ is a component of $\nklt(\FF, \Delta)\cap E$
then observe that $B$ is ample in $E$, in particular we see that $\nklt(\FF, \Delta)\cap E$ is connected.

Thus we may assume that $f(E)=C$ is a curve in $Y$.
We may find $t \geq 0$ so that we may write $\Delta'+tE = \Gamma+E$ 
where $\Supp(\Gamma)$ does not contain $E$.  Since $-E$ is $f$-ample
we see that $-(K_{\cal F}+\Gamma+E)$ is $f$-ample.
By the foliated adjunction formula, cf.~\cite[Lemma 3.18]{CS18} and~\cite[Lemma 8.9]{Spicer17}, or cf.~Lemma~\ref{adjunction} below,  we may write $(K_{\cal F}+\Gamma+E)\vert_E = K_{\cal G}+\Gamma_E$ where
$\cal G$ is the restricted foliation,
$\Gamma_E \geq 0$ and $\nklt(\cal F, \Gamma)\cap E \subset \nklt(\cal G, \Gamma_E)$,
and so by assumption $\nklt(\cal G, \Gamma_E)$ contains at least two components meeting 
a fibre of $f$

Let $\Sigma_0$ be an irreducible curve contracted by $f$.
Since $\Sigma_0$ is tangent to $\cal G$ it is therefore a rational curve with $K_{\cal G}\cdot \Sigma_0 \geq -2$.
Moreover, exactly one of the following two scenarios hold:
\begin{enumerate}
\item 
either $\Sigma_0$ meets two distinct components
of $\nklt(\cal G, \Gamma_E)$; or,
\item 
the fibre containing $\Sigma_0$ is a union of two rational curves meeting at a point and, up to switching the two components of this fibre, we can assume that $\Sigma_0$ meets at least one connected component of $\nklt(\cal G, \Gamma_E)$.
\end{enumerate} 
Hence, $0 > (K_{\cal G}+\Gamma_E)\cdot \Sigma_0 \geq \deg (K_{\Sigma_0}+p_1+p_2)$
where, in scenario (1), $p_1, p_2$ are the intersections of $\Sigma_0$ with two distinct connected components of $\nklt(\FF, \Delta')$ along $\Sigma_0$, while in scenario (2), $p_1$, is the intersection of $\Sigma_0$ with the other component of the fibre and $p_2$ is the intersection of $\Sigma_0$ with $\nklt(\FF, \Delta')$.  
However, $\deg (K_{\Sigma_0}+p_1+p_2) \geq 0$, which provides a contradiction since $-(K_{\cal G}+\Gamma_E)$ is $f\vert_E$-ample.

{\bf Case 2}.
\emph{ The morphism $f$ is a flipping contraction}.
\\
We denote by
\[
\xymatrix{
X \ar@{-->}[rr] \ar[dr]^f & & X^+ \ar[dl]_{f^+}
\\
& Y &
}
\]
the flip of $f$ and by $\Sigma$ a curve in the exceptional locus.  
Then there exists two divisorial components $D_1, D_2$ of $\nklt(\FF, \Delta')$ which intersect $\Sigma$, and do not contain it.
But then on $X^+$, the strict transforms $D^+_i$ of the $D_i$ contain the exceptional locus of the map $f^+$, hence this must be contained in the intersection of the $D_i^+$ and as such it is a non-klt center.

Since $\Sigma$ is tangent to $\cal F$ we may assume that
$(\FF, \Delta')$ is terminal along $\Sigma$, as otherwise $\Sigma$ would be an lc centre, see \cite[Lemma 3.14]{CS18}); 
on the other hand, the above observation implies that $\FF^+$, the birational transform of $\FF$ on $X^+$, is canonical along the exceptional locus of $f^+$. 
But this leads to a contradiction, as by the Negativity Lemma the discrepancies of $(\FF, \Delta')$ along the $f^+$-exceptional locus must decrease since $-(K_\FF+\Delta')$ is $f$-ample, see, for example,~\cite[Lemma~2.7]{CS18}.
\end{proof}

\begin{lem}
\label{bir.conn.ample.lem}
With the notation above, we assume that $f$ is birational and $(\FF, \Delta')$ is F-dlt.
Suppose moreover that every lc centre of $(\FF, \Delta')$ is contained in a codimension 1 lc centre
of $(\FF, \Delta')$.
If $-(K_X+\Delta)$ is $f$-ample then $\Nklt(\FF, \Delta)$ is connected in a neighborhood of any fibre of $f$.
\end{lem}

\begin{proof}
Let $y \in Y$ be a point on $Y$ and let $X_y$ denote the fiber of $f$ over $Y$.
We assume that $\nklt(\FF, \Delta)$ is disconnected in a neighborhood of $X_y$ and we will show that such assumption leads to a contradiction.
As each lc center of $(\FF, \Delta')$ is contained in a codimension $1$ lc center and $\nklt(\FF, \Delta)$ is disconnected in a neighborhood of $X_y$, there exist prime divisors $E_1, E_2$ on $E$ such that:
\begin{itemize}
\item 
$E_1, E_2$ intersect $X_y$; and,
\item
$E_1, E_2$ belong to different connected components of $\nklt(\FF, \Delta)$ in a neighborhood of $X_y$;
in particular, $E_1 \cap E_2= \emptyset$ in a neighborhood of $X_y$.
\end{itemize}
As $H$ is $f$-ample, there exists $0<\epsilon\ll1$ such that $G:=H-\epsilon (E_1+E_2)$ is $f$-ample.
Then
\begin{equation}
\label{eq.gen.lin.equiv}
K_\FF + \Delta' + G\sim_{\mathbb{R}, f} -\epsilon (E_1+E_2) - \Delta''.
\end{equation}
We can then run the $(K_\FF+\Delta'+G)$-MMP with scaling of $G$ over $Y$, see Section \ref{scaling.ssect},
\begin{equation}
\label{mmp.long.eq.1}
		\xymatrix{
			X=X_0 \ar@{-->}[r]^{s_1} \ar[drrr]_{f} & 
			X_{1}  \ar@{-->}[r]^{s_2} \ar[drr]^{f_1}& 
			\dots  \ar@{-->}[r]^{s_{i}} & 
			X_{i}  \ar@{-->}[r]^{s_{i+1}} \ar[d]^{f_i}& 
			X_{i+1} \ar@{-->}[r]^{s_{i+2}} \ar[dl]^{f_{i+1}}& 
			\dots \; \\
			& & & Y & &.
			}
\end{equation}
We quickly explain how to run such an MMP.
As we are only interested in what happens over a neighborhood of $y \in Y$, we can assume that each step of~\eqref{mmp.long.eq.1} is non-trivial in a neighborhood of $X_{i, y}$.

As $G$ is ample, then there exists $0< \eta \ll 1$ such that $G':= G +\eta \lfloor \Delta' \rfloor$ is also ample. 
Hence, choosing a suitable effective $P \sim_{\mathbb{R}, f} G'$, by Bertini's theorem we see that 
$K_\FF + (\Delta' - \eta \lfloor \Delta' \rfloor) + P$ is F-dlt. 
Hence, the MMP exists for $K_\FF + (\Delta' - \eta \lfloor \Delta' \rfloor) + P$ and a fortiori for $K_\FF+\Delta'+G$ as well.
Since each step of this MMP is $G_i$ positive, where $G_i$ is the strict transform of $G$ on $X_i$,
then each step of this MMP is in fact a step of the $(K_{\cal F}+\Delta')$-MMP.
In particular, we may observe moreover that at each step any lc centre of $(\FF_i, \Delta'_i)$ is contained in a codimension $1$ lc centre of $(\FF_i, \Delta'_i)$.
Indeed, by~\cite[Lemma~2.7]{CS18} an lc centre cannot lie in $\text{exc}(s_i^{-1})$ and so if $W$ is an lc centre of $(\FF_i, \Delta'_i)$ then each $s_j$ for $j\leq i$ must be an isomorphism at the generic point of $W$.

Lemma~\ref{connct.lemma.bir.case.rho=1} shows that the number of connected components of $\nklt(\FF_i, \Delta_i)$ in a neighborhood of $X_{i, y}$ cannot decrease with $i$.
Assume at the $i$-th step $X_{i-1} \dashrightarrow X_i$ of~\eqref{mmp.long.eq.1} the strict transform of one of the $E_j$, say $E_1$, gets contracted. 
Denoting by $R_{i-1}$ the generator of the extremal ray of $\overline{NE}(X_{i-1}/Y)$ contracted at this step and by $E_{1, i-1}$ the strict transform of $E_1$ on $X_{i-1}$, then $R_{i-1}\cdot E_{1, i-1} <0$.
Lemma~\ref{connct.lemma.bir.case.rho=1} implies that $E_{1, i-1} \cap E_{2, i-1} = \emptyset$ in a neighborhood of $X_{i-1, y}$, as otherwise the number of connected components would have decreased at some point of the MMP.
This observation implies that $E_{2, i-1} \cdot R_{i-1}=0$. 
As $-\epsilon(E_{1, i-1} + E_{2, i-1}) \cdot R_{i-1} < \Delta''_{i-1} \cdot R_{i-1}$, then $\Delta''_{i-1} \cdot R_{i-1} >0$; 
thus, there exists a component $D_1$ of $\Delta''_{i-1}$ that intersects $E_{1, i-1}$  and such that $f(D_1) \supset f(E_{1, i-1})$.
But then on $X_{i}$,
\[
K_{\FF_{i}} + \Delta'_{i} + G'_{i} \sim_{\mathbb{R}, f_{i}} -\epsilon E_{2, i} - D_{1, i} -\Delta''_{i}
\]
and we can repeat the argument just illustrated as $E_{2, i} \cap D_{1, i} = \emptyset$ around $X_{i, y}$ and they belong to different connected components of $\nklt(\FF_i, \Delta_i)$ in a neighborhood of $X_{i, y}$.
By Corollary~\ref{cor_Rel_MMP}, the MMP in~\eqref{mmp.long.eq.1} terminates with a minimal model $f' \colon X' \to Y$ of $X$ over $Y$, since $f$ is birational and hence $K_\FF+\Delta'+G$ is relatively pseudoeffective over $Y$.
Hence the strict transform of $-\epsilon(E_1+E_2)-\Delta''$ is $f'$-nef. 
By the Negativity Lemma,~\cite[Lemma~3.39]{KM98} it must contain the whole fiber $X'_y$, which leads to a contradiction.
\end{proof}

\begin{prop}
\label{bir.conn.prop}
With the above notation, we assume that $f$ is birational and $(\FF, \Delta')$ is F-dlt.
If $-(K_\FF+\Delta)$ is $f$-big and $f$-nef then $\nklt(\FF, \Delta)$ is connected in a neighborhood of any fibre of $f$.
\end{prop}

\begin{proof}
First, observe that we may freely replace $(\FF, \Delta')$ by a higher model so
that every lc centre of $(\FF, \Delta')$ is contained in a codimension $1$ lc centre.
Indeed, by Theorem \ref{t_existencefdlt} we may find a modification
$\mu\colon \overline{X} \rightarrow X$ so that if we write 
$K_{\overline{\cal F}}+\Theta = \mu^*(K_{\cal F}+\Delta')$ where $\overline{\cal F} = \mu^{-1}\cal F$
then $(\overline{\cal F}, \Theta)$ is F-dlt, every lc centre is contained in a codimension 1 lc centre
and $\nklt(\overline{\cal F}, \Theta)\subset \mu^{-1}(\nklt(\FF, \Delta'))$.
Thus $\nklt(\overline{\cal F}, \Theta)$
is connected in a neighborhood of a fibre over $Y$
if and only if $\nklt(\FF, \Delta')$ is.

We next reduce the general case to the case of ample $H$, which then follows from Lemma \ref{bir.conn.ample.lem}.
\newline
As $H=-(K_\FF +\Delta)$ is $f$-big and $f$-nef there exists an effective $\mathbb{R}$-divisor $B = \sum a_iB_i$ for which $H-\delta B$ is $f$-ample for any $0<\delta \ll1$. 
We can decompose $B$ as 
\[
B = B_1+B_2, \quad \text{where} \;  B_1:= \sum_{B_i \subset \nklt(\FF, \Delta)} a_iB_i, \; B_2 :=B-B_1.
\]
We claim that for $\delta$ sufficiently small then 
\begin{equation}
\label{nklt.eq}
\nklt(\FF, \Delta') =\nklt(\FF, \Delta'+\delta B_2)
\end{equation}
Indeed, let $r_Z \colon Z \to X$ be a foliated log resolution of $(\FF, \Delta'+\delta B_2)$.
We denote by $\GG$ the strict transform of $\FF$ on $Z$.
Thus,
\[
K_\GG + r_{Z\ast}^{-1}(\Delta' + \delta B_2) = r_Z^\ast(K_\FF + \Delta' + \delta B_2)+ \sum_i b_i(\delta) E_i
\]
and $\nklt(\cal F, \Delta'+\delta B_2) = r_Z(\nklt(\cal G, r_{Z\ast}^{-1}(\Delta' + \delta B_2)-\sum_i b_i(\delta)E_i))$.
Each $b_i(\delta)$ depends linearly on $\delta$ and $b_i(0) \geq -\epsilon(E_i)$, since $(\FF, \Delta')$ is F-dlt.
Hence, if $b_i(\delta) \leq -\epsilon(E_i)$ for all $\delta >0$, then $E_i$ is an lc center for $(\FF, \Delta')$.
\newline
For fixed sufficiently small $\delta>0$ satisfying \eqref{nklt.eq}, let $\pi \colon Y \to X$ be a F-dlt modification of $(\FF, \Delta'+\delta B_2)$.
By Theorem~\ref{t_existencefdlt}, writing
\[
K_\GG + \Gamma + C=\pi^\ast(K_\FF+\Delta'+\delta B_2),
\]
then
\begin{itemize}
\item $\Gamma= \pi_\ast^{-1} \Delta' + E$ and $E = \sum_i \epsilon(E_i)E_i$ where we sum over the expceptional divisors of $\pi$, 
\item $(\GG, \Gamma)$ is F-dlt, 
\item $\nklt(\GG, \Gamma)= \rounddown{\Gamma}$,
\item $C \geq 0$, and 
\item the support of $C$ is contained in $\nklt(\GG, \Gamma)$ and is $\pi$-exceptional.
\end{itemize} 
Moreover, the $\mathbb{Q}$-factoriality of $X$ implies that there exists an effective $\pi$-exceptional divisor $G\geq 0$ such that $-G$ is $\pi$-ample.
Since $K_\FF+\Delta+\delta B_2 \sim_{\bb R, X} -(H- \delta B + \delta B_1)$ then
\[
K_\GG +\Gamma + C + \delta \pi^\ast B_1 + \pi^*(\Delta-\Delta')+  \epsilon G \sim_{\bb R, X} -\pi^\ast (H-\delta B) + \epsilon G.
\]
Let $\Theta:=C + \delta \pi^\ast B_1 + \pi^*(\Delta-\Delta')+  \epsilon G$ and
observe that we have
\[\pi^{-1}(\nklt(\FF, \Delta)) =  \pi^{-1}(\nklt(\FF, \Delta+\delta B)) = 
\nklt(\GG, \Gamma+\Theta)\]
For $\epsilon >0$ sufficiently small we know that $(\Gamma+\Theta)' = \Gamma$ (in the notation at the beginning of the section)
and 
$-(K_{\GG}+\Gamma+\Theta) = \pi^\ast (H-\delta B) - \epsilon G$ is $f$-ample and this concludes the proof.
\end{proof}

\begin{proof}[Proof of Theorem \ref{conn.f-dlt.thm}]
In view of Lemma \ref{connct.lemma.bir.case.rho=1} and \ref{bir.conn.ample.lem}, we are only left to prove the case where $f$ is a non-birational contraction.
Hence, we assume that $\nklt(\FF, \Delta)$ is disconnected in a neighborhood of some fiber $X_y, \; y \in Y$ of $f$ with $\dim X > \dim Y$ and we derive a contradiction.

{\bf Step 1.} \emph{ In this step, we assume $H$ to be $f$-ample}.
\newline
As $H$ is $f$-ample, there exists $0<\epsilon\ll1$ such that $G=H-\epsilon F$ is $f$-ample.
We can then run the $(K_\FF+\Delta')$-MMP with scaling of $G$ over $Y$, see Section \ref{scaling.ssect},
\begin{equation}\label{mmp.long.eq}
		\xymatrix{
			X=X_0 \ar@{-->}[r]^{s_1} \ar[drrr]_{f} & 
			X_{1}  \ar@{-->}[r]^{s_2} \ar[drr]^{f_1}& 
			\dots  \ar@{-->}[r]^{s_{i}} & 
			X_{i}  \ar@{-->}[r]^{s_{i+1}} \ar[d]^{f_i}& 
			X_{i+1} \ar@{-->}[r]^{s_{i+2}} \ar[dl]^{f_{i+1}}& 
			\dots \; \\
			& & & Y & &.
			}
\end{equation}
We denote $\FF_i := s_{i \ast} \FF_{i-1}, \; \Delta_i := s_{i \ast} \Delta_{i-1}, \; \Delta_i' := s_{i \ast} \Delta_{i-1}', \; \Delta_i'' := s_{i \ast} \Delta_{i-1}'',$ $F_i := s_{i\ast}F_{i-1}, \; I_i := s_{i\ast}I_{i-1}$, $G_i:= s_{i\ast} G_{i-1}$, and $f_i$ is the structural map for $X_i$.
\vspace{.2cm}

\begin{claim} 
\label{claim_11}
For any $i$, either $F_i \cdot R_i >0$ or $\Delta''_i \cdot R_i >0$.
\end{claim}
\begin{proof}[Proof of Claim \ref{claim_11}]
By the definition of the MMP with scaling, at each step of \eqref{mmp.long.eq} there exists a positive real number $\lambda_i$ such that $K_{\FF_i}+\Delta_i'+\lambda_i G_i$ is $f$-nef and moreover
\begin{eqnarray}
\label{inters.1.eq}
(K_{\FF_i}+\Delta_i'+\lambda_i G_i) \cdot R_i=0, & \; \text{and}\\ 
\label{inters.2.eq}
(K_{\FF_i}+\Delta_i') \cdot R_i<0. &
\end{eqnarray}
For any $i, \; \lambda_i >1$: in fact, assuming $\lambda_i \leq 1$ we reach an immediate contradiction since
\[
K_{\FF_i}+\Delta_i'+\lambda_i G_i = (1-\lambda_i)(K_{\FF_i}+\Delta_i') - \lambda_i \epsilon F_i -\lambda_i \Delta_i''
\]
would then be non-pseudoeffective over $Y$ -- this holds true in view of the fact that $\dim X_i > \dim Y$.
By \eqref{inters.1.eq}, 
\[
((1-\lambda_i)(K_{\FF_i}+\Delta_i')\cdot R_i=\lambda_i (\epsilon F_i+ \Delta_i'') \cdot R_i,
\]
and the condition $\lambda_i > 1$ together with \eqref{inters.2.eq} imply that $(\epsilon F_i+ \Delta_i'') \cdot R_i >0$, which proves the claim.
\end{proof}

\begin{claim}
\label{claim_12}
For any $i, \; \nklt(\FF_i, \Delta_i) = \nklt(\FF_i, \Delta'_i) = {\rm Supp}(F_i+ I_i)$ and the number of connected components of $\nklt(\FF_i, \Delta'_i)$ is independent of $i$.
\end{claim}
\begin{proof}[Proof of Claim \ref{claim_12}]
For any $i, \; \Delta_i \geq \Delta'_i$ hence $\nklt(\FF_i, \Delta_i) \supset \nklt(\FF_i, \Delta'_i)$. 
On the other hand, as the support of $\Delta_i - \Delta_i'$ is contained in $\Delta_i'+I_i$, then $\nklt(\FF_i, \Delta_i) \subset \nklt(\FF_i, \Delta'_i)$.
Moreover, since $(\FF_i, \Delta'_i)$ is F-dlt, $\nklt(\FF_i, \Delta'_i) = {\rm Supp}(F_i+ I_i)$.

We now prove the second part of the statement.
If $s_{i+1} \colon X_{i} \to X_{i+1}$ is a divisorial contraction, let $E$ be the prime divisor contracted by $s_{i+1}$.
Since $F_i \cdot R_i > 0$ or $\Delta''_i \cdot R_i >0$ it follows that the image of the exceptional locus of $s_{i+1}$ is contained in $\nklt(\FF_{i+1}, \Delta'_{i+1})$. 
But then Lemma \ref{connct.lemma.bir.case.rho=1} implies that the number of connected components of $\nklt(\FF_{i+1}, \Delta'_{i+1})$ in a neighborhood of $X_{i+1, y}$ must be the same as that of $\nklt(\FF_{i}, \Delta'_{i})$ around $X_{i, y}$, since $s_{i+1}$ is $(K_{\FF_i}+ \Delta_i')$-negative.

If $s_{i+1} \colon X_{i} \dashrightarrow X_{i+1}$ is a flip, let $z^-_{i} \colon X_i \to Z_i$ be the associated flipping contraction and $z^+_{i+1} \colon X_{i+1} \to Z_i$ the other small map involved in the flip.
By the first part of the proof, we know that $\nklt(\FF_{i}, \Delta'_{i})=F_i+I_i$ and $\nklt(\FF_{i+1}, \Delta'_{i+1})=F_{i+1}+I_{i+1}$.
Hence, on $Z_i$, $z_i^-(\nklt(\FF_{i}, \Delta'_{i})=z_i^{+}(\nklt(\FF_{i+1}, \Delta'_{i+1}))$, as $s_{i+1}, z_i^-, z_i^+$ are all small maps.
Hence it suffices to prove that the number of connected components of $\nklt(\FF_{i}, \Delta'_{i})$ (resp. $\nklt(\FF_{i+1}, \Delta'_{i+1})$) around $X_{i, y}$ (resp. $X_{i+1, y}$) is the same as that of $z_i^-(\nklt(\FF_{i}, \Delta'_{i}))$ (resp. $z_i^{+}(\nklt(\FF_{i+1}, \Delta'_{i+1}))$) around $Z_{i, y}$.
Lemma \ref{connct.lemma.bir.case.rho=1} implies that the number of connected components of $\nklt(\FF_{i}, \Delta'_{i})$ in a neighborhood of $X_{i, y}$ must be the same as that of $z_i^-(\nklt(\FF_{i}, \Delta'_{i}))$ around $Z_{i, y}$, since $z_{i}^-$ is $(K_{\FF_i}+ \Delta_i')$-negative.
On the other, by Claim~\ref{claim_11} either $F_i \cdot R_i >0$ or $\Delta''_i \cdot R_i >0$, which implies that the exceptional locus of $z_i^-$ is contained in either $F_{i+1}$ or $\Delta''_{i+1}$. 
Since $F_i = \Supp(\Delta_i'') \subset \nklt(\FF_{i+1}, \Delta'_{i+1})$, then $\nklt(\FF_{i+1}, \Delta'_{i+1})$ is connected around every fibre of $z_i^-$, hence the number of connected components of $\nklt(\FF_{i+1}, \Delta'_{i+1})$ around $X_{i+1, y}$ is the same as that of $z_i^{+}(\nklt(\FF_{i+1}, \Delta'_{i+1}))$ around $Z_{i, y}$, which concludes the proof.
\end{proof}
By Special Termination, \cite[Theorem 7.1]{CS18}, and Claim \ref{claim_11}, the run of the MMP in \eqref{mmp.long.eq} must terminate and, since $K_\FF + \Delta'$ is non-pseudoeffective over $Y$, the final step will be a Mori fibre space
\begin{equation*}
\xymatrix{ 
    X_n \ar[dr]^g \ar[dd]_{f_n} & \\
    & Z \ar[dl]. \\
    Y &
}
\end{equation*}
By Claim \ref{claim_12} it suffices to prove that $\nklt(\FF_n, \Delta_n)$ is connected in a neighborhood of $X_{n, y}$.
On $X_n, \; \nklt(\FF_n, \Delta_n) = {\rm Supp}(F_n+ I_n)$.
As $I_n$ is $\cal F_n$-invariant every component of $I_n$ must be vertical over $Z$.
As $F_n \cdot R_n >0$ or $\Delta''_n \cdot R_n >0$, there exists at least one component of $F_n$ which dominates $Z$ and contains only one horizontal component.  
Let $z \in Z$ be a point and observe that $\text{dim}(g^{-1}(z)) \leq 2$.

If $\text{dim}(g^{-1}(z)) = 2$ for all $z \in Z$, then since $\rho(X/Z) = 1$, it follows
that every horizontal component of $F_n$ is $g$-ample;
hence, any 2 horizontal components of $F_n$ intersect along any fibre of $g$.
If $\text{dim}(g^{-1}(z)) = 1$ for some (equivalently any) $z$ then since $-(K_{\FF_n}+\Delta_n)$ is $g$-ample
we see that $F_n$ contains at most $1$ horizontal component.
Thus, $\nklt(\FF_n, \Delta_n)$ must be connected a neighborhood of $X_{n, y}$. 
But this gives a contradiction.

{\bf Step 2.} \emph{ In this step we reduce the general case to the case of $f$-ample $H$}.
\newline
Here it suffices to copy the proof of Proposition \ref{bir.conn.prop} verbatim.
\end{proof}

\subsection{Adjunction for foliated pairs}
\label{adj.sect}
The goal of this section is to illustrate adjunction theory for foliated threefolds.
Let us highlight the fact that in~\cite{CS18} a $\mathbb Q$-factorial threefold $X$ is simply an analytic variety which is (globally) $\mathbb Q$-factorial.
We will work in this set-up throughout \S~\ref{adj.sect}-\ref{inv.adj.sect}; 
the reader should keep this observation in mind when encountering foliated adjunction throughout the paper.

Let us recall the following adjunction for foliations with non-dicritical singularities.

\begin{lem}[Adjunction]
\cite[Lemma~3.18]{CS18}
\label{adjunction}
Let $X$ be a $\bb Q$-factorial threefold, let $\cal F$ be a co-rank one foliation with non-dicritical singularities.
\\
Suppose that $(\cal F, \epsilon(S)S+\Delta)$ is lc (resp. lt, resp. F-dlt) for a prime divisor $S$ and a $\bb Q$-divisor $\Delta\ge 0$ on $X$ which does not contain $S$ in its support.  
Let $\nu\colon S^{\nu}\to S$ be the normalization and let $\cal G$ be the restricted foliation to $S^\nu$.  
\\
Then, there exists $\Theta\ge 0$ on $S^\nu$ such that  
\begin{equation}
\label{old.adj.form.eqn}
\nu^*(K_{\cal F}+\epsilon(S)S+\Delta) = K_{\cal G}+\Theta.
\end{equation}

Moreover, we have: 
\begin{itemize}
\item Suppose $\epsilon(S) =1$. Then 
$(\cal G, \Theta)$ is lc (resp. lt, resp. F-dlt).

\item Suppose  $\epsilon(S) = 0$ and that $(\cal F, \Delta)$ is F-dlt and $S$ and $\text{sing}(\cal F)\cap S$ are normal.
Then $(S^\nu, \Theta' := \lfloor \Theta \rfloor_{red}+\{\Theta\})$
is lc (resp. lt, resp. dlt).
\end{itemize}
\end{lem}

We wish to generalize this result to an adjunction formula which holds in full generality.

\begin{lem}[General Adjunction]
\label{adjunction_new}
Let $X$ be a threefold and let $\cal F$ be a co-rank one foliation on $X$.
\newline
Suppose that $(\cal F, \epsilon(S)S+\Delta)$ is a foliated log pair for a prime divisor $S$ and a $\bb Q$-divisor $\Delta\ge 0$ on $X$ which does not contain $S$ in its support.  
Let $\nu\colon S^{\nu}\to S$ be the normalization and let $\cal G$ be the restricted foliation to $S^\nu$.  
\newline
Then, there exists $\Theta\ge 0$ on $S^\nu$ such that  
\begin{equation}
\label{new.adj.form.eqn}
\nu^*(K_{\cal F}+\epsilon(S)S+\Delta) = K_{\cal G}+\Theta.
\end{equation}
\end{lem}

In the hypotheses of Lemma \ref{adjunction_new}, we will refer to $\Theta$ as the different ${\rm Diff}_S{\Delta}$ of $\Delta$ on $S$.

\begin{proof}
Let $\pi \colon Y \to X$ be a F-dlt modification for $(\FF, \epsilon(S)+\Delta)$ and let $S'$ be the strict transform of $S$ on $Y$. Writing
\begin{equation}
\label{lc.inv.adj.eq}
K_{\FF_Y}+ \epsilon(S')S' + \Delta_Y = \pi^\ast(K_\FF+\epsilon(S)S+\Delta),
\end{equation}
the pair $(\FF_Y,\epsilon(S')S'+ \Delta'_Y)$ is F-dlt, where $\Delta'_Y := \pi_\ast^{-1} \Delta + \sum_{\pi-{\rm exc}} \epsilon(E) E$ 
Denoting $\Delta''_Y := \Delta_Y- \Delta'_Y$, it is immediate that the support of $\Delta''_Y$ does not contain $S'$ and $K_{\FF_Y} + \epsilon(S')S'+\Delta_Y'+\Delta''_Y = \pi^\ast(K_\FF+\epsilon(S)S+\Delta)$.
As $(\FF_Y,\epsilon(S')S'+ \Delta'_Y)$ is F-dlt, Lemma \ref{adjunction} implies that there exists $\Theta_1$ such that
\[
(K_{\FF_Y}+\epsilon(S')S'+\Delta'_Y)\vert_{S'^\nu} = K_{\GG}+\Theta_1.
\]
Hence, 
\[
(K_{\FF_Y}+\epsilon(S')S'+\Delta_Y)\vert_{S'^\nu} = 
(K_{\FF_Y}+\epsilon(S')S'+\Delta'_Y + \Delta_Y'')\vert_{S'^\nu}
= K_{\GG}+\Theta_1 + \Delta_Y''\vert_{S'^\nu}.
\]
Hence, it suffices to take $\Theta:= \Theta_1+\Delta_Y''\vert_{S'^\nu}$.
\end{proof}

The two equations \eqref{old.adj.form.eqn}, \eqref{new.adj.form.eqn} represent the adjunction formula for foliations, where \eqref{new.adj.form.eqn} is a 
generalized version of the one proven in \cite{CS18}.
On the other hand, in the more general framework of Lemma \ref{adjunction_new}, it is not possible control the 
singularities of the restriction of the pair $(\FF, \Delta)$ to a codimension one log canonical center.

\subsection{Inversion of adjunction}
\label{inv.adj.sect}
We are now ready to prove inversion of adjunction for foliated pairs.
\begin{thm}
\label{t_inversion_of_adjunction}
Let $X$ be a $\mathbb{Q}$-factorial threefold and let $\FF$ be a co-rank one foliation.
Consider a prime divisor $S$ and an effective $\mathbb{Q}$-divisor $\Delta$ on $X$ which does not contain $S$ in its support.  
Let $\nu\colon S^{\nu}\to S$ be the normalization and let $\cal G$ be the restricted foliation on $S^\nu$ and $\Theta$ be the foliation different for $(\FF, \Delta)$ on $S^\nu$.  
Suppose that 
\begin{itemize}
    \item if $\epsilon(S)=1$ then $(\GG, \Theta)$ is lc;
    \item if $\epsilon(S) = 0$ then $(S^\nu, \Theta)$ is lc.
\end{itemize}
Then $(\FF, \epsilon(S)S+ \Delta)$ is lc in a neighborhood of $S$.
\end{thm}

\begin{proof}
Let $\pi \colon Y \to X$ be an F-dlt modification for the pair $(\FF, \epsilon(S)S+ \Delta)$ and let $S'$ be the strict transform of $S$ on $Y$. 
Writing
\begin{equation}
\label{dlt.mod.adj.eq}
K_{\FF_Y} + \varepsilon(S')S'+ \Delta_Y = \pi^\ast(K_\FF+\varepsilon(S)S+\Delta),
\end{equation}
the pair $(\FF_Y, \varepsilon(S')S'+ \Delta'_Y)$ is F-dlt, where $\Delta'_Y := \pi_\ast^{-1} \Delta + \sum_{\pi-{\rm exc}} \epsilon(E) E$ and $K_{\FF_Y}+ \varepsilon(S')S' + \Delta_Y'$ is $\pi$-nef, see \cite[Proof of Theorem~8.1]{CS18}.
Denoting $\Delta''_Y := \Delta_Y- \Delta'_Y$, then $-\Delta''_Y$ is $\pi$-nef, since by \eqref{dlt.mod.adj.eq} $-\Delta''_Y \sim_{\pi, \mathbb{R}} K_{\FF_Y}+\Delta'_Y$.
When $\epsilon(S') =1$, we will denote by $\GG'$ the restriction of $\FF_Y$ to the normalization $\nu_1 \colon S'^\nu \to S'$ of $S'$ and let $\Xi'$ be the different given by adjunction of $(\FF_Y, \varepsilon(S')S'+\Delta_Y)$.

{\bf Step 1}. {\it In this step we prove that $(\FF, \varepsilon(S)S+\Delta)$ is lc in a neighborhood of $S$ if and only if $(\FF_Y, \varepsilon(S')S'+\Delta_Y)$ is lc in a neighborhood of $S'$}.
\newline
It follows from Definition~\ref{nonklt.loc.def} that $(\FF, \varepsilon(S)S+ \Delta)$ is lc in a neighborhood $U$ of $S$ if and only $S \cap \Nlc(\FF, \varepsilon(S)S+\Delta) = \emptyset$, since $\Nlc(\FF, \varepsilon(S)S+\Delta)$ is closed.
By Remark~\ref{rmk:nonklt.F-dlt.mod}, as $(\FF_Y, \varepsilon(S')S'+ \Delta_Y')$ is F-dlt, $\Nlc(\FF_Y, \varepsilon(S')S'+\Delta_Y) = \Supp(\Delta''_Y)$ and $\Nlc(\FF, \varepsilon(S)S+\Delta) = \pi (\Supp(\Delta''_Y))$. 
By the Negativity Lemma, see~\cite[Lemma~1.3 and Appendix~A]{1907.06705},
since $-\Delta''_Y$ is $\pi$-nef, then $-\Delta''_Y$ contains any fibre of $\pi$ intersecting $\Supp(\Delta_Y'')$;
thus, $\Nlc(\FF, \varepsilon(S)S +\Delta) \cap S = \emptyset$ if and only if $\Nlc(\FF_Y, \varepsilon(S')S'+ \Delta_Y) \cap S' = \emptyset$.

{\bf Step 2} {\it In this step we prove that if $\epsilon(S')=1$, then $(\GG', \Xi')$ is lc, then
we deduce that $(\FF_Y, S'+\Delta_Y)$ is lc in a neighborhood of $S'$}.\\
By Lemma~\ref{lem_fdlt_implies_dlt}, $S'$ is normal.
Hence, $S'^\nu=S'$ and $\mathcal G'=\mathcal G \vert_{S'}$.
\begin{equation}
\nonumber
K_{\GG'}+ \Xi' = 
(K_{\FF_Y}+S'+\Delta_Y)\vert_{S'} =  
\pi^\ast(K_{\FF}+S+\Delta)\vert_{S'}.
\end{equation}
Hence, considering the birational morphism $\psi \colon S' \to S^\nu$, then 
\[
K_{\GG'}+ \Xi' = \psi^\ast(K_{\GG}+ \Theta).
\]
As $(\GG, \Theta)$ is lc, the same holds for $(\GG', \Xi')$.
As shown in Step $1$, we need to show that $\Supp(\Delta_Y'') \cap S' =\emptyset$.
Seeking a contradiction, let $E$ be a prime component of $\Supp(\Delta_Y'')$ intersecting $S'$, so that $\mu_E \Delta_Y >\varepsilon(E)$ and $\mu_E \Delta'_Y =\varepsilon(E)$.
Let $G$ be any prime component of $E \cap S'$.
\begin{claim}
\label{claim:G.lc.center}
$G$ is an lc center of $(\FF_Y, \Delta_Y')$
\end{claim}
\begin{proof}[Proof of Claim~\ref{claim:G.lc.center}]
This fact is an immediate consequence of Lemma~\ref{lem_fdlt_implies_dlt} and its proof if $\varepsilon(E)=1$, while, if $\varepsilon(E)=0$, then~\cite[Lemma~3.16]{CS18} implies that $Y$ has quotient singularities at the generic point of $G$ at which point the conclusion follows from a local computation on foliated surfaces, upon localizing at the generic point of $G$.
\end{proof}

By~\cite[Lemma~3.8]{CS18} and Claim~\ref{claim:G.lc.center}, $(\FF_Y, \Delta_Y')$ is log smooth at the generic point of $G$: in particular, $Y$ is smooth at the generic point of $G$ and $E$ meets $S'$ generically transversely along $G$.
Hence, taking
$
K_{\GG'}+ \Psi' = 
(K_{\FF_Y}+S'+\Delta'_Y)\vert_{S'},
$
then $\mu_G \Psi'=\varepsilon(G)=\varepsilon(E)$.
As $\mu_{E} \Delta''_{Y} >0$, 
\[
K_{\GG'}+ \Xi' = 
(K_{\FF_Y}+S'+\Delta_Y' +\Delta''_Y)\vert_{S'}=
K_{\GG'}+ \Psi' +\Delta''_Y\vert_{S'}.
\]
and $\mu_G\Xi' >\varepsilon(G)$ which contradicts the fact that $(\GG', \Xi')$ is lc.

{\bf Step 3} 
{\it In this step we prove that if $\epsilon(S')=0$ and if $(S^\nu, \Theta)$ is lc, then $(S'^\nu, \Xi')$ is lc, then
we prove that $(\FF_Y, \Delta_Y)$ is lc in a neighborhood of $S'$}.\\
We have
\begin{equation}
\nonumber
K_{S'^\nu}+ \Xi' = 
\nu_1^\ast(K_{\FF_Y}+\Delta_Y) =  
(\nu_1\circ \pi)^\ast(K_{\FF}+\Delta).
\end{equation}
Considering the birational morphism $\psi \colon S'^\nu \to S^\nu$, then
\[
K_{S'^\nu}+ \Xi' = \psi^\ast(K_{S}+ \Theta).
\]
As $(S^\nu, \Theta)$ is lc, the same holds for $(S'^\nu, \Xi')$.
We need to show that $\Supp(\Delta_Y'') \cap S' =\emptyset$.
Seeking a contradiction, let $E$ be a prime component of $\Supp(\Delta_Y'')$ intersecting $S'$, 
so that $\mu_E \Delta_Y >\varepsilon(E)$ and $\mu_E \Delta'_Y =\varepsilon(E)$.
Let $G$ be any prime component of $\nu_1^{-1}(E)$.
Setting 
$K_{S'^\nu}+ \Psi' =  \nu_1^\ast(K_{\FF_Y}+\Delta'_Y)$,
then $\Psi' \leq \Xi'$ and ~\cite[Corollary~3.20]{CS18} implies that $\mu_G \Psi' \geq 1$.
As $\mu_{E} \Delta''_{Y} >0$ and 
$
K_{S'^\nu}+ \Xi' = 
\nu_1^\ast(K_{\FF_Y}+S'+\Delta_Y' +\Delta''_Y)=
K_{S'^\nu}+ \Psi' + \nu_1^\ast\Delta''_Y,
$
then $\mu_G\Xi' >1$ which contradicts the fact that $(S'^\nu, \Xi')$ is lc.
\end{proof}

This is the adaptation to foliations of the classic statement of inversion of adjunction for log pairs, cf. \cite[Theorem 5.50]{KM98}.
Nonetheless, it is not the most general form of inversion of adjunction that one could hope for.
In fact, if we look at the statement of Lemma \ref{adjunction}, we see that the natural divisor to look at when $\epsilon(S)=0$ would be, in the notation of the lemma, the divisor $\Theta'$ rather than the foliated different $\Theta$ -- 
let us recall that $\Theta':=\lfloor \Theta \rfloor_\mathrm{red} + \{ \Theta\}.$
As, by definition $\Theta' \leq \Theta$ it follows immediately that if $(S, \Theta)$ is lc, so is 
$(S, \Theta')$, but it would be even more interesting to have a statement of inversion of adjunction 
that only assumes $(S, \Theta')$ is lc.

\section{A vanishing result}

In this section we prove a relative vanishing theorem for foliations.

We make the following easy observation whose proof we leave to the reader.

\begin{lem}
\label{l_stupidcover}
Let $f\colon Y \rightarrow X$ be a morphism of varieties

Let $\sigma \colon X' \rightarrow X$ be surjective and \'etale and let $Y' = Y \times_{X} X'$.
Let $f'\colon Y' \rightarrow X'$ and $\tau\colon Y' \rightarrow Y$ be the projections.

Let $L$ be a line bundle on $Y$.  Suppose that $R^if'_*\tau^*L = 0$. Then $R^if_*L = 0$.
\end{lem}

\begin{lem}
\label{l_approx}
Let $f\colon Y \rightarrow X$ be a surjective birational
projective morphism
of normal varieties of dimension at most 3
 and let $(\cal F, \Delta)$ be an F-dlt foliated pair on $Y$ with $\lfloor \Delta \rfloor = 0$.
Suppose that $Y$ is $\bb Q$-factorial
and that every fibre of $f$ is tangent to $\cal F$. 

Let $P \in X$ be a closed point.  Then there exists an \'etale neighborhood
$\sigma\colon X' \rightarrow X$ of $P$,
a small $\bb Q$-factorialization $\mu \colon W \rightarrow Y':=Y \times_X X'$
and a reduced divisor $\sum T_i$ on $W$
such that writing
$\Delta_W = \mu^*\Delta$ and $\cal F_W = \mu^{-1}\cal F$ and $f'\colon W \rightarrow X'$ for the induced map
we have
\begin{enumerate}
\item \label{i_aa} $\sum T_i$ is nef over $Y'$;
\item \label{i_bb} $(K_{\cal F_W}+\Delta_W) - (K_{W}+\Delta_W+\sum T_i)$ is $f'$-nef; and
\item \label{i_cc} $(W, \Delta_W+(1-\epsilon)\sum T_i)$ is klt for all $\epsilon>0$.
\end{enumerate}
\end{lem}
\begin{proof}
First, notice that since $Y$ is $\bb Q$-factorial we may apply \cite[Theorem 11.3]{CS18}
to see that $Y$ is klt.  Since $(\cal F, \Delta)$ is F-dlt we also know that $\cal F$ is non-dicritical by 
Theorem \ref{t_canimpliesnondicritical}.

Let $\{S_1, ..., S_N\}$ be the collection of all separatrices of
$\cal F$ meeting $f^{-1}(P)$, formal or otherwise.
Fix $n\gg 0$ sufficiently large. By \cite[\S 4, 5]{CS18} there is an \'etale cover 
$\sigma\colon X' \rightarrow X$ such that
we may find divisors $R_i$ on $Y'$
such that $R_i\vert_{Y'_n} = \overline{S_i}\vert_{Y'_n}$
where $\overline{S_i}= \tau^*S_i$, $Y'_n$ be the $n$-th infinitessimal neighborhood of $\sigma^{-1}(f^{-1}(P))$ and
where $\tau \colon Y' \rightarrow Y$ is the projection
and such that $\cal O_{\widehat{Y'}}(\overline{S_i}) \cong \cal O_{\widehat{Y'}}(R_i)$ where
$\widehat{Y'}$ is the formal completion of $Y'$ along $\sigma^{-1}(f^{-1}(P))$.
Let $g\colon Y' \rightarrow X'$ be the other projection.

Set $(\cal F':=\tau^{-1}\cal F, \Delta':=\tau^*\Delta)$.   Since $(\cal F, \Delta)$ is F-dlt we see
that $(\cal F', \Delta')$ is as well and we may find 
an F-dlt modification $\pi\colon Z \rightarrow Y'$ such that $\mu$ is small.
Observe that $Z$ is $\bb Q$-factorial and
so $R_i':=\pi_*^{-1}R_i$ and $\overline{S_i}' :=\pi_*^{-1}\overline{S_i}$ are
$\bb Q$-Cartier.
Set $\cal G = \pi^{-1}\cal F$ and we may write $K_{\cal G}+\Gamma = \pi^*(K_{\cal F'}+\Delta')$.
We also remark that $(\cal G, \Gamma)$ is necessarily terminal at the generic point of a curve
$C \subset \text{exc}(\pi)$.  Indeed, otherwise $C$ would be a lc centre of
$(\cal G, \Gamma)$ since it is tangent to $\cal G$ which by \cite[Lemma 3.8]{CS18}
would imply that $(\cal F', \Delta')$ is log smooth at $\pi(C)$, a contradiction.

Note that $\overline{S_i}'$ are all the separatrices (formal or otherwise) which meet 
$\pi^{-1}(g^{-1}(p))$ and that $R_i'$ still approximate all the $\overline{S_i}'$.
Note that in particular they have the same intersection numbers with all
curves contained in $\pi^{-1}(g^{-1}(p))$.

Since $Y$ is klt, the same is true of $Z$ and so we may run a
$K_Z+\delta \sum R'_i$-MMP over $Y'$ for some $\delta>0$ sufficiently small.  
Denote this MMP $\phi\colon Z \dashrightarrow W$ 
and set $T_i:=\phi_*R'_i$ and $\widetilde{S}_i =\phi_*\overline{S}'_i$.
Set $\cal F_W = \phi_*\cal G$ and $\Delta_W = \phi_*\Gamma$.
Observe that each step of this MMP is $K_{\cal G}+\Gamma$ trivial
and that $(\cal F_W, \Delta_W)$ is F-dlt.
We claim that $W$ satisfies all our required properties.

Item \ref{i_aa} holds since $K_Y$ is $\bb Q$-Cartier and so $K_Z$ (and hence $K_W$)
is trivial over $Y'$.  Thus $K_W+\delta\sum T_i$ being nef over $Y'$ implies
that $\sum T_i$ is nef over $Y'$.  

To prove~(\ref{i_bb}), 
let $C \subset \pi^{-1}(P)$.  
Note that by non-dicriticality of $(\cal F, \Delta)$ and our assumptions on $f$ we have that $C$ is 
tangent to $\cal F_W$.   Moreover, if $(\cal F_W, \Delta_W)$ is canonical at the generic point
of $C$ then $(\cal F_W, \Delta_W)$ is log smooth at a general point of $C$, \cite[Lemma 3.8]{CS18}.
So, up to relabelling the $S_i$ we may assume that $C \subset \widetilde{S}_1$ and that $\widetilde{S}_1$
gives a strong separatrix at a general point of $C$ if $(\cal F_W, \Delta_W)$ is canonical at the generic point of $C$ and $(\cal F_W, \Delta_W)$
has a saddle node at the general point of $C$, 
\cite[pg. 3]{Brunella00} for a recollection on saddle nodes
and weak separatrices on surfaces, but which works equally in the current setting.

By \cite[Corollary 3.20]{CS18} we may write
\[(K_{\cal F_W}+\Delta_W)\vert_{\widetilde{S}_1} = K_{\widetilde{S}_1}+\Theta\]
and
\[(K_{W}+\Delta_W+\sum \widetilde{S}_i)\vert_{\widetilde{S}_1} = K_{\widetilde{S}_1}+\Theta'\]
where $\Theta \geq \Theta'$ and the coefficient of $C$ in both these divisors is the same.
It follows that 
\[((K_{\cal F_W}+\Delta_W)-(K_{W}+\Delta_W+\sum \widetilde{S}_i))\cdot C \geq 0\]
and since $\{T_1, ..., T_N\}$ approximate the $\widetilde{S}_1$
we have
\[((K_{\cal F_W}+\Delta_W)-(K_{W}+\Delta_W+\sum T_i))\cdot C \geq 0.\]
Since $C$ was arbitrary we get our claimed nefness.

Next, observe that each step of the MMP $\phi\colon Z \dashrightarrow Y$
is $K_{\cal G}+\Gamma$ trivial so 
we still have that $(\cal F_W, \Delta_W)$ is F-dlt and so $\cal F_W$ is non-dicritical by Theorem \ref{t_canimpliesnondicritical}
 and all the
log canonical centres of $(\cal F_W, \Delta_W)$ are contained in $\Supp(\sum T_i)$. 
So we may apply \cite[Lemma 3.16]{CS18}
to see that $(W, \Delta_W+(1-\epsilon)(\sum T_i))$ is klt for all $\epsilon>0$. This gives us item \ref{i_cc}.
\end{proof}

\begin{thm}
\label{t_vanishing}
Set up as above.
Let $(\cal F, \Delta)$ be a F-dlt pair and
let $L$ be a line bundle such that $L-(K_{\cal F}+\Delta)$ is $f$-nef and big.

Suppose moreover that either 
\begin{enumerate}
\item $L-(K_{\cal F}+\Delta)$ is $f$-ample; or
\item $\Delta = A+B$ where $A$ is $f$-ample and $B \geq 0$.
\end{enumerate}

Then $R^if_*L = 0$ for $i>0$.
\end{thm}

\begin{proof}
If $\Delta = A+B$ where $A$ is $f$-ample then replacing $\Delta$ by $\Delta-\delta A$ for $\delta>0$
small we may freely assume that $L-(K_{\cal F}+\Delta)$ is $f$-ample.  Moreover, perhaps replacing
$\Delta$ by $\Delta -\epsilon \lfloor \Delta \rfloor$ for some $\epsilon >0$ sufficiently small
we may assume that $\lfloor \Delta \rfloor = 0$.

As in Lemma \ref{l_approx} we see that $Y$ is klt and so it has rational singularities.
By Lemm \ref{l_stupidcover}  and the fact that $Y$ has rational singularities
we see that $R^if_*L = 0$ provided $R^if'_*L' = 0$
where $L' = \mu^*\tau^*L$ where $\mu$ and $\tau$ are as in Lemma \ref{l_approx} (and its proof).

Next, 
observe that $L'- (K_{\cal F_W}+\Delta_W)$ is $f'$-big and nef
and is strictly positive on any curve not contracted by $\mu$.
Thus we see by Lemma \ref{l_approx} Item \ref{i_bb}
 that
\begin{multline*}
L' - (K_W+\Delta_W+\sum T_i) = \\
 (L'-(K_{\cal F_W}+\Delta_W))+((K_{\cal F_W}+\Delta_W)-(K_W+\Delta_W+\sum T_i))
\end{multline*}
is $f'$-big and nef and is strictly positive on any curve not contracted by $\mu$.

So for $\epsilon>0$ sufficiently small since $\sum T_i$ is $\mu$-nef
by Lemma \ref{l_approx} Item \ref{i_aa}
\[L'-(K_W+\Delta_W+(1-\epsilon)\sum T_i)\]
is $f'$-big and nef.

Thus we may apply relative Kawamata-Viehweg vanishing to conclude that $R^if'_*L' =0$ for $i>0$.
\end{proof}

\section{Malgrange's theorem}
\label{sect.malgr}

In this section we prove a version of Malgrange's theorem on singular threefolds. 
A weaker version of this statement was proven in \cite{Spicer17}.
Results in this direction were
achieved in \cite{CLN08} and some of our ideas have been inspired by their approach.

Let $P \in X$ be a germ of a threefold and let $\cal F$ be a co-rank 1 foliation on $X$ defined by a holomorphic
1-form $\omega$.  We say that $f \in \cal O_{X, p}$ is a first integral for $\cal F$ if $df\wedge \omega =0$.

\begin{thm}
\label{t_Malgrange}
Let $P \in X$ be a germ of an isolated (analytically) $\bb Q$-factorial threefold singularity with a co-rank 1 foliation $\cal F$. 
Suppose that $\cal F$ has an isolated canonical singularity at $P$.

Then $\cal F$ admits a holomorphic first integral.
\end{thm}

It would be ideal to drop the $\bb Q$-factoriality assumption in the theorem.  We are able to do this when $\cal F$
is terminal, see Corollary \ref{c_term_malg}. 

Theorem \ref{t_Malgrange} has the following immediate consequence.

\begin{cor}
\label{c_sep_can}
Let $P \in X$ be a germ of an isolated threefold singularity with a co-rank 1 foliation $\cal F$.
Suppose that $X$ is $\bb Q$-factorial and that $\cal F$ is canonical.
Then $\cal F$ has a separatrix at $P$.
\end{cor}
\begin{proof}
If $\cal F$ is smooth outside of $P$ then this follows directly from Theorem \ref{t_Malgrange}.
Otherwise let $Z \subset \text{sing}(\cal F)$ be a curve and note that $Z$ is tangent to $\cal F$.  
Observe that there is a germ of a separatrix for all $Q \in Z-P$.
By Theorem \ref{t_canimpliesnondicritical} $\cal F$ is non-dicritical 
and so by Lemma \ref{l_formalseparatrix} we may extend $S_Q$ to a neighborhood of $Z$, which in turn 
gives separatrix at $P$.
\end{proof}

Recall that in general even if $\cal F$ is non-dicritical, if $P \in X$ is a singular point then there may be no separatrices at $P$.  See \cite{Camacho88} results in this direction on surfaces.

\subsection{Controlling the singularities of $X$ and $\cal F$}
\label{ss:controllingsing}

The goal of this subsection is to show that under the hypotheses of Theorem \ref{t_Malgrange} we have that $X$ has log terminal singularities.

We will need the following version of the classical Camacho-Sad formula for F-dlt foliations.
It follows as a special case of the Camacho-Sad formula for foliations on varieties with quotient singularities proven in
\cite[Proposition 3.12]{DruelOu19}.  We refer to \cite[Definition 3.10]{DruelOu19} for the definition of the Camacho-Sad index.

\begin{lem}
\label{FdltCS}
Let $X$ be a normal surface and $\cal F$ an F-dlt foliation.
Let $C$ be a compact $\cal F$-invariant curve.
Then 
\[C^2 = \sum_{p \in \text{sing}(\cal F)\cap C} CS(p, \cal F, C).\]
\end{lem}

\begin{lem}
\label{l_canimpliesklt}
Let $P \in X$ be a germ of an isolated threefold singularity and let $\cal F$ be a co-rank 1 foliation with canonical singularities such that $\cal F$ is smooth away from $P$. 
Suppose that $K_X$ is $\bb Q$-Cartier. Then $X$ is log terminal.
\end{lem}
\begin{proof}
If $\cal F$ is terminal then the result follows from Theorem \ref{t_canimpliesnondicritical}.
So suppose that $\cal F$ has canonical but not terminal singularities.

Let $\mu:(\overline{X}, \overline{\cal F}) \rightarrow (X, \cal F)$ be an F-dlt modification.
Let $E = \sum E_i = \text{exc}(\mu)$. Since $\cal F$ is canonical we have $\mu^*K_{\cal F} = K_{\overline{\cal F}}$.
Moreover, we may freely assume that $\mu$ is not the identity.
By Theorem \ref{t_canimpliesnondicritical} $\cal F$ is non-dicritical and so $E$ is $\overline{\cal F}$-invariant.
Let  $Z$ be a 1-dimensional component of $\text{sing}(\overline{\cal F})\cap E$.
By \cite[Lemma 3.14]{CS18} either $\cal F$ is terminal at the generic point of $Z$ or
$\overline{X}$ is smooth at the generic point of $Z$, and at a general point of $Z$
$\overline{\cal F}$ has simple singularities and there there are two separatrices (possibly formal) containing $Z$.  
However, Proposition \ref{prop_term_surf_class} applied to a general hyperplane passing through $Z$ and the 
restricted foliation on this hyperplane implies that $\cal F$ cannot be terminal at the generic point of $Z$.

Write $K_{\overline{X}}+\sum E_i = \pi^*(K_X)+\sum a_iE_i$.
By \cite[Lemma 8.9]{Spicer17} we see that $K_{\overline{\cal F}} - (K_{\overline{X}}+\sum E_i) = -\sum a_iE_i$ is $\pi$-nef 
away from finitely many curves which implies by the Negativity Lemma,~\cite[Lemma~1.3]{1907.06705} , that $\sum a_iE_i \geq 0$ and since $\Supp(\sum E_i) = \pi^{-1}(P)$, then either
\begin{enumerate}
\item \label{i_lt} $a_i > 0$ for all $i$; or 
\item \label{i_lc} $a_i=0$ for all $i$.
\end{enumerate}

By \cite[Lemma 3.16]{CS18}, we see that that $(\overline{X}, (1-\epsilon)\sum E_i)$ is klt and so 
if we are in Case \ref{i_lt} then we see immediately that $X$ is klt.

So suppose for sake of contradiction that we are in Case \ref{i_lc}, i.e., $a_i = 0$ for all $i$ and so $K_{\overline{X}}+E \sim_{\bb Q} 0$.

We first claim that if $Z \subset \text{sing}(\overline{\cal F})\cap E$ is a 1-dimensional component admitting two separatrices contained in $E$
then $Z$ is not a saddle node. Indeed, suppose for sake of contradiction 
that there exists $Z \subset E_i$ such that
$Z$ is a saddle node and $E_i$ is the weak separatrix of the saddle node, see \cite[pg. 3]{Brunella00} for a recollection on saddle nodes
and weak separatrices on surfaces, but which works equally in the current setting.
Write
\[K_{\overline{\cal F}}\vert_{E_i} = K_{E_i}+\Theta\]
and
\[(K_{\overline{X}}+\sum E_j)\vert_{E_i} = K_{E_i}+\Theta'.\]
By Lemma \ref{adjunction} we know that $\Theta \geq \Theta'$.
Since $E_i$ is the weak separatrix of a saddle node along $Z$
in appropriate (formal) local coordinates around a general point of $Z$
we see that $\overline{\cal F}$ is generated by a $1$-form $\omega$ of the form
$z(1+\nu w^k)dw+w^kdz$ where $E_i = \{z = 0\}$, $\nu \in \bb C$ and $k \geq 2$.
The coefficient of $Z$ in $\Theta$ is the order of vanishing 
of $\omega\vert_{E_i}$ along $Z$, which in turn is exactly $k\geq 2$.
On the other hand, since $(\overline{X}, \sum E_j)$ is log canonical we see that
the coefficient of $Z$ in $\Theta'$ is at most 1. However, this implies that 
$K_{\overline{\cal F}} - (K_{\overline{X}}+\sum E_j)$ cannot be $\pi$-trivial, a contradiction.

A similar argument shows that for each 1-dimensional component $Z \subset \text{sing}(\overline{\cal F})\cap E$
that $Z$ admits two separatrices, both of which are contained in $E$.
In particular, each 1-dimensional component $Z \subset \text{sing}(\overline{\cal F})\cap E$ admits 2 non-zero eigenvalues.

The rest of the argument proceeds in an essentially identical fashion to the proof of the first part of \cite[Theorem IV.2.2]{McQuillan08}.
We will explain the rest of his argument for the reader's convenience.
Since $K_{\overline{\cal F}}\sim_{\bb Q} 0$ and $K_{\overline{X}}+E \sim_{\bb Q} 0$ 
we see that $N^*_{\overline{\cal F}}+E \sim_{\bb Q}0$.

Let $H \subset \overline{X}$ be a general ample divisor,
and let $\cal G$ be the restricted foliation on $H$
and let $E \cap H = \cup C_i = C$.  Set $S = \text{sing}(\cal G)\cap C$
and notice that $C_i \cap C_j \subset S$ for $i \neq j$.

If $H$ is general enough we see that $(\overline{\cal F}, H)$ is F-dlt and so
$\cal G$ is F-dlt.  Even better, if $H$ is general enough we see
that $N^*_{\overline{\cal F}}\vert_H = N^*_{\cal G}$ and so $N^*_{\cal G}+\sum C_i \sim_{\bb Q}0$.
Observe that since $\cal G$ is F-dlt we see that $C$ is a nodal curve.

For $	q \in S$ let $\partial_q$ be a vector field generating $\cal G$ near $q$.
\begin{claim}
The ratio $\lambda_q$ of the eigenvalues of $\partial_q$ is a root of unity. 
\end{claim}
\begin{proof}[Proof of claim]
We may check this after taking a cover ramified along a general ample divisor $A$, and so after taking the index one cover
associated to $K_{\cal G}$ on $H\setminus A$ we may freely assume that $K_{\cal G}$ is Cartier.  By Lemma \ref{lem_Fdlt_surface}
it follows that $H$ is smooth.

For any $p \in C \setminus S$ set $U_p$ to be small open set so that $\cal G$ is defined by a $1$-form $\omega_p = dz_p$ where $\{z_p = 0\} = C \cap U_p$.
For any $q \in S$ set $U_q$ to be a small open subset so that $\cal G$ is defined by a $1$-form $\omega_q = x_qa_qdy_q+y_qb_qdx_q$
where $\{x_qy_q = 0\} = C \cap U_q$ and where $a_q(q), b_q(q) \neq 0$.

Let $\{(U_{pq}, h_{pq})\}, \{(U_{pq}, g_{pq})\} \in H^1(H, \cal O_H^*)$ be the coycles associated to $\cal O_H(C)$ and $N_{\cal G}$ respectively
where $U_{pq} = U_p \cap U_q$.

If $p, p' \in C \setminus S$ so that $U_{pp'} \neq \emptyset$ then $dz_p = h_{pp'} dz_{p'}$ and so $g_{pp'} = h_{pp'}$
when restricted to $C$.
If $p \in C \setminus S$ and $q \in S$ so that $U_{pq} \neq \emptyset$ we have that $z_p = h_{pq}(x_qy_q)$ and so $dz_p = h_{pq}d(x_qy_q)$,
when restricted to $C$,
which in turn gives that $dz_p = h_{pq}a_q^{-1}\omega_q$ or $=h_{pq}b_q^{-1}\omega_q$ depending on whether
$p \in \{x_q=0\}$ or $p \in \{y_q = 0\}$.  In particular, we see that after restricting to $C$ we have an equality
$g_{pq} = h_{pq}b_q^{-1} = h_{pq}a_q^{-1}$.

Since $N^*_{\cal G}+C\sim_{\bb Q}0$ we see that $\{(U_{pq}, h_{pq}g_{qp})\}$ is a torsion cocycle, and so for some $m$
we have that $\{(U_{pq}, (h_{pq}g_{qp})^m)\}$ is a trivial cocycle.  
Set $C_{p} = U_p \cap C$, $C_{pq} = U_{pq}\cap C$ and note that $\{(C_{pq}, (h_{pq}g_{qp})^m)\}$ is still a trivial cocyle.

We may therefore find invertible functions $f_p$ on $C_p$
so that $f_p/f_q =  (h_{pq}g_{qp})^m)$.  Without loss of generality we may assume that for $p \in C \setminus S$
that $f_p = 1$.  From our previous calculations we see that for $q \in S$ that $f_q = a_{q}^m = b^m_q$ (where we consider $a_q, b_q$
as functions restricted to $C$).  In particular, $1 = f_q/f_q = (a_q(q)/b_q(q))^m = \lambda_q^m$ as required.
\end{proof}

Since $C$ is contractible we see that $C^2<0$ which implies
$\sum_i C_i^2 < - \sum_{i, j} C_i\cdot C_j = -2\#S$.
On the other hand, Lemma \ref{FdltCS} gives
us
\[(\sum C_i)^2 = \sum_{p \in Z} CS(p, \cal G, \sum C_i) = \sum_{p \in S} 2+\lambda_p+\frac{1}{\lambda_p}\]
which in turn gives us
\[
\sum C_i^2 = \sum \lambda_p+\frac{1}{\lambda_p}.
\]
However, each $\lambda_p$ is a root of unity and so the modulus of  $\sum_{p \in S} \lambda_p+\frac{1}{\lambda_p} = \sum_{p \in S} \lambda_p +\overline{\lambda_p}$  is bounded by $2\#S$.
This is our sought after contradiction.
\end{proof}

\subsection{Holomorphic Godbillon-Vey sequences}
\label{godb.sect}

We say that a 1-form $\omega$ is integrable provided $\omega \wedge d\omega = 0$.

\begin{definition}
Let $M$ be a complex manifold of dimension $\geq 2$ and let $\omega$ be an integrable holomorphic 1-form on $M$.
A holomorphic Godbillon-Vey sequence for $\omega$ is a sequence of holomorphic 1-forms $(\omega_k)$ on $M$
such that $\omega_0 = \omega$ and the formal 1-form
\[\Omega = dt +\sum_{j = 0}^\infty \frac{t^j}{j!}\omega_j\]
is integrable.
\end{definition}

\begin{lem}
\label{l:kill_cohom}
Let $P \in X$ be an analytic germ of an isolated $\bb Q$-factorial klt singularity with $\text{dim}(X) \geq 3$. 
Then \[H^1(X-P, \cal O_{X-P}) = 0.\]
\end{lem}
\begin{proof}
Notice that since $X$ is klt it is also a rational singularity.
Consider the long exact sequence coming from the exponential exact sequence
\[H^1(X-P, \bb Z) \xrightarrow{a} H^1(X-P, \cal O_{X-P}) 
\xrightarrow{b} H^1(X-P, \cal O^*_{X-P}).\]
By \cite[Lemma 6.2]{Flenner81} we know that $\text{im}(a) = 0$, in particular $b$ is injective.

On the other hand, we see that we have an injection 
\[H^1(X-P, \cal O^*_{X-P}) = \text{Pic}(X-P) \rightarrow Cl(P \in X)\] given by 
$L \mapsto i_*L$ where $i\colon X-P \rightarrow X$ is the inclusion, indeed  by \cite{Siu69} $i_*L$
is a coherent reflexive sheaf on $X$.  
By assumption $Cl(P \in X)$ is torsion 
and so the same is true of $H^1(X-P, \cal O^*_{X-P})$.  Since $H^1(X-P, \cal O_{X-P})$
is a $\bb C$-vector space it is a divisible group, which implies that $\text{im}(b) =0$.
Thus $H^1(X-P, \cal O_{X-P}) = 0$.
\end{proof}

The following result is proven in \cite[Lemma 2.1.1]{CLN08}.

\begin{lem}
\label{lem_hgvs_cln}
Let $M$ be a complex manifold of dimension $\geq 3$ and let $\omega$ be a holomorphic
$1$-form on $M$.  Assume that the codimension of $\text{sing}(\omega)$ is at least 3
and $H^1(M, \cal O_M) = 0$.  Then $\omega$ admits a holomorphic Godbillon-Vey sequence.
\end{lem}

\begin{cor}
\label{c:cover_hgvs}
Let $P \in X$ be a germ of an isolated analytically $\bb Q$-factorial klt 3-fold singularity.
Let $\omega$ be an integrable 1-form on $X-P$ such that $\text{sing}(\omega)$ has codimension at least 3 in $X-P$.  
Then $\omega$ admits a holomorphic Godbillon-Vey sequence.
\end{cor}

\begin{proof}
By Lemma \ref{l:kill_cohom} we have $H^1(X-P, \cal O_{X-P}) = 0$ in which case we may apply Lemma \ref{lem_hgvs_cln} to conclude.
\end{proof}

\subsection{A few technical lemmas}

\begin{lem}
\label{l:q_fact}
Let $P \in X$ be an analytic germ of a $\bb Q$-factorial and klt singularity with $\text{dim}(X) \geq 3$.
Let $\pi\colon (Q \in Y) \rightarrow (P \in X)$ be a quasi-\'etale morphism of germs.  
Then $Q \in Y$ is $\bb Q$-factorial.
\end{lem}

\begin{proof}
Let $\overline{\pi} \colon \overline{Y} \rightarrow X$ be the Galois closure of $\pi$.  
Observe that $\overline{\pi}$ is quasi-\'etale and if $\overline{Y}$ is $\bb Q$-factorial then $Y$ is $\bb Q$-factorial, \cite[Lemma 5.16]{KM98}.
Thus we may freely replace $Y$ by $\overline{Y}$ and so may assume that $\pi$ is Galois with Galois group $G$.
\newline
Suppose for sake of contradiction that $Y$ is not $\bb Q$-factorial.  Since $\pi$ is quasi-\'etale we see that $Y$ is klt and therefore $Y$ admits a small $\bb Q$-factorialization $f\colon Y' \rightarrow Y$ such that 
\begin{enumerate}
\item $G$ acts on $Y'$;
\item $f$ is $G$ equivariant; and
\item $f$ is not the identity.
\end{enumerate}
Indeed, such a $Y'$ can be found by taking a $G$-equivariant resolution $\mu\colon W \rightarrow X$ and running a $G$-equivariant $K_W+(1-\epsilon)\sum E_i$-MMP over $X$ where $\sum E_i$ is the union of the $\mu$-exceptional divisors and $\epsilon>0$ is sufficiently small.
\newline
Let $X' = Y'/G$ and observe that we have a birational morphism $g\colon X' \rightarrow X$.  Moreover, we see that $g\colon X' \rightarrow X$ is small, a contradiction of the fact that $X$ is $\bb Q$-factorial.
\end{proof}

\begin{lem}
\label{l:replace_cover}
Let $\pi\colon Y \rightarrow X$ be a finite morphism of complex varieties  and let $\cal F$ be a co-rank one foliation on $X$.
Then $\cal F$ admits a holomorphic (resp. meromorphic) first integral if and only if $\pi^{-1}\cal F$ does.
\end{lem}
\begin{proof}
We may assume without loss of generality that $\pi\colon Y \rightarrow X$ is Galois, in which case the claim is easy.
\end{proof}

We say that $f \in \bb C[[x_1, ..., x_n]]$ is a {\bf power} if there exists $g \in \bb C[[x_1, ..., x_n]]$ and an integer $m \geq 2$ such that $g^m = f$.
Observe that if $f$ is a first integral of $\omega$  and $g^m = f$ then $g$ is also a first integral of $\omega$.
Let $\widehat{\Delta}$ denote the formal completion of $\bb C$ at the origin.

\begin{lem}
\label{l_formal_frobenius}
Consider $\bb C^3 \times \bb C$ with coordinates $(z_1, z_2, z_3, t)$ and let $\Omega = dt +\sum t^i\omega_i$ be a formal 1-form where $\omega_i \in H^0(U, \Omega^1_U)$ is a  holomorphic 1-form on $0 \in U \subset \bb C^3$.  
Suppose that $\Omega$ is integrable. 
\newline
Let $0 \in D \subset U$ be a normal crossings divisor such that $\omega_i$ is zero when restricted to $D$ for all $i$.
Let $\widehat{X}$ be the formal completion of $U \times \bb C$ along $D \times 0$.
\newline
Then $\Omega$ admits a first integral in $H^0(\widehat{X}, \cal O_{\widehat{X}})$.
\end{lem}

\begin{remark}
A priori, the formal Frobenius theorem only guarantees that $\Omega$ admits a first integral in $H^0(\widehat{\bb C^4}, \cal O_{\widehat{\bb C^4}})$, with $\widehat{\bb C^4}$ 
the completion of $\bb C^4$ at the origin. 
\end{remark}
\begin{proof}
Following a change of coordinates and for ease of notation we will assume that $D = \{z_1z_2z_3 = 0\}$ (the cases where $D$ has 1 or 2 components are simpler). 

Since $\omega_i$ vanishes when restricted to $D$ for $j= 1, 2, 3$ we may write 
\[\omega_i = f^i_jdz_j+z_j\theta^i_j\]
where $f^i_j$ and $\theta^i_j$ are holomorphic.  It follows that we may write $\Omega = dt+F_jdz_j+z_j\Theta_j$
where $F_j(z_1, z_2, z_3, t) \in H^0(X_j, \cal O_{X_j})$, 
$\Theta_j = \sum H^j_i(z_1, z_2, z_3, t)dz_i \in H^0(X_j, \Omega^1_{X_j})$ and where $X_j$ is the formal completion
of $U \times \bb C$ along  $\{t = z_j = 0\}$.

We may then apply \cite[Lemma 3.1.1.]{CLN08} (or, more precisely, its proof) 
to find a first integral $G_j \in H^0(X_j, \cal O_{X_j})$ of $\Omega$.
Moreover, if we write $G_j = \sum_{m, n}t^mz_j^ng^j_{mn}$ where $g^j_{mn}$ is a convergent power series in the set of variables 
$\{z_1, z_2, z_3\} \setminus z_j$
then we may choose $G_j$ so that $g^j_{00} = 0$.  In particular, observe that this implies that if $\phi \in \text{Aut}(\widehat{\Delta})$,
then $\phi \circ G_j$ is still an element of $H^0(X_j, \cal O_{X_j})$.  Indeed, if we write $\phi\circ G_j = \sum_{mn}t^mz_j^ng'_{mn}$
then $g'_{mn} = P(g^j_{lp})_{l \leq m, p \leq n}$ where $P$ is some polynomial depending on $\phi$, in particular, $g'_{mn}$ is convergent
provided all the $g^j_{lp}$ are.
Without loss of generality we may also assume
that $G_j$ is not a power.

By considering $G_1, G_2, G_3$ as elements in $H^0(\widehat{\bb C^4}, \cal O_{\widehat{\bb C^4}})$, with $\widehat{\bb C^4}$ 
the completion of $\bb C^4$ at the
origin, we may apply \cite[Th\'{e}orem\`{e} de factorisation]{MM80} to find $\phi_{ij} \in \text{Aut}(\widehat{\Delta})$
so that $G_i = \phi_{ij} \circ G_j$.  Thus, perhaps replacing $G_j$ by $\phi_{ij}\circ G_j$ we may assume
that $G_1, G_2, G_3$ all give the same element in $H^0(\widehat{\bb C^4}, \cal O_{\widehat{\bb C^4}})$ call it $G$.
However, since $G_i \in H^0(X_i, \cal O_{X_i})$ this implies that $G$ is in fact an element of 
$H^0(\widehat{X}, \cal O_{\widehat{X}})$ and we are done.
\end{proof}

\begin{lem}
\label{l:formal_holonomy}
Let $X$ be a normal complex variety, let $D \subset X$ be a compact subvariety and let $\widehat{X}$ be the completion
of $X$ along $D$.
Let $\cal F$ be a co-rank 1 formal foliation on $\widehat{X}$
and suppose that that $D$ is tangent to $\cal F$.

Suppose that the following hold:
\begin{enumerate}
\item for all $p \in D$ there exists an open neighborhood
$p \in U_p \subset X$ and $F_p \in H^0(\widehat{U_p}, \cal O_{\widehat{U_p}})$ with $F_p$ a first integral of $\cal F$ and where $\widehat{U_p}$
is the formal completion of $U_p$ along $D$;

\item for any $p, q \in D$ we have $\text{sing}(X)\cap U_p\cap U_q = \emptyset$; and

\item for any $p, q \in D$ if $U_p \cap U_q \neq \emptyset$ then $F_p\vert_{\widehat{U_p} \cap \widehat{U_q}}$ is not a power.
\end{enumerate}

Then we may produce a representation $\rho\colon \pi_1(D) \rightarrow \text{Aut}(\widehat{\Delta})$ such
that if this representation if trivial then $\cal F$ admits a first integral
$F \in H^0(\widehat{X}, \cal O_{\widehat{X}})$.  Moreover, if the $F_p$ can be taken to be convergent, then $F$ may be taken to 
be convergent as well.
\end{lem}
\begin{proof}
Without loss of generality we may assume that $F_p\vert_D = 0$ for all $p$.

If $p \neq q$ is such that $U_p \cap U_q \neq \emptyset$ choose some $z \in U_p\cap U_q \cap D$. 
By considering $F_p, F_q$ as elements in the completion $\widehat{\cal O_{\widehat{X}, z}}$ we may apply 
\cite[Th\'{e}orem\`{e} de factorisation]{MM80}
to find a $\phi_{p, q} \in \text{Aut}(\widehat{\Delta})$ such that $F_p = \phi_{p, q} \circ F_q$.

We may then produce a representation of $\pi_1(D)$ 
along the same lines of the classical holonomy representation, 
see for instance \cite[Chapter IV]{CN85}.
Let $\gamma \in \pi_1(D)$ be a path 
$\gamma\colon [0, 1] \rightarrow D$.  We may find a collection of points $p_0, ..., p_{n-1}, p_n = p_0 \in D$ and a partition 
$0 = t_0 <t_1 <... <t_n = 1$ of $[0, 1]$
so that $\gamma([t_{i-1}, t_i]) \subset U_{p_{i-1}}$.  We may then define our representation
by setting $\rho(\gamma) = \phi_{p_n, p_{n-1}}\circ ... \circ \phi_{p_1, p_0}$.

If $\rho(\gamma) = 1$, then for $1 \leq j\leq n-1$ we may replace 
$F_{p_j}$ by $(\phi_{p_j, p_{j-1}}\circ... \circ \phi_{p_1, p_0})^{-1}\circ F_{p_j}$ and so we may freely assume that $F_{p_j} = F_{p_0}$ 
for all $j$.  Thus, if the image of $\rho$ is $\{1\}$ it follows that all the $F_p$ glue to give a section $F\in H^0(\widehat{X}, \cal O_{\widehat{X}})$.

Our claim about convergence follows by observing that if the $F_p$ are all convergent then the $\phi_{p, q}$ may be taken
to be convergent as well.
\end{proof}

\subsection{Proof of Theorem \ref{t_Malgrange}}

\begin{proof}[Proof of Theorem \ref{t_Malgrange}]
First, by Lemma \ref{l_canimpliesklt} we know that $X$ is klt.

Next, by Lemma \ref{l:q_fact} we may replace $X$ by a quasi-\'etale cyclic cover
and so may assume that $N^{[*]}_{\cal F}$ is Cartier and so $\cal F_{|X\setminus P}$ is defined by an integrable
1-form $\omega$ which is non-vanishing on $X\setminus P$.
By Corollary \ref{c:cover_hgvs} we have that that 
$\omega$ admits a holomorphic Godbillon-Vey sequence
$(\omega_k)$.

Let $L_X$ be the link of $X$.  By \cite[Corollary 1.4]{TX17} we know that $\pi_1(L_X)$ is finite and let
$\widetilde{L} \rightarrow L_X$ be the universal cover. We may find a Galois \'etale
morphism of complex spaces $Y' \rightarrow X\setminus P$ corresponding to this cover
and by \cite[Proposition 3.13]{GKP16} this cover extends to a Galois quasi-\'etale cover
$\pi \colon Y \rightarrow X$. So by replacing $X$ by $Y$
we may assume that $\pi_1(L_X) = \{1\}$.

Let $\mu:Y \rightarrow X$ be a log resolution of $X$ and let $E = \sum E_k$ be the sum of the $\mu$-exceptional divisors.
Let $Y^* := Y \setminus \mu^{-1}(P) \cong X\setminus P$.
By \cite[Theorem 4.3]{GKKP11} we see that $\omega_i\vert_{X\setminus P}$ extends to
a holomorphic 1-form $\widetilde{\omega}_i$ on $Y$.  

There exist maps
\[\pi_1(L_X) \cong \pi_1(Y^*) \xrightarrow{a} \pi_1(Y) \xrightarrow{b} \pi_1(E)\]
where $a$ is a surjective and $b$ is an isomorphism, since $Y$ deformation retracts onto $E$.  
This implies that $\pi_1(E)$ is trivial. 

Define
\[\Omega = dt + \sum_{k = 0}^\infty \frac{t^k}{k!} \widetilde{\omega}_k \]
and recall that by definition $\Omega$ is an non-singular integrable 1-form defined on $\widehat{Y \times \bb C}$,
the completion of $Y \times \bb C$ along $E\times 0$, and where $t$ is a local coordinate on $\bb C$. 
We then have that $\Omega$ defines a smooth foliation $\widehat{\cal G}$ on $\widehat{Y \times \bb C}$.
Observe that by construction $\widehat{\cal G}\vert_{\widehat{Y}\times 0} = \mu^{-1}\cal F\vert_{\widehat{Y}}$ where $\widehat{Y}$ 
is the formal completion of $Y$ along $E$.

Since $\cal F$ has non-dicritical singularities $E$ is $\mu^{-1}\cal F$ invariant which implies 
that $E\times 0$
is tangent to $\widehat{\cal G}$.

By Corollary \ref{c_pullbackvanishing} we have $\widetilde{\omega}_k$ vanishes when restricted to $E$.
Thus we may apply Lemma \ref{l_formal_frobenius} to $\Omega$ to find for all $p \in E$ a neighborhood $p \in U_p \subset Y$
and first integral of $\Omega$ denoted $F_p \in H^0(\widehat{U_p \times \bb C}, \cal O_{\widehat{U_p \times \bb C}})$
where $\widehat{U_p \times \bb C}$ is the completion of $U_p \times \bb C$ along $E\times 0$.
Since $\widehat{\cal G}$ is smooth without loss of generality we may assume that $F_p$ is not a power
on $\widehat{U_p \times \bb C} \cap \widehat{U_q \times \bb C}$ for any $p, q$.

We may therefore
apply Lemma \ref{l:formal_holonomy} 
and since $\pi_1(E) = \{1\}$ we
produce a formal first integral $\hat{F} \in H^0(\widehat{Y\times \bb C}, \cal O_{\widehat{Y\times C}})$.
Restricting $\hat{F}$ to $\widehat{Y} \times 0$ we see that $\widetilde{\omega}_0$ admits a first integral  
$\hat{f} \in H^0(\widehat{Y}, \cal O_{\widehat{Y}})$.
We now show that we can take this first integral to be convergent.

Write $\hat{f}^*0 = \sum a_iE_i$.  By \cite[Chapter II]{KKMS73} we may find dominant proper generically 
finite morphism $W \xrightarrow{\sigma} Y$
such that the central fibre of $(\hat{f} \circ \sigma)$ is reduced and $\sigma$ is ramified only
over foliation invariant divisors. 
Write $\widetilde{E} = \sigma^{-1}(E)$, $\widehat{W}$ the completion of $W$ along
$\widetilde{E}$
and $\tilde{f} = \hat{f}\circ \sigma$.  

From the above construction we see that we may write 
$\tilde{f} = \bar{f}^r$ such that for all $p \in \widetilde{E}$ we have $\bar{f}$
is not a power in $\cal O_{\widehat{W}, p}$.
Thus we may apply \cite[Theor\'em\`e A]{MM80} to find a $\phi_p \in \text{Aut}(\widehat{\Delta})$
so that $\phi_p \circ \bar{f}$ is convergent on a neighborhood $U_p$ of $p$. 
We may apply Lemma \ref{l:formal_holonomy} by taking $F_p = \phi_p \circ \bar{f}$ 
to produce a representation
$\rho\colon \pi_1(W) \rightarrow \text{Aut}(\widehat{\Delta})$ 
which vanishes when $\sigma^{-1}\mu^{-1}\cal F$ admits a convergent first integral.

By taking the Stein factorization of $W \rightarrow X$ we produce a birational morphism $W \rightarrow X'$ contracting $\widetilde{E}$ to a point, 
and so that $r\colon X' \rightarrow X$ is branched only over the separatrices of $\cal F$.
We claim that $X'$ is klt.  Indeed, 
we see that $K_{r^{-1}\cal F} = r^*K_{\cal F}$ and so $r^{-1}\cal F$ has canonical singularities.
Let $S$ be a separatrix of $r^{-1}\cal F$ at $r^{-1}(P)$ 
(which exists since 
we know that $r^{-1}\cal F$ admits a formal first integral).
By \cite[Lemma 3.16]{CS18} we know that $(X', S)$ is log canonical
and since $S$ is $\bb Q$-Cartier it follows that $X'$ is in fact klt.

Thus, perhaps passing to a higher quasi-\'etale cover we may freely assume that $\pi_1(W) = 0$.
 Thus $\sigma^{-1}\mu^{-1}\cal F$ admits a convergent first integral.  By Lemma \ref{l:replace_cover} this implies that
$\mu^{-1}\cal F$, and hence $\cal F$, admits a convergent first integral.
\end{proof}

\subsection{Classification of terminal foliation singularities}

We will need the following which is a direct generalization of \cite[Lemma 9.7]{Spicer17}

\begin{cor}
\label{c_term_malg}
Let $P \in X$ be a normal threefold germ and let $\cal F$ be a terminal co-rank 1 foliation.
Then $\cal F$ admits a holomorphic first integral.  In particular $K_X$ is $\bb Q$-Cartier.
\end{cor}
\begin{remark}
A priori we only know that $K_{\cal F}$ is $\bb Q$-Cartier.
\end{remark}
\begin{proof}
After replacing $P\in X$ by a finite cover we may assume that $K_{\cal F}$ is Cartier.  Since $\cal F$ is terminal and $K_{\cal F}$ is Cartier
this implies that $P \in X$ is in fact an isolated singularity. 
Moreover, perhaps shrinking about $P$ we may assume that $Cl(P \in X)$
is generated by the classes of divisors $D_1, ..., D_N$ on $X$.

By Theorem \ref{t_existencefdlt}
we may take
\[\mu \colon (Y, \cal G) \rightarrow (X, \cal F)\] an F-dlt modification of $\cal F$.  Since $\cal F$ is terminal
we see that $\mu$ is small, i.e., $\mu^{-1}(P)$ is a union of curves.  Observe that $Y$ is $\bb Q$-factorial.
In particular, $D_i' := \mu_*^{-1}D_i$ is $\bb Q$-Cartier
and so if $P \in U \subset X$ is a smaller germ then 
$\mu\colon \mu^{-1}(U) \rightarrow U$ is also an F-dlt modification of
$\cal F\vert_U$. Indeed, to see this it suffices to show that $\mu^{-1}(U)$ is globally $\bb Q$-factorial.  If $D$ is any global divisor on $U$ then observe
that $\mu_*D \sim \sum a_iD_i$ by assumption and so $D \sim \sum a_iD'_i$ and hence
is $\bb Q$-Cartier.  
Thus we may freely replace $X$ by a smaller germ about $P$
at any point should we need to do so. 

We claim
\begin{claim}
\label{anqfact} 
For all $Q \in \mu^{-1}(P) \subset Y$ we have that $Y$ is analytically $\bb Q$-factorial about $Q$.
\end{claim}

\begin{claim}
\label{simplyconn}
 $Y$ is simply connected.
\end{claim}

\begin{proof}[Proof of Claim \ref{simplyconn}]
Let $T$ be a germ of a $\cal G$-invariant surface containing $\mu^{-1}(P)$.  Since $\cal G$
is terminal and $\mu^{-1}(P)$ is connected we see that $T$ is irreducible. Let $S = \mu_*T$ and observe by the proper mapping theorem
that $S$ is a divisor on $X$.

Observe that since $\cal F$, and hence $\cal G$, is terminal and Gorenstein (i.e., $K_{\cal F}$ is Cartier) 
we have that $(X, \cal F)$ and $(Y, \cal G)$ are both smooth
in codimension $2$ and so $K_{\cal G}\vert_T = K_T$ and $K_{\cal F}\vert_S = K_S$ and so we see that $\mu^*K_S = K_T$.
By \cite[Lemma 3.16]{CS18} we see that $T$ is a log terminal surface and so $P \in S$ is a germ of a log terminal singularity.
Thus we see that $\text{exc}(T \rightarrow S) = \text{exc}(\mu)$ is a tree of rational curves and therefore $\mu^{-1}(P)$ is simply connected.
Notice that $Y$ deformation retracts onto $\mu^{-1}(P)$ and so $Y$ is simply connected.
\end{proof}

Assuming Claim \ref{anqfact} we complete the proof as follows. Observe that $\mu^{-1}\cal F$ is terminal and so for all $q \in \mu^{-1}(P)$
by Theorem \ref{t_Malgrange} there exists a holomorphic first integral $F_q$ defined on a neighborhood $U_q$ of $q$ so 
that $F_q\vert_T = 0$.  

Let $s\colon Y' \rightarrow Y$ be an index 1 cover associated to $T$ ramified only over $T$, see \cite[Definition 2.52, Lemma 2.53]{KM98},
and let $\mu'\colon Y' \rightarrow X'$ be the Stein factorization of $Y' \rightarrow X$.
Notice that $r\colon X' \rightarrow X$ is ramified only along invariant divisors so
$K_{r^{-1}\cal F} = r^*K_{\cal F}$, in particular $r^{-1}\cal F$ is still terminal.
Replacing $X$ by $X'$ we may freely assume that $T$ is Cartier.

In particular, for any $q$, up to taking a root, we may assume that $(F_q = 0) = T\cap U_q$, i.e., $(F_q = 0)$ is reduced.
Thus for any $q$ and $q'$ so that $U_q \cap U_{q'} \neq \emptyset$, we see that $F_{q}$ is not a power
on $U_q \cap U_{q'}$.  Moreover, since $Y$ is smooth in codimension $2$ we see that $\mu^{-1}(P) \cap \text{sing}(X)$ consists
of a finite collection of points, and so by shrinking the $U_q$ if necessary we may also assume that 
$U_q \cap U_{q'} \cap \text{sing}(X) = \emptyset$.
We may then apply
Lemma \ref{l:formal_holonomy} 
to produce representation
$\rho \colon \pi_1(Y) \rightarrow \text{Aut}(\widehat{\Delta})$.  
Since $\pi_1(Y)$
is trivial we see that $\rho$ is trivial and so 
we get global first integral on $Y$,
which descends to $X$.

To show that $K_X$ is $\bb Q$-Cartier, let $\phi \colon (P \in X) \rightarrow (0 \in \bb C)$ be a holomorphic first integral for $\cal F$
where $0 \in \bb C$ is a (germ of a) curve.  Let $F = \phi^{-1}(0)$ and observe that $K_{\cal F} = K_{X/\bb C}(-mF)$ 
where $K_{X/\bb C} = K_X-\phi^*K_{\bb C}$
and where $m+1$ is the multiplicity of the fibre over $0$.  By assumption $K_{\cal F}$ is $\bb Q$-Cartier, 
$\phi^*K_{\bb C}$ is Cartier 
since $\bb C$ 
is a smooth curve and $F = \frac{1}{m+1}\phi^*0$ is $\bb Q$-Cartier and so $K_X$ is $\bb Q$-Cartier as claimed, thus completing the proof.

We now prove Claim \ref{anqfact},
\begin{proof}[Proof of Claim \ref{anqfact}]
Let $Q \in \mu^{-1}(P) \subset Y$ be any point.  We make the following preliminary observation. Let $D$ be any divisor defined in
an (analytic) neighborhood $U$ of $Q$ and suppose that $D \cap \mu^{-1}(P) = Q$.  Then, perhaps shrinking $X$ to a smaller neighborhood 
of $P$, we may extend $D$ to a divisor on all of $Y$.  Indeed,
for any $Q' \in \mu^{-1}(P) \setminus (\mu^{-1}(P) \cap U)$ we may find an open set $V_{Q'} \subset Y$ such
that $V_{Q'} \cap D = \emptyset$.  By compactness of $\mu^{-1}(P) \setminus (\mu^{-1}(P) \cap U)$
we may find $Q_1, ..., Q_n$ such that $\mu^{-1}(P) \subset U' := U \cup V_{Q_1} \cup \cdots \cup V_{Q_n}$.
By construction we see that $D$ is an analytic divisor defined on all of $U'$, by setting $D \cap V_{Q_i} = \emptyset$.
We may then find an open subset $W$ of $P$ in $X$ such that $\mu^{-1}(W) \subset U'$.  Replacing $X$ by $W$ 
we see that our observation follows.

So, suppose that $Q \in \mu^{-1}(P)$ and and suppose for sake of contradiction
that $Y$ is not analytically $\bb Q$-factorial about $Q$
 and let $D$ be a local divisor defined on a neighborhood $V$ of $Q$
which is not $\bb Q$-Cartier.  A priori, it is possible $D \cap \mu^{-1}(P)$ is $1$-dimensional
and so it is not clear if we can extend $D$ to a divisor on all of $Y$.

Since $Y$ is klt this implies there exists
a small $\bb Q$-factorialization about $Q$. Let $f\colon Z \rightarrow (Q \in Y)$ be this $\bb Q$-factorialization
and let $D'$ be the strict transform of $D$ and let $f^{-1}(Q) = \bigcup_i C_i$ be a decomposition into irreducible
components.  

Observe that for all $i$ we may find an irreducible effective Cartier 
divisor $S_i$ defined on $Z$ such that $S_i\cdot C_j = \delta_{ij}$
and such that $S_i \cap f^{-1}(Q)$ is a single point.

By choosing $a_i \in \bb Q$ appropriately we may assume that $D'+\sum a_iS_i$ is numerically trivial over $Y$.
Since $f$ is small we see that $(D'+\sum a_iS_i)-K_Z$ is nef and big over $Y$ and therefore by the relative basepoint free theorem,
\cite[Theorem 3.24]{KM98}
for $n >0$ sufficiently divisible we have that $n(D'+\sum a_iS_i) \sim_f 0$.  In particular, if we 
let $T_i = f_*S_i$ we see that 
$D+\sum a_iT_i$ is $\bb Q$-Cartier near $Q$.

Since $T_i \cap \mu^{-1}(P)$ is a point, by our observation at the beginning of this proof we may extend $T_i$ to a divisor
on all of $Y$, in particular, it follows that $T_i$ is $\bb Q$-Cartier.
This in turn implies that $D$ is in fact $\bb Q$-Cartier, proving our claim.
\end{proof}

\end{proof}

We can now provide a classification of terminal foliation singularities.

\begin{prop}
\label{prop_cart_case}
Let $(P \in X)$ be a normal threefold germ and let $\cal F$ be a co-rank one foliation on $(P \in X)$.
Suppose $K_X$ and $K_{\cal F}$ are Cartier and suppose that $\cal F$ is terminal.
Then $\cal F$ is given by the smoothing of a Du Val surface singularity, i.e., $\cal F$
admits a first integral $\phi\colon (P\in X) \rightarrow (0 \in \bb C)$ where $\phi^{-1}(0)$ is a Du Val surface singularity
and $\phi^{-1}(t)$ is smooth for $t \neq 0$.  In particular, $X$ is terminal.

Moreover, it is possible to write down a list of all such smoothings.
In an appropriate choice of coordinates 
we have that \[X = \{\psi(x, y, z)+tg(x, y, z, t)=0\}\]
and that $\cal F$ is defined by the 1-form $dt$, i.e., our first integral is just $(x, y, z, t)\mapsto t$
and where $\psi(x, y, z)$ is one of the following, see \cite[Theorem 4.20]{KM98}:
\begin{enumerate}

\item \label{i_an} $x^2+y^2+z^{n+1}$ with $n \geq 0$;

\item $x^2 +zy^2+z^{n-1}$ with $n \geq 4$;

\item $x^2+y^3+z^4$;

\item $x^2+y^3+yz^3$;

\item \label{i_e8}  $x^2+y^3+z^5$;

\item \label{i_smooth} $x$.

\end{enumerate}
Conversely, if
$g(x, y, z, t)$ is such that $X$ has at worst an isolated singularity at $P$ and $\cal F$ is defined by $dt$
then $\cal F$ has a terminal singularity at $P$.
\end{prop}

\begin{thm}
\label{t_3fold_term}
Let $P \in X$ be a threefold germ and let $\cal F$ be a co-rank one foliation on $X$ and suppose that $\cal F$ is terminal.
Then $P \in X$ is a quotient of one of the foliations \ref{i_an}-\ref{i_smooth} in the above list by $G = \bb Z/m \times \bb Z/n$.  
\end{thm}

\begin{proof}
By Corollary \ref{c_term_malg} we see that $K_{\cal F}$ and $K_X$ are both $\bb Q$-Cartier so we may find a Galois cover 
$\pi\colon (X', \cal F') \rightarrow (X, \cal F)$ with Galois group $\bb Z/n \times \bb Z/m$ so that $K_{\cal F'}$ and $K_{X'}$ are both Cartier.
Indeed, Let $X_1 \rightarrow X$ and $X_2 \rightarrow X$ be the index one coves associated to 
$K_{\cal F}$ and $K_X$, with Galois groups $\bb Z/m_1$ and $\bb Z/m_2$ respectively.
Then if $X'$ is the normalization of a component of $X_1\times_X X_2$ dominating $X$ then $X' \rightarrow X$
is Galois and its Galois group is a subgroup of $\bb Z/m_1 \times \bb Z/m_2$ as required.
\newline
By Proposition \ref{prop_cart_case} we see that $(X', \cal F')$ is one of the foliations \ref{i_an}-\ref{i_smooth}
and we can conclude.
\end{proof}

\begin{cor}
Let $p \in X$ be a germ of a normal threefold and let $\cal F$ be a co-rank one foliation on $X$ and suppose that $\cal F$ is terminal.  Then $X$ and $\cal F$ admit a $\bb Q$-smoothing, i.e.,  there exists a family of foliated threefold germs $X_t$ and $\cal F_t$ such that $(X_0, \cal F_0) = (X, \cal F)$ and such that for $t \neq 0$ we have that $(X_t, \cal F_t)$ is a quotient of a smooth foliation on a smooth variety.
\end{cor}
\begin{proof}
This is a direct consequence of the classification in Proposition \ref{prop_cart_case}. Indeed, in each case we may explicitly
construct a smoothing of $X$ and $\cal F$ by perturbing the defining equations of $X$ and $\cal F$.
\end{proof}

\subsection{Structure of terminal flips}
\label{term.flip.sect}
We finish by providing a rough structural statement for terminal foliated flips.

\begin{thm}
\label{t_term_flip_struct}
Let $X$ be a $\bb Q$-factorial threefold and let $\cal F$ be a co-rank 1 foliation on $X$ with terminal singularities.
Let $\phi\colon X \rightarrow Z$ be a $K_{\cal F}$-flipping contraction and let $C = {\rm Exc}(\phi)$.
\newline
Then there exists an analytic open neighborhood $C \subset U$ and a holomorphic first integral $F\colon U \rightarrow \bb C$ of $\cal F$.  
\end{thm}
\begin{proof}
By Theorem \ref{t_vanishing} we have that $R^1f_*\cal O_X = 0$, and so $C$ is in fact a tree of rational curves, in particular it is simply connected.
For all $p \in C$ by Theorem \ref{t_Malgrange}, we may find a holomorphic first integral of $\cal F$ near $p$.  
However, since $C$ is simply connected by arguing as in the proof of Corollary \ref{c_term_malg} we may produce a first integral in a neighborhood of $C$.
\end{proof}

\section{Existence of separatrices for log canonical foliation singularities}

The goal of this section is to prove the following.

\begin{thm}
\label{t_lc_sep}
Let $P \in X$ be an isolated klt singularity.
Let $\cal F$ be a germ of a log canonical co-rank 1 foliation singularity on $P \in X$.  Then $\cal F$ admits a separatrix.
\end{thm}

Recall that log canonical foliation singularities which are not canonical are always dicritical and in general  dicritical singularities do not admit separatrices as the following classical example due to Jouanolou shows.

\begin{example}
The foliation on $0 \in \bb C^3$ defined by \[(x^mz-y^{m+1})dx+(y^mx-z^{m+1})dy+(z^my-x^{m+1})dz\] 
has no separatrices at the origin for $m \geq 2$.  
The blow up of this foliation at $0$ has discrepancy $=-m$, and therefore is not log canonical
for $m \geq 2$.
\end{example} 

As the next example shows a log canonical singularity may not admit a separatrix if no assumption is made on the base space.

\begin{example}
Let $A$ be an abelian surface that admits an automorphism $\tau$ so that $X := A/\langle \tau \rangle$ is a rational surface
and $A \rightarrow X$ is \'etale in codimension 1.
We may find a linear foliation on $A$ which admits no algebraic leaves and is $\tau$-invariant
and so descends to a foliation $\cal F$ without algebraic leaves on $X$.

Let $P \in Y$ be the cone over $X$ with vertex $P$ and let $\cal G$ be the cone over $\cal F$.  It is easy to check
that $\cal G$ is log canonical and admits no separatrices at $P$.  However, observe that $P \in Y$ is log canonical and not klt.
\end{example}

We also have the following interesting corollary.
\begin{cor}
\label{transversal_sep}
Let $\cal F$ be a germ of a foliation $0 \in \bb C^3$ and let $i\colon (0 \in S) \rightarrow (0 \in \bb C^3)$ be a germ of a surface
transverse to $\cal F$ such that $i^{-1}\cal F$ is log canonical, e.g., is a radial singularity.
Then $\cal F$ admits a separatrix.
\end{cor}
\begin{proof}
This follows by combining Theorems \ref{t_lc_sep}
and \ref{t_inversion_of_adjunction} 
\end{proof}
We now proceed with the proof of Theorem \ref{t_lc_sep}.
We will first need the following generalization of Lemma \ref{l_formalseparatrix}.

\begin{lem}
\label{l_extension}
Let $X$ be a complex threefold with a co-rank 1 foliation $\cal F$ with non-dicritical singularities.
Let $D \subset X$ be a compact subvariety.
and let $V \subset D$ be a closed proper subvariety of $D$ tangent to $\cal F$ with the following property:
\[(\star) \text{  For all } p \in V \text{ if } S_p \text{ is a separatrix of } \cal F \text{ at } p  
\text{ then } S_p \cap D \subset V\]
Let $q \in V$ be any point, let $U_q$ be a neighborhood of $q$ 
and let $S_q \subset U_q$ be a separatrix at $q$.  
Then
there exists an analytic open neighborhood $U$ of $D$
and an invariant subvariety $S \subset U$ such that $S\cap U_q = S_q$.

\end{lem}
\begin{proof}
Let $\pi:\overline{X} \rightarrow X$ be a resolution of singularities of $X$ and $\cal F$
and so that $\pi^{-1}(V)$ is an invariant divisor.

Observe that Condition $(\star)$ still holds for $\pi^{-1}(V)$ and $\pi^{-1}(D)$.  
Moreover, if $q \in V$ is some point and $\pi^{-1}(S_q)$ admits an extension, $\overline{S}$,
to a neighborhood $\overline{U}$ of $\pi^{-1}(D)$ then since $\pi$ is proper,  $\pi(\overline{S}) \subset \pi(\overline{U})$
is an extension of $S_q$ to a neighborhood of $D$.

Thus, without
loss of generality we may assume that $X$ is smooth, $\cal F$ has simple singularities and
that $V$ is a divisor.

Let $q \in V$ be a point and let $S_q$ be any separatrix at $q$.  By Lemma \ref{l_formalseparatrix}
we may find a neighborhood $U'$ of $V$ and an invariant divisor $S'$ which agrees with $S_q$ near $q$.
Let $D' = D \cap U'$.
By $(\star)$ we see that $S' \cap (D' - V) = \emptyset$.  Thus, perhaps shrinking $U'$ if necessary, for all $p \in D - V$
there exists a neighborhood $U_p$ of $p$ such that $U_p \cap S' = \emptyset$. 

Taking $U = U' \cup \bigcup_{p \in D-V} U_p$ we see that $S'$ extends to a subvariety of $U$ and we are done.
\end{proof}

We recall the following classification result due to \cite{McQuillan08}.

\begin{thm}
\label{t_triv_surface_foliation} 
Let $X$ be a normal projective surface and let $\cal L$ be a rank one foliation on $X$  with canonical foliation singularities.  
Suppose $c_1(K_{\cal L}) = 0$.

Then there exists a birational morphism $\mu\colon X \rightarrow X'$ contracting only rational curves
tangent to $\cal L$ and a cyclic cover, $\tau \colon Y \rightarrow X'$, \'etale in codimension one such that
one of the following holds where $\cal G = \tau^{-1}\mu_*\cal L$:

\begin{enumerate}
\item \label{i_sesqui} $\mu$ is an isomorphism, $X = C \times E/G$ where $g(E) =1$, $C$ is a smooth projective curve, $G$ is a finite group
acting on $C \times E$ and $\cal G$ is the foliation induced by the $\GG$-invariant fibration $C\times E \rightarrow C$;

\item \label{i_kron} $\mu$ is an isomorphism and $\cal G$ is a linear foliation on the abelian surface $Y$;

\item \label{i_ricatti} $\mu$ is an isomorphism, $Y$ is a $\bb P^1$-bundle over an elliptic curve and $\cal G$ is transverse
to the bundle structure and leaves at least one section invariant; or

\item \label{i_qtoric} Up to blowing up $Y$ at $P \in \text{sing}(\cal L)$
we have $Y$ is a compactification of $\bb G_m \times \bb G_a$ and $\cal L$ restricted to this open subset 
is generated by a $\bb G_m \times \bb G_a$ invariant vector field; or

\item \label{i_toric} Up to blowing up $Y$ at $P \in \text{sing}(\cal L)$ we have that $Y$ is a compactification
of $\bb G_m \times \bb G_m$ and $\cal L$ restricted to this open subset is generated by a $\bb G_m \times \bb G_m$
invariant vector field.
\end{enumerate}
\end{thm}
\begin{proof}
This follows directly from \cite[Theorem IV.3.6]{McQuillan08} except for the claim in items \ref{i_sesqui} - \ref{i_ricatti}
that $\mu$ is an isomorphism.  
In each of these cases we claim that $\mu_*\cal L$ is terminal.  This follows because for all $P \in X'$ there exists
a cyclic cover (namely $\tau$) 
such that $\tau^{-1}\mu_*\cal L$ is smooth in a neighborhood of $\tau^{-1}(P)$ and so we may apply
Proposition \ref{prop_term_surf_class} to conclude.

Since $\mu_*\cal L$ is terminal and $c_1(K_{\cal L}) = 0$ this implies that $\mu$ is an isomorphism.
\end{proof}

\begin{lem}
\label{l_sep_kod_zero}
Let $S$ be a surface and let $\cal L$ be a co-rank 1 foliation on $S$.
Suppose that $c_1(K_{\cal L}) = 0$ and that $\cal L$ has canonical singularities. 

Then the following hold.
\begin{enumerate}
\item \label{i_each_sep_alg} For all $p \in \text{sing}(\cal L)$ each
separatrix at $p$ is algebraic.  In particular, the union of all such separatrices is an algebraic
subvariety of $S$.

\item \label{i_big_item} Either there exists a quasi-\'etale cover $\tau:A \rightarrow S$
where $A$ is an abelian variety, or there exists an algebraic curve $V \subset S$ such that each component
of $V$ is $\cal L$ invariant and if $p \in \text{sing}(\cal L)\cap V$ then each separatrix at $p$
is contained in $V$.
\end{enumerate}
\end{lem}
\begin{proof}
To prove item \ref{i_each_sep_alg} observe that
in order to check if each separatrix at a singular point is algebraic we may freely contract curves tangent to the foliation,
as well as replacing by a finite cover.  Thus, it suffices to check
the claim for each of the 5 types of foliation listed in the statement of Theorem \ref{t_triv_surface_foliation}.

In cases \ref{i_sesqui} - \ref{i_ricatti} the foliation is smooth and so there is nothing to prove. Thus it remains
to consider cases \ref{i_qtoric} and \ref{i_toric}.

In this case, we see that the vector field generating $\cal L$ on $\bb G_m \times \bb G_a$
or $\bb G_m \times \bb G_m$, respectively, is smooth.  Hence $\text{sing}(\cal L)$ is contained in the boundary
of the compactification.  Moreover, since $\cal L$ is invariant under the action of $\bb G_m \times \bb G_a$ or
$\bb G_m \times \bb G_m$ we see that every separatrix of $p \in \text{sing}(\cal L)$ must be contained 
in the boundary.

To prove item \ref{i_big_item} again we may freely contract curves tangent to $\cal L$ and replace by a finite cover.
Thus we may assume that $(S, \cal L)$ is one of the foliations listed in Theorem \ref{t_triv_surface_foliation}.
We argue based on the case.

If we are in case \ref{i_toric} or \ref{i_qtoric} then $\text{sing}(\cal L)$ is non-empty and so by item \ref{i_each_sep_alg} as proven 
above we may take $V$
to be the union of all separatrices at $\text{sing}(\cal L)$. 

If we are in case \ref{i_sesqui} then $\cal L$ is algebraically integrable and we may take $V$ to be the closure of a general leaf.

If we are in case \ref{i_ricatti} let $\Sigma$ be the invariant section.  We claim that $\cal L$ is smooth along $\Sigma$.  Indeed,
on one hand $K_{\cal L}\cdot \Sigma = K_\Sigma+\Delta$ where $\Delta \geq 0$ is supported on $\text{sing}(\cal L)\cap \Sigma$.
On the other hand by assumption $K_{\cal L}\cdot \Sigma = 0$ and since $\Sigma$ is an elliptic curve we have 
$K_\Sigma = 0$ and so $\Delta =0$. This gives us
$\text{sing}(\cal L)\cap \Sigma = \emptyset$ and so we may take $V = \Sigma$.

Otherwise $S$ is an abelian variety and there is nothing more to prove.
\end{proof}

\begin{lem}
\label{l_lc_modification}
Let $P \in X$ be a germ of a normal threefold and let $\cal F$ be a co-rank one foliation on $X$.
Suppose that $\cal F$ is log canonical but not canonical.
Then there exists a birational morphism $\pi\colon Y \rightarrow X$ 
and an irreducible $\pi$-exceptional divisor $E_0$ such that
\begin{enumerate}
\item \label{a_5} $E_0$ is a $\pi$-exceptional divisor transverse to $\cal G := \pi^{-1}\cal F$;

\item \label{a_6} $\pi^{-1}(P) \subset E_0$;

\item \label{a_4} $\cal G$ has non-dicritical singularities;

\item \label{a_1} $K_{\cal G}+E = \pi^*K_{\cal F}$ where $E = \sum_i \epsilon(E_i)E_i$ where we sum over all $\pi$-exceptional divisors ;

\item \label{a_2} $(\cal G, E)$ is log canonical and $(\cal G, (1-\epsilon)E)$ is F-dlt for all $1>\epsilon>0$; and

\item \label{a_3} $Y$ is $\bb Q$-factorial and klt.

\end{enumerate}
\end{lem}
\begin{proof}

Let $\mu\colon (\overline{X}, \overline{\cal F}) \rightarrow (X, \cal F)$ be an F-dlt modification of $(X, \cal F)$
and write $K_{\overline{\cal F}}+\sum \epsilon(E'_i)E'_i= \mu^*K_{\cal F}$
where the $E'_i$ are the $\mu$-exceptional divisors.
Observe that $\overline{X}$ is $\bb Q$-factorial and klt and $\overline{\cal F}$ has non-dicritical singularities.

Since $\cal F$ is not canonical it must be the case that $\mu$ extracts some divisor
transverse to the foliation. We may therefore assume, after relabeling, that $\epsilon(E'_0) = 1$
and $E'_0 \cap \mu^{-1}(P) \neq \emptyset$.
For $0 < \delta \ll 1$ we know that $(\overline{\cal F}, \sum \epsilon(E'_i)E'_i - \delta E'_0:=\Theta)$ is F-dlt
and so by Corollary \ref{cor_Rel_MMP} we may run a $K_{\overline{\cal F}}+ \sum \epsilon(E'_i)E'_i - \delta E'_0$-MMP over $X$, call this MMP 
$\phi \colon \overline{X} \dashrightarrow Y$ and 
let $\pi\colon (Y, \cal G) \rightarrow (X, \cal F)$ be the induced map.

Since the MMP preserves $\bb Q$-factoriality and klt singularities and the output of the MMP
has non-dicritical singularities we see that items \ref{a_3} and \ref{a_4} are satisfied.
Item \ref{a_1} follows by construction and item \ref{a_2} follows since the MMP preserves F-dlt singularities.

Since 
\[K_{\overline{\cal F}}+ \sum \epsilon(E'_i)E'_i - \delta E'_0 \equiv_\mu -\delta E'_0\]
we see that each ray $R$ contracted by this MMP has positive intersection
with the strict transform of $E'_0$, in particular $E'_0$ is not contracted by this MMP.
Set $E_0 = \phi_*E'_0$.  Since $E'_0$
is transverse to the foliation $E_0$ is as well proving item \ref{a_5}.  Moreover we have that  
\[(K_{\cal G}+\phi_*\Theta) - (K_{\cal G}+\phi_* \sum \epsilon(E'_i)E'_i) = -\delta\phi_*E'_0 = -\delta E_0\]
is nef over $X$.  
By the Negativity Lemma,~\cite[Lemma~1.3]{1907.06705}, for all $x \in X$ either $\pi^{-1}(x)$ is disjoint from $E_0$ or $\pi^{-1}(x)$ is contained in $E_0$.  
By our choice of $E_0$ we have $E_0 \cap \pi^{-1}(P) \neq \emptyset$ which proves item \ref{a_6}.
\end{proof}

\begin{lem}
\label{l_lc_sep}
Let $P \in X$ be a germ of a klt singularity 
with a co-rank one foliation $\cal F$ with log canonical but not canonical singularities.  
Let $\pi:(Y, \cal G) \rightarrow (X, \cal F)$ be a birational morphism as in Lemma \ref{l_lc_modification} above.

Suppose that $\text{dim}(\pi^{-1}(P)) = 2$ and that $\pi^{-1}(P)$ is the only $\pi$-exceptional divisor transverse to 
$\cal G:=\pi^{-1}\cal F$.
Then there is a separatrix at $P$.
\end{lem}
\begin{proof}
Let $E_0$ be a divisor as in Lemma \ref{l_lc_modification}
containing $\pi^{-1}(P)$.  Since $E_0$ is irreducible this implies that $\pi^{-1}(P) = E_0$.

We will find a closed subset $V \subset E_0$ satisfying
the hypotheses of Lemma \ref{l_extension} in order to produce a $\cal G$-invariant divisor 
in a neighborhood of $E_0$ whose pushforward will be our desired separatrix.

Let $\{E_i\}$ denote the collection of $\pi$-exceptional divisors
so that we have $K_{\cal G} +E_0 = \pi^*K_{\cal F}$ and $(\cal G, E_0)$ is log canonical and where
$E_i$ is $\cal G$-invariant for $i \neq 0$.
Note that since $E_i$ is invariant we see that $C_i := \pi(E_i)$ 
is a curve tangent to $\cal F$ passing through $P$.

Since $\cal G$ is non-dicritical and $E_0$ is the only $\pi$-exceptional divisor which is not $\cal G$-invariant
it follows that $\cal F$ restricted to $X \setminus P$ is non-dicritical.

By foliation adjunction, Lemma \ref{adjunction}, we know
that 
\[0 \sim_{\bb Q} n^*(K_{\cal G}+E_0) = K_{\cal G_0}+\Delta_0\]
where $\Delta_0 \geq 0$ and where $n\colon E^n_0 \rightarrow E_0$ is the normalization.

Next, since $X$ is klt 
we may write $K_{Y}+E_0+B = \pi^*K_X+aE_0$ where $a> 0$ and $B$ is not necessarily effective, but is supported
on the $\cal G$-invariant $\pi$-exceptional divisors.  Write $n^*(K_Y+E_0+B) = K_{E^n_0}+\Theta_0$.

We claim that $-E_0\vert_{E_0}$ is big.  
Let $A$ be an ample divisor on $Y$.  We may find a divisor $D \geq 0$ on $X$ so that $D+\pi_*A$ is $\bb Q$-Cartier.
We may then write $\pi^*(D+\pi_*A) = tE_0+A+D'$ where $t>0$, $D' \geq 0$, $\pi_*D' = D$ and $E_0$
is not contained in the support of $D'$.  Since $\pi^*(D+\pi_*A)\vert_{E_0} \sim 0$ we have
that $-E_0\vert_{E_0}\sim \frac{1}{t}(A+D')\vert_{E_0}$ which is big, as required.
It then follows that $-(K_{E_0}+\Theta_0)$ is big.  
Observe that $\Theta_0$ is not necessarily effective, but if we write $\Theta_0 = \Theta^+_0-\Theta^{-}_0$ where $\Theta^+_0, \Theta^-_0 \geq 0$ then
$\Theta^-_0$ is $\cal G_0$-invariant since it is supported on $n^{-1}(B \cap E_0)$.

First we handle the case $\Delta_0 \neq 0$.  In this case $K_{\cal G_0}$ is not psef, hence $\cal G_0$ is
algebraically integrable, by \cite[Main Theorem]{BMc01}.
Take $V$ to be the closure of general leaf of $\cal G_0$.  Observe
that $\cal G_0$ is non-dicritical since $\pi^{-1}\cal F$ is, and so $V$ is disjoint from the closure
of any other leaf of $\cal G_0$.  Moreover, in this case we see that $E^n_0$ is a $\bb P^1$-fibration
over a curve and that $V$ is a general fibre in this fibration.  In particular, notice that $K_{E^n_0}\cdot V = -2$.

We claim that $n^{-1}(n(V)) = V$.  Indeed, if not then $E_0$ would not be normal in a neighborhood of some point
of $n(V)$.  Let $W \subset \text{sing}(E_0)$ be a one dimensional component meeting $n(V)$.  Observe that since $V$
is general that $W$ is transverse to the foliation.
Since $(\cal G, (1-\epsilon)E_0)$ is F-dlt for all $1>\epsilon>0$ it follows from 
\cite[Lemma 3.11]{Spicer17} that $(Y, (1-\epsilon)E_0)$ is dlt at the generic point of $W$.
It follows by \cite[Corollary 5.55]{KM98}
in a neighborhood of a general point of $W$ we have that $E_0$ consists of two smooth components
meeting transversely.

Since $V$ is general, it follows that in an (analytic) neighborhood of $V$ that $n^{-1}(W)$ consists of two components
transverse to $\cal G_0$.
A straightforward calculation shows that the coefficient of each of these components in $\Delta_0$ and $\Theta_0$ is $=1$.
Notice that $V\cdot \Theta^-_0 = 0$ and so $(K_{E^n_0}+\Theta_0)\cdot V \geq 0$.  
However, $V$ is a movable curve and this contradicts
the fact that $-(K_{E^n_0}+\Theta_0)$ is big. 

Thus for all $q \in n(V)$ if $S_q$ is a separatrix of $\pi^{-1}\cal F$ at $q$ we see that $S_q \cap E_0 \subset n(V)$
and so we may apply Lemma \ref{l_extension} to produce an extension $T$ of $S_q$ to a neighborhood $U$ of $E_0$.  Perhaps
shrinking $U$ we may assume that $U = \pi^{-1}(W)$ for some neighborhood $W$ of $P$.  Notice also that since $V$
was chosen to be general we may assume that $T$ is not contained in the union of the $\pi$-exceptional divisors.  
Since $U \rightarrow W$
is proper we see that $S = \pi_*T \subset V$ is a divisor and is invariant
under $\cal F$, and hence is our desired separatrix.

Now we handle the case $\Delta_0 = 0$.  First observe that $\Delta_0 = 0$ implies that $E_0$ is normal.
By foliation adjunction, Lemma \ref{adjunction}, $\cal G_0$ is log canonical and since $\cal G_0$ is non-dicritical we see
that $\cal G_0$ is in fact canonical.

First, suppose that there exists a quasi-\'etale cover
$r:Y \rightarrow E_0$ be a quasi-\'etale cover of $E_0$ such that $Y$ is an abelian variety.
We claim that $\cal G_0$ is algebraically integrable in this case (in which case we are done by arguing as above).
If $\Theta_0^- \neq 0$ then $\cal G_0$, and hence $r^{-1}\cal G_0$, admits an invariant algebraic curve
and by Theorem \ref{t_triv_surface_foliation} we see
that $\cal G_0$ is algebraically integrable.
So suppose for sake of contradiction that $\Theta_0^- = 0$.  In this case,$-K_{E_0}$ is big and so we have that $-K_Y = -r^*K_{E_0}$ is big,
contradicting $K_Y \sim 0$.

Next, suppose that there is no such cover.  We 
may apply Lemma \ref{l_sep_kod_zero} to produce $V \subset E_0$ such that each component
of $V$ is tangent to $\cal G_0$ and each separatrix of $\cal G_0$ meeting $V$ is contained in $V$.
Thus, we may apply Lemma \ref{l_extension} to produce an invariant divisor $T$ in a neighborhood
of $E_0$ and which contains $V$.  We claim that $T$ is not contained in the union of the $\pi$-exceptional divisors.
Supposing the claim we see that $S = \pi_*T$ is our desired separatrix.

We now prove the claim.  First if $E_0$ is the only exceptional divisor there is nothing to show.  So suppose that there is some other $\pi$-exceptional divisor $E_i$.  Notice that if $Q \in C_i \setminus P$ is a general point then there exists
a separatrix of $\cal F$ at $Q$, call it $S_Q$ (recall $C_i = \pi(E_i)$ is tangent to $\cal F$).  This follows because in a neighborhood of $Q$ we know that
$X$ has quotient singularities, since $X$ is klt and klt singularities are quotient singularities outside a subset of codimension $\geq 3$.
Thus (up to replacing $X$ by a cover) we may assume that $X$ is smooth at $Q$.  Next, we know that $\cal F$ is non-dicritical in a 
neighborhood of $X$ and so we may apply \cite[Existence of Separatrix Theorem]{CC92} to produce $S_Q$.

 Let $S'_Q := \pi_*^{-1}S_Q$.    By Lemma \ref{l_formalseparatrix} we may extend $S'_Q$ to an invariant divisor 
 in an (analytic) open neighborhood of $\sum_{i =1}^mE_i$.  Call this extension $H$ and by construction
 $H$ is not contained in $\sum_{i = 0}^m E_i$.  Let $\Sigma = H \cap \sum_{i = 1}^m E_i$
 and notice that $\Sigma$ is a closed analytic subset of $\sum E_i$.  Let $\Sigma_0= \Sigma\cap E_0$ and let $x \in \Sigma_0$ 
 be a point.
 We know that $\sum_{i = 1}^m E_i \cap E_0 \subset V$ by construction.  
However, $H \cap E_0$ is a separatrix of $\cal G_0$ at $x$ intersecting $V$: in fact, $H\cap E_0 \subset V$.  
Then, in a neighborhood of 
 $E_0$, we have $H \subset T$, in particular, $T$ is not contained in the union of the $\pi$-exceptional divisors.
\end{proof}

We are now ready to prove the main theorem of this section.

\begin{proof}[Proof of Theorem \ref{t_lc_sep}]
Suppose first that $\cal F$ has canonical singularities. 
If $X$ is $\bb Q$-factorial then we may apply Corollary \ref{c_sep_can} to produce a separatrix.  
Otherwise, since $X$ is klt, it admits a small $\bb Q$-factorialization $\mu\colon X' \rightarrow X$.  
Since $\cal F$ is non-dicritical we know that $\mu^{-1}(P)$ is tangent to the foliation and is therefore contained in a germ of an invariant surface, $S$.  
We may then take $\mu_*S$ as our desired separatrix.

So we may assume that $\cal F$ is not canonical and let $\pi\colon (Y, \cal G) \rightarrow (X, \cal F)$ be a modification as in Lemma \ref{l_lc_modification} and let $E_0$ be a divisor as in the statement of the Lemma.

There are two cases, either $\pi^{-1}(P)$ is of dimension 2 or it is of dimension 1.
Notice moreover, that if there exists some $\pi$-exceptional divisor $E$ transverse to $\cal G$  such that $E$ is centred over a curve in $X$ then by choosing $E = E_0$ in the proof of Lemma \ref{l_lc_modification} we have that $\pi^{-1}(P)$ is of dimension 1.

If $\pi^{-1}(P)$ is of dimension 2 we may therefore freely assume that the only $\pi$-exceptional divisor transverse to $\cal G$ is $\pi^{-1}(P)$.  We may apply Lemma \ref{l_lc_sep} to conclude.

Otherwise $C:=\pi^{-1}(P) \subset E_0$ is a curve. Let $\cal G_0$ be the induced foliation on $E_0$.
Suppose first that some component $C_0 \subset C$ is transverse to $\cal G_0$.  
Then we may apply Lemma \ref{l_extension} with $D = \pi^{-1}(P)$ and $V$ a general point in $C_0$ to produce an invariant divisor $S$
in a neighborhood of $\pi^{-1}(P)$. 
In this case $\pi_*S$ will be our desired separatrix.

Now suppose that each component of $C$ is invariant by $\cal G_0$.
In this case, perhaps shrinking $X$, we may assume that the union of all convergent separatrices meeting $C$ is an analytic subset of $E_0$, call it $\widetilde{C}$.  
In this case we may apply Lemma \ref{l_extension} with $D = E_0$ and $V = \widetilde{C}$ to produce a separatrix $S$ in a neighborhood of $E_0$.   
Again, $\pi_*S$ is our desired separatrix.
\end{proof}

\begin{remark}
In fact, the arguments above prove a slightly stronger claim which may be of interest.
In the set up as above, if we let $C \subset \text{sing}(\cal F)$ be a curve of singularities passing through $P$ then $C$ is contained in a separatrix.
\end{remark}

\section{Foliations and hyperbolicity}
\label{hyperb.sect}

The goal of this section is to prove the following foliated version of \cite[Theorem 1.1]{Svaldi14}.
Given a foliated pair $(\FF, \Delta)$ and an lc center $S$ we will denote by $\bar{S} \subset S$ the locally closed subvariety obtained by removing from $S$ the lc centers of $(\FF, \Delta)$ strictly contained in $S$.

\begin{thm}\label{hyperb.thm}
Let $(\FF, \Delta)$ be a foliated log canonical pair on a normal projective variety $X$. 
Assume that 
  \begin{itemize}
   \item $X$ is potentially klt,
   \item there is no non-constant morphism $f : \mathbb{A}^1 \to X \setminus \nklt(\FF, \Delta)$ tangent to $\cal F$, and
   \item for any stratum $S$ of $\nklt(\FF, \Delta)$ there is no non-constant morphism $f \colon \mathbb{A}^1 \to \bar{S}$ which is tangent to $\FF$.
  \end{itemize}
Then $K_\FF + \Delta$ is nef.
\end{thm}

The notions of potentially kltness and potentially lcness have been defined in Definition~\ref{def:potent.klt}

\subsection{A special version of dlt modifications.}
We prove a refinement of Theorem \ref{t_existencefdlt}, which will be useful in the proof of the main result of this section.

\begin{thm}[Existence of special F-dlt modifications]
\label{t_sp_dlt_mod}
Let $\mathcal F$ be a co-rank one foliation on a 
normal projective variety $X$ of dimension at most $3$.
Let $(\mathcal F, \Delta= \sum a_i D_i)$ be a foliated pair.
We will denote by $\Delta':= \sum_{a_i < \epsilon(D_i)} a_i D_i + \sum_{a_j \geq \epsilon(D_j)} \epsilon(D_j)D_j$.
\newline
Then there exists a birational morphism $\pi\colon Y \rightarrow X$
which extracts divisors $E$ 
of foliation discrepancy $\leq -\epsilon(E)$
such that if we write $K_{\mathcal G}+\Gamma = \pi^*(K_{\mathcal F}+\Delta)$ then $(\mathcal G, \Gamma' :=\pi^{-1}_\ast \Delta' + \sum_{E_i \;  \pi-{\rm exc.}} \epsilon(E_i)E_i)$ is F-dlt.
\newline
Furthermore, we may choose $(Y, \mathcal G)$ so that 
\begin{enumerate}

\item if $W$ is a non-klt centre of 
$(\mathcal G, \Gamma)$ then $W$ is contained in a codimension one lc centre of $(\mathcal G, \Gamma')$,

\item $Y$ is $\mathbb Q$-factorial,

\item $Y$ is klt, and

\item $\pi^{-1}\nklt(\FF, \Delta)= \nklt(\GG, \Gamma) = \nklt(\GG, \Gamma')$.
\end{enumerate}
\end{thm}


\begin{proof}
For the proof of (1), (2), and (3) one can refer to \cite[Theorem 8.1]{CS18}.
Let $\pi_Z \colon Z \to X$ be a modification of $(\FF, \Delta)$ satisfying these three properties.
Let us denote by $(\mathcal{H}, \Theta, \Theta')$ the triple given by the birational transform of $\mathcal{F}$ on $Z$, 
\begin{eqnarray}
\nonumber 
K_\HH+\Theta= \pi^\ast_Z(K_\FF+\Delta), \;  \Theta'= \pi_{Z\ast}^{-1} \Delta'+ \sum_{F_i \; \pi_Z-{\rm exc.}} \epsilon(F_i)F_i.
\end{eqnarray}
By these inequalities $K_\HH +\Theta'^{<1} \sim_{\bb R, X} -\Theta''$, where $\Theta'':= \Theta - \Theta'^{< 1}$.
As $K_\HH +\Theta'^{<1}$ is big/$X$,
there exists $A$ ample/$X$ and an 
effective divisor $G$ such that
\[
 K_\HH +\Theta'^{<1} \sim_{\bb R, X} A+G.
\]
We can decompose $G$ as 
\[
 G=G_1+G_2+G_3,
\]
where $G_1$ is the part of $F$ supported on $\pi_Z$-exceptional divisors or $\HH$-invariant divisors, $G_2$ is the part of $G$ whose components are not $\HH$-invariant but contain an $\HH$-invariant lc center for $(\HH, \Theta)$, and $G_3:= G - G_1 - G_2$.
For any $0< \epsilon \ll 1$ we can write 
\begin{eqnarray}
\nonumber
& -\Theta'' \sim_{\mathbb{R}, X} K_\HH+\Theta'^{<1} = (1-\epsilon)(K_\HH+\Theta'^{<1}) + \epsilon(K_\HH+\Theta'^{<1}) \sim_{\bb R, X}
\\
\nonumber 
& \sim_{\bb R, X} (1-\epsilon)(K_\HH+\Theta'^{<1}+ \frac{\epsilon}{1-\epsilon} (A+G)) \sim_{\bb R, X} 
\\
\nonumber
& \sim_{\bb R, X}
(1-\epsilon)(K_\HH+\Theta'^{<1}+ \frac{\epsilon}{1-\epsilon} (A+G_1+G_2+G_3),
\end{eqnarray}
so that 
\begin{eqnarray}
\nonumber
 & K_\HH+\Theta'^{<1}  + \frac{\epsilon}{1-\epsilon} (A+G_2+G_3) 
 \sim_{\bb R, X} -\frac{1}{1-\epsilon}\Theta'' - \frac{\epsilon}{1-\epsilon}G_1
\end{eqnarray}
Choosing an effective divisor $L$ whose support coincides with the divisorial part of $\mathrm{exc}(\pi_Z)$ such that $A-L$ is ample, then
\[
 K_\HH+\Theta'^{<1}+ \epsilon' (G_2+G_3 +A - L) 
 \sim_{\bb R, X} -\frac{1}{1-\epsilon}\Theta'' - \frac{\epsilon}{1-\epsilon}(G_1 + L).
\]
Let us choose a sufficiently general effective $A' \sim_\mathbb{R} A-L$ and define $G' := G_2+ G_3 +A'$, $\epsilon':=\frac{\epsilon}{1-\epsilon}$, and $\Xi_{\epsilon'}:=\frac{1}{1-\epsilon}\Theta'' + \frac{\epsilon}{1-\epsilon}(G_1+L)$.
Hence, $K_\HH+\Theta'^{<1}+ \epsilon' G' \sim_{\bb R, X} -\Xi_{\epsilon'}$. 

\begin{claim} 
\label{claim.dlt.mod.proof}
For $\epsilon' \ll 1$, there exists an F-dlt modification $\bar r \colon \bar Z \to Z$ of $(\HH, \Theta'^{<1}+ \epsilon' G')$ such that for any $\bar r$-exceptional prime divisor $E$, $a(E; \HH , \Theta')= -\epsilon(E)$.
\end{claim} 

\begin{proof}[Proof of Claim~\ref{claim.dlt.mod.proof}]
Fix $0 < \epsilon' \ll 1$. Let $\tilde{r} \colon \tilde{Z} \to Z$ be an  F-dlt modification in the sense of Theorem \ref{t_existencefdlt} for $(\HH, \Theta'^{<1}+ \epsilon' G')$.
Let us define $\tilde{\HH}$ to be the birational transform of $\HH$ on $\tilde{Z}$.
Write 
\begin{eqnarray}
\nonumber
K_{\tilde{\HH}}+ \tilde{r}_\ast^{-1}(\Theta'^{<1}+ \epsilon' G')+ \sum a_i E_i = \tilde{r}^\ast(K_\HH+ \Theta'^{<1}+ \epsilon' G'), \; a_i \geq \epsilon(E_i).
\end{eqnarray}
As it is a F-dlt modification, it follows that $(\tilde{\HH}, \tilde{r}_\ast^{-1}(\Theta'^{<1}+ \epsilon' G')+ \sum \epsilon(E_i)E_i)$ is F-dlt.
Let $E_i$ be a $\tilde{r}$-exceptional prime divisor such that $a(E_i; \HH, \Theta') >- \epsilon(E_i)$; as $\Theta' \geq \Theta'^{<1}$
then also $a(E_i; \HH, \Theta'^{<1}) >- \epsilon(E_i)$.
Since the discrepancy $a(E_i; \HH, \Theta'^{<1}+ \epsilon ' G')$ is a linear function of $\epsilon'$, we can choose $0 <\epsilon'' \ll \epsilon'$ such that 
\begin{eqnarray}
\nonumber
K_{\tilde{\HH}}+ \tilde{r}_\ast^{-1}(\Theta'^{<1}+ \epsilon'' G')+ \sum b_i E_i = \tilde{r}^\ast(K_\HH+ \Theta'^{<1}+ \epsilon'' G'),
\end{eqnarray}
and $b_i <\epsilon(E_i)$ whenever $a(E_i; \HH, \Theta') >- \epsilon(E_i)$.
Hence, 
\begin{eqnarray}
\nonumber
K_{\tilde{\HH}}+ \tilde{r}_\ast^{-1}(\Theta'^{<1}+ \epsilon'' G')+ \sum \epsilon(E_i)E_i \sim_{\mathbb{R}, Z} P - N,
\end{eqnarray}
where $P, N$ are effective $\tilde{r}$-exceptional divisors with disjoint supports, the support of $P$ contains all the $E_i$ with $a(E_i; \HH, \Theta') > -\epsilon(E_i)$.
The pair $(\tilde{\HH}, \tilde{r}_\ast^{-1}(\Theta'^{<1}+ \epsilon'' G')+ \sum \epsilon(E_i)E_i)$ is dlt. 
By Corollary \ref{cor_Rel_MMP} we may run the $(K_{\tilde{\HH}}+ \tilde{r}_\ast^{-1}(\Theta'^{<1}+ \epsilon'' G')+ \sum \epsilon(E_i) E_i)$-MMP over $Z$, to obtain a model
\[
\xymatrix{\tilde{Z} \ar@{-->}[rr] \ar[dr]^{\tilde{r}}& &\bar{Z} \ar[dl]_{\bar{r}}\\
& Z &}
\]
where, denoting with $\bar{\HH}$ the birational transform of $\HH$, $K_{\bar{\HH}}+ \bar{r}_\ast^{-1}(\Theta'^{<1}+ \epsilon'' G')+ \sum \epsilon(F_i)F_i$ is relatively nef,
where the $F_i$ are the strict transforms of the $\bar{r}$-exceptional divisors of .
The Negativity Lemma implies that 
\[
K_{\bar{\HH}}+ \bar{r}_\ast^{-1}(\Theta'^{<1}+ \epsilon'' G')+ \sum \epsilon(F_i) F_i \sim_{\mathbb{R}, Z} - \bar{N},
\]
where $\bar{N}$ is the strict transform of $N$ on $\bar{Z}$.
Thus, by construction, $\bar{Z}$ is the model that satisfies the statement of the claim for the chosen value of $\epsilon''$.
\end{proof}
Let us recall that on $Z$, $K_\HH+\Theta'^{<1}+ \epsilon' G' \sim_{\bb R, X} -\Xi_{\epsilon'}$.
Thus, on $\bar{Z}$ there exists an effective divisor $\bar F$ supported on the $F_i$ such that $\bar F \geq \sum \epsilon(F_i)F_i$ and
\[
\bar{r}^\ast(K_\HH+\Theta'^{<1}+ \epsilon' G')
= K_{\bar{\HH}}+ \bar{r}_\ast^{-1}(\Theta'^{<1}+ \epsilon' G') + \bar{F}
\sim_{\bb R, X} -\bar{r}^\ast \Xi_{\epsilon'}.
\]
Moreover, the support of $\bar{F} +\bar{r}^\ast \Xi_{\epsilon'}$ is the union of the divisorial part of the exceptional locus of the morphism $\bar{Z} \to X$ together with some $\bar{\cal H}$-invariant components and
\[
K_{\bar{\HH}} + \bar{\Theta}' = \bar{r}^\ast(K_\FF + \Theta'), \; \bar{\Theta}':=r_\ast^{-1}(\Theta') + \sum \epsilon(F_i)F_i.
\]
Running the $(K_{\bar{\HH}}+ \bar{r}_\ast^{-1}(\Theta'^{<1}+ \epsilon' G'))$-MMP over $X$ 
\[
\xymatrix{\bar{Z} \ar@{-->}[rr] \ar[dr]^{\bar{r}}& & Y_1 \ar[dl]_{\pi}\\
& X &
}
\] 
terminates with a model $\pi_1 \colon Y_1 \to X$ on which $-(\bar{F} +\bar{r}^\ast \Xi_{\epsilon'})$ is nef.
We denote by $\HH_{Y_1}, \Theta'_{Y_1}$ the strict transforms of $\bar{\HH}, \bar{\Theta}'$ on $Y_1$. 
To conclude the proof, we take an F-dlt modification $r_Y \colon Y \to Y_1$ of the pair $(\HH_{Y_1}, \Theta'_{Y_1})$.
The Negativity Lemma,~\cite[Lemma~1.3]{1907.06705}, and Claim~\ref{claim.dlt.mod.proof} imply that $Y$ is the desired model whose existence we claimed in the statement of the theorem.
\end{proof}

\subsection{Mori hyperbolicity and non-klt locus}

We recall the following hyperbolicity result for standard log pairs with dlt support which will be used throughout this section.
\begin{prop}
\label{mh.nonlc.dlt.prop}
\cite[Prop. 5.2]{Svaldi14}
Let $(X, \Delta = \sum_i b_i D_i \geq 0)$ be a normal, projective, $\mathbb{Q}$-factorial 
log pair s.t. $(X, \Delta' =\sum_{b_i < 1} b_i D_i + \sum_{b_i \geq 1} D_i)$ is 
dlt. Suppose that $K_X + \Delta$ is nef when restricted to 
${\rm Supp}(\sum_{b_i \geq 1} b_i D_i) = \nklt(\Delta')= \nklt(\Delta)$. Then, exactly one of the following two possibilities holds:
\begin{itemize}
    \item $K_X+ \Delta$ is nef, or 
    \item $X \setminus \nklt(\Delta)$ contains an algebraic curve whose normalization is $\mathbb{A}^1$.
\end{itemize}
\end{prop}

In the case of a general foliated log pair, using dlt modifications we get the following criterion, which will be fundamental in the proof of Theorem \ref{hyperb.thm}.

\begin{cor}
\label{fund.cor}
Let $X$ be a normal, projective, $\mathbb{Q}$-factorial threefold.
Let $(\FF, \Delta = \sum_i b_i D_i \geq 0)$ be a foliated log pair such that $(\FF, \Delta' =\sum_{i | b_i < \epsilon(D_i)} b_i D_i + \sum_{i | b_i \geq \epsilon(D_i)} \epsilon(D_i)D_i)$ is F-dlt. 
Assume that $X \setminus \nklt(\FF, \Delta)$ does not contain algebraic curves tangent to $\FF$ whose normalization is $\mathbb{A}^1$.
Then $K_\FF + \Delta$ is nef if and only if $K_\FF + \Delta$ is nef when restricted to $\nklt(\FF, \Delta)$.
\end{cor} 

\begin{proof}
If $K_\FF + \Delta$ is nef, then, a fortiori, it is nef when restricted to any subvariety of $X$.
\newline
We now assume that $K_\FF + \Delta$ is nef when restricted to $\nklt(\Delta)$.
As $(\FF, \Delta')$ is F-dlt, it follows that 
\[
\nklt(\FF, \Delta) = \bigcup_{\mu_{D_i}\Delta \geq \epsilon(D_i)} D_i =\nklt(\FF, \Delta'),
\]
by definition of $\Delta'$.
Now, let us suppose that $K_\FF+\Delta$ is not nef.
Then there exists a negative extremal ray $R \subset \overline{NE}(X)$.
Since $K_\FF+\Delta$ is nef when restricted along $\nklt(\FF, \Delta)$, it follows that $R \cdot D_i \geq 0$ for any $D_i$ with $\mu_{D_i}\Delta \geq \epsilon(D_i)$. 
Hence, $R$ is a negative extremal ray also for $K_\FF+\Delta'$.
As $(\FF, \Delta')$ is an F-dlt pair, it is non-dicritical by Theorem~\ref{t_canimpliesnondicritical}; in particular,~\cite[Lemma~3.30]{CS18} implies that any curve $C \subset X$ satisfying $[C] \in R$ is tangent to $\FF$. 
Moreover,~\cite[Theorem~6.7]{CS18} implies that there exists a contraction $\phi \colon X \to Y$ within the category of projective varieties which only contracts curves in $X$ whose numerical class belongs to $R$. 
In particular as $K_\FF+\Delta$ is nef along $\nklt(\FF, \Delta)$, it follows that each fiber of $\phi$ intersects $\nklt(\FF, \Delta)$ in at most finitely many points.
As $X$ is $\bQ$-factorial, it follows that each fiber of $\phi$ intersecting $\Nklt(\FF, \Delta)$ must have dimension at most $1$; 
otherwise, if $X_y$, $y \in Y$, were a 2-dimensional fiber, no component of $\Delta'$ could intersect $X_y$, as this intersection would contain a $(K_\FF+\Delta)$-negative curve contained in $\Nklt(\FF, \Delta)$, hence there would be a rational curve $C \subset X \setminus \Nklt(\FF, \Delta)$, thus leading to a contradiction.
\newline
Let $\Sigma \subset X$ be an irreducible curve contracted by $\phi$.  
We claim that $\Sigma$ is a rational curve.  
Indeed, $\Sigma$ is tangent to $\cal F$, thus we may find a germ of an invariant surface, call it $S$, containing $\Sigma$.
If $\Sigma \not \subset \text{sing}(\cal F)$, then $S$ is simply a leaf containing $\Sigma$, while if $\Sigma \subset \text{sing}(\cal F)$, then we may take $S$ to be a strong separatrix at a general point of $\Sigma$.
As $(\FF, \Delta')$ is F-dlt, then we can apply Lemma~\ref{adjunction} to write $\nu^*(K_{\cal F}+\Delta') = K_{S^{\nu}}+\Delta'_{S^{\nu}}$, where 
$\nu\colon S^\nu \rightarrow S$ is the normalization of $S$, and $(S^{\nu}, \Delta'_{S^\nu})$ is lc.  
Taking $T$ to be the normalization of  $\phi(S)$ then the strict transform of $\Sigma$ on $S^\nu$ is a $(K_{S^{\nu}}+\Delta'_{S^{\nu}})$-negative curve contracted by the morphism $S^\nu \rightarrow T$ and is therefore (by classical adjunction) necessarily a rational curve.
\newline
The $\mathbb{Q}$-factoriality of $X$ implies that we are in either of the following two cases:
\begin{enumerate}
\item[1)] $\phi$ is a Mori fibre space and all the fibres are one dimensional;

\item[2)] $\phi$ is birational and the exceptional locus intersects $\nklt(\Delta)$.
\end{enumerate}
We claim that in both cases $R^1\phi_\ast \mathcal{O}_X=0$. 
In fact, in case 1) as all fibers are rational curves, we have that $\phi$ must be a $K_X$-negative contraction, while in case 2) the conclusion can be reached by direct application of Theorem \ref{t_vanishing}.
At this point, Theorem \ref{conn.f-dlt.thm} implies that
$\nklt(\FF, \Delta)$ is connected in a neighborhood of every fibre of $\phi$.
In case 1), the generic fibre of $\mu$ is a smooth projective rational curve. Theorem \ref{conn.f-dlt.thm} implies that the generic fibre intersects $\nklt(\Delta)$ 
in at most one point. This concludes the proof in case 1). 
In case 2), the positive dimensional fibres are chains of rational curves and by the vanishing $R^1\phi_\ast \mathcal{O}_X=0$, the generic fibre has to be a tree of smooth rational curves. By Theorem \ref{conn.f-dlt.thm}, $\nklt(\FF, \Delta)$ intersects this chain in at most one point. 
In particular, there exists a complete rational curve $C$ such that 
$C \cap (X \setminus \nklt(\FF, \Delta)) = f(\mathbb{A}^1)$, 
where $f$ is a non-constant morphism, which provides the sought contradiction.
\end{proof}

\begin{proof}[Proof of Theorem \ref{hyperb.thm}]
We divide the proof into two distinct cases.

{\bf Case 1}: {\it We first prove the theorem under the assumption that $(\FF, \Delta)$ is F-dlt}.
\newline
If $K_\FF+\Delta$ is nef along $\nklt(\FF, \Delta)$ the conclusion follows from Corollary~\ref{fund.cor}.
Hence, we can assume that there exists a positive dimensional lc center $W$ for $(\FF, \Delta)$ and $K_\FF+\Delta$ is not nef along $W$.
By induction on the dimension, we can consider $W$ to be a minimal (with respect to inclusion) lc center satisfying such property, so that $(K_\FF+\Delta)\vert_W$ is nef when restricted to the lc centers of $(\FF, \Delta)$ strictly contained in $W$.
Clearly, $\dim W >0$ and~\cite[Theorem~4.5]{Spicer17} implies that if $\dim W =1$, then $W$ is tangent to $\FF$.
As $(\FF, \Delta)$ is F-dlt, it follows that either one of the following conditions hold:
\begin{itemize}
    \item[a)] $W$ is a component of $\Delta$ of coefficient $1$;
    \item[b)] $W$ is an invariant divisor; 
    \item[c)] $W \subset \Sing(\FF)$, $\dim W=1$, and $\FF$ is canonical along $W$ by \cite[Lemma 3.12]{CS18}; or
   \item[d)] $\dim W=1$ and it is tangent to $\cal F$, $W \not \subset \Sing(\FF)$, but $W \subset D$, where $D$ is a component of $\Delta$ with $\mu_D\Delta = 1$.
\end{itemize}

{\bf Case 1.a}.
If $W$ is a component of $\Delta$ of coefficient $1$, then we can apply the adjunction formula along the normalization $\nu \colon W^\nu \to W$:
\[
\nu^*(K_\FF+\Delta) = K_\GG + \Theta,
\]
where $\GG$ is the restriction of $\FF$ to $W^\nu$ and $\Theta$ is the different as defined in Lemma \ref{adjunction}.
The adjunction formula guarantees that $(\GG, \Theta)$ is F-dlt, see Lemma \ref{adjunction}, 
that $\nu^{-1}(Z) = \nklt(\GG, \Theta)$, where $Z$ is the union of all lc centers of $(\FF, \Delta)$ strictly contained in $W$.
This follows from \cite[Lemma 3.8]{CS18} as $(X, \Delta)$ is log smooth in a neighborhood of any lc center;
in particular, $W$ is normal at the general point of any codimension $2$ lc centers contained in it, thus,
\begin{equation}
\label{nklt.norm.eq}
    \nklt(\GG, \Theta) = \nu^{-1}(Z).
\end{equation}
Hence, the conclusion follows from the 2-dimensional case, that is, from  Proposition \ref{hyperb.surf.lc.prop}.
In fact, the proposition implies that there is a non-constant map $f \colon \mathbb{A}^1 \to W^\nu$ and $f(\mathbb{A}^1) \subset W^\nu \setminus \nklt(\GG',\Theta)$ and by \eqref{nklt.norm.eq} composing with $\nu$ we obtain a map $f' \colon \mathbb{A}^1 \to (W \setminus Z)$.

{\bf Case 1.b}.
If $W$ is an invariant divisor, then we can apply the adjunction formula along the normalization $\nu \colon W^\nu \to W$:
\begin{equation}
\label{fol.adj.eq.case1b}
\nu^*(K_\FF+\Delta) = K_{W^\nu} + \Theta,
\end{equation}
where $\Theta$ is the foliation different.
Moreover, we know that $\nu^{-1}(Z) = \nklt(W^\nu, \Theta)$, where $Z$ is the union of all lc centers of $(\FF, \Delta)$ strictly contained in $W$, see \cite[Lemma 3.16]{CS18}.
Hence, by Proposition \ref{mh.nonlc.dlt.prop}, it follows that there exists a non-constant map $f \colon \mathbb{A}^1 \to (W^\nu \setminus \nklt(W^\nu, \Theta))$.
As $\nu^{-1}(Z)= \nklt(W^\nu)$, then $\nu \circ f$ produces the desired curve in $W \setminus Z$.

{\bf Case 1.c}.
If $W$ is a curve contained in $\Sing(\FF)$ and $\FF$ is canonical along $W$, then by \cite[Lemma 3.14]{CS18} there exist two possibly formal separatrices of $\FF$ through $W$ and we can choose one of them, say $S$, to be the (convergent) strong separatrix, see~\cite[Corollary~5.6]{Spicer17}. 
Hence, applying adjunction along $S$, it follows that
\[
\nu^*(K_\FF+\Delta)= K_{S^\nu} + W + \Theta,
\] 
where $\nu \colon S^\nu \to S$ is the normalization of $S$ and $W+\Theta$ is the different of $(\FF, \Delta)$ along $S^\nu$.

So
by Lemma \ref{adjunction} we see that if $P$ is a non-klt centre of $(S^\nu, W+\Theta)$ then $\nu(P)$ is an lc centre
of $(\cal F, \Delta)$.  Let $n\colon V \rightarrow W$ the normalization of $W$ and by (classical) adjunction
we may write $n^*(K_{S^\nu}+W+\Theta) = K_V+\Theta_V$ where $\lfloor \Theta_V \rfloor$ is supported on the pre-images
of the non-klt centres of $(S^\nu, W+\Theta)$ contained in $W$.  Since $(K_{S^\nu}+W+\Theta)\cdot W <0$ we have
 $V \cong \bb P^1$ and $\lfloor \Theta_V \rfloor$ contains at most one point.
Thus we see that the normalization $W-Z'$ is $\bb P^1$ or $\bb A^1$
where $Z'$ are the strata of $\nklt(\cal F, \Delta)$ contained in $W$.

{\bf Case 1.d} 
Let $\cal F_D$ be the foliation restricted to $D$ and write $(K_{\cal F}+\Delta)\vert_D = K_{\cal F_D}+\Delta_D$.
Again, the result follows directly from Proposition \ref{hyperb.surf.lc.prop} and Lemma \ref{adjunction} (as in Case 1.a) which imply
that the normalization of $W-Z'$ contains a copy of $\bb A^1$.

{\bf Case 2}: {\it We prove the theorem when $(\FF, \Delta)$ is lc by reducing to Case 1}. 
\newline
By Theorem \ref{t_sp_dlt_mod}, we can take a dlt modification $\pi \colon Y \to X$ with
\[
K_{\FF_Y}+\Delta_Y = \pi^\ast(K_\FF+\Delta)
\]
and $\pi^{-1}(\nklt(\FF, \Delta)) = \nklt(\FF_Y, \Delta_Y)$.
If $K_\FF+\Delta$ is not nef, then the same must hold for $K_{\FF_Y}+\Delta_Y$.
\newline
We first discuss the case where $K_\FF+\Delta$ is nef along $\nklt(\FF, \Delta)$.
Under this assumption, Corollary \ref{fund.cor} implies the existence of a non-constant map $f \colon \mathbb{A}^1 \to (Y \setminus \nklt(\FF_Y, \Delta_Y))$ whose image is tangent to $\mathcal F$.
This produces the desired contradiction.
Hence, we may assume that there is an lc center $W_Y \subset Y$ of $(\FF_Y, \Delta_Y)$ and a non-constant algebraic morphism $f \colon \mathbb{A}^1 \to W_Y$ such that
\begin{itemize}
    \item $f(\mathbb{A}^1) \subset (W_Y \setminus Z_Y)$ where $Z_Y$ is the union of all lc centers strictly contained in $W_Y$, and
    \item $(K_{\FF_Y}+\Delta_Y) \cdot C <0$, where $C$ is the Zariski closure of $f(\mathbb{A}^1)$.
\end{itemize}
We define $W := \pi(W_Y)$: this is an lc center of $(\FF, \Delta)$.
We wish to show the existence of a non constant morphism $g \colon \mathbb{A}^1 \to (W \setminus Z)$, where $Z$ is the union of all lc centers in $W$.

Let $Z^0 \subset Z$ be the union of all those lc centres $Z'$ in $W$ such that $\pi^{-1}(Z')$ is a union of lc centres.
By our above work we see
that $(\pi \circ f)(\bb A^1) \subset W - Z^0$.  

Notice moreover that if $Z'$ is an lc centre such that $\pi^{-1}(Z')$
is pure codimension $1$ then $Z' \subset Z^0$.

We argue in cases based on the dimension of $W$.  If $\text{dim}(W) = 0$ there is nothing to show,
so suppose for the moment that $\text{dim}(W) = 1$

Let $T$ be a codimension 1 lc centre of $(\cal F_Y, \Delta_Y)$ dominating $W$ and which contains $f(\bb A^1)$
and denote by $\sigma\colon T \rightarrow W$ the projection.

Suppose first that $T$ is transverse to $\cal F_Y$ and write by adjunction 
$(K_{\cal F_Y}+\Delta_Y)\vert_T = K_{\cal G}+\Theta$ and let $C$ be as above,
notice that $(\cal G, \Theta)$ is F-dlt.
Set $\Theta_0$ to be the part of $\lfloor \Theta \rfloor$ supported on $\sigma^{-1}(Z)$
and set $\Theta_1$ to be those divisors $D$ contained in $\sigma^{-1}(Z)$ with $\epsilon(D) = 1$ and $D$
is not contained in the support of $\lfloor \Theta \rfloor$.

Fix $0<\epsilon , \delta \ll 1$ and run the $(K_{\cal G}+\Theta-\epsilon \Theta_0+\delta \Theta_1)$-MMP 
over $W$ and denote it by $\phi\colon T \rightarrow S$.
Let $\cal H$ be the pushforward of $\cal G$, let $D = \phi_*C$, let $\Gamma =\phi_*\Theta$ and let $\tau\colon S \rightarrow W$
denote the induced map.
We have that $(\cal H, \Gamma)$ is log canonical and by the Negativity Lemma,~\cite[Lemma~1.3]{1907.06705}, we see
that $\tau^{-1}(Z) \subset \Supp(\Gamma)$, and so the pre-image of an lc centre is a union of lc centres.  
Since $(K_{\cal H}+\Gamma)\cdot D<0$ it follows from 
Proposition \ref{hyperb.surf.lc.prop} that there exists a map $\bb A^1 \rightarrow S - \nklt(\cal H, \Gamma)$
and we may pushforward this map along $\tau$ to give a map $\bb A^1 \rightarrow W-Z$.

The case where $T$ is invariant can be proven in a similar manner.

Now suppose that $\text{dim}(W) = 2$.  Let $W_Y$ denote the strict transform of $W$
and let $\sigma\colon W_Y \rightarrow W$ be the induced map. 

Suppose first that $W$ is transverse to the foliation and write 
$(K_{\cal F_Y}+\Delta_Y)\vert_{W_Y} = K_{\cal G}+\Theta$. Let $C \subset W_Y$ be a $1$ dimensional lc centre
and observe by foliated Riemann-Hurwitz that if $B \subset W_Y$ is a divisor such that $\sigma(B) = C$
then $B \subset W_Y$ is an lc centre of $(\cal G, \Theta)$.
Let $Q \subset W$ be the union of the zero dimensional lc centres contained in $W$,
set $\Theta_0$ to be the part of $\lfloor \Theta \rfloor$ supported on $\sigma^{-1}(Q)$
and set $\Theta_1$ to be those divisors $D$ contained in $\sigma^{-1}(Z)$ with $\epsilon(D) = 1$ and $D$
is not contained in the support of $\lfloor \Theta \rfloor$.

Again, we run a $K_{\cal G}+\Theta-\epsilon \Theta_0+\delta \Theta_1$
over $W$
for $0<\epsilon, \delta \ll 1$.  Let $\phi\colon W_Y \rightarrow S$ denote this MMP and let $\tau\colon S \rightarrow W$
denote the induced map.  Again, notice that the pre-image of an lc centre under $\tau$ is a union of lc centres
and by applying Proposition \ref{hyperb.surf.lc.prop} we may produce a map $\bb A^1 \rightarrow S$
whose pushforward along $\tau$ gives a map $\bb A^1 \rightarrow W-Z$.
Again, the case where $W_Y$ is invariant can be handled in a similar manner.

In all cases, if $K_{\cal F}+\Delta$ is not nef we have produced a map $\bb A^1 \rightarrow W-Z$.
\end{proof}

\begin{prop}
\label{hyperb.surf.lc.prop}
Let $(\FF, \Delta)$ be a log canonical foliated pair on a normal projective surface $X$. 
Assume that 
  \begin{itemize}
   \item $X$ is potentially lc,
   \item there is no non-constant morphism $f : \mathbb{A}^1 \to X \setminus \nklt(\FF, \Delta)$, and
   \item for any stratum $S$ of $\nklt(\FF, \Delta)$ there is no non-constant morphism $f : \mathbb{A}^1 \to \bar{S}$.
  \end{itemize}
Then $K_\FF + \Delta$ is nef.
\end{prop}

\begin{proof}
Assume for sake of contradiction 
that $K_\FF+\Delta$ is not nef and $(\FF, \Delta)$ satisfies all the hypotheses in the statement of the proposition.
\newline
We divide the proof into two distinct cases.
 
{\bf Case 1}: {\it We assume that $(\FF, \Delta)$ is F-dlt and we show that the above hypothesis leads to a contradiction}.
\newline
By~\cite[Theorem~3.31]{CS18}, since $K_\FF+\Delta$ is not nef, there exists a rational curve $C \subset X$ with $(K_\FF +\Delta) \cdot C < 0$ and $C$ is tangent to $\FF$.
\newline
As $C$ is $\FF$-invariant we see that $C$ cannot be contained in $\Supp(\Delta)$. 
Thus,
\[
\nu^\ast(K_{\cal F}+\Delta) = K_{C^\nu}+\Delta_{C^\nu},
\] 
where $\nu\colon C^\nu \rightarrow C$ is the normalization, and $\Supp(\lfloor \Delta_{C^\nu}\rfloor) \supset \nu^{-1}(\text{sing}(\cal F) \cup \lfloor \Delta \rfloor)$.
\newline
Finally, observe $\nklt(\cal F, \Delta)$ and all its strata are supported on $\text{sing}(\cal F) \cup \lfloor \Delta \rfloor$
to conclude that the normalization of $C - Z'$ is $\bb P^1$ or $\bb A^1$ where $Z'$ are
all the strata of $\nklt(\FF, \Delta)$ meeting $C$.  This is our desired contradiction.
 
{\bf Case 2}: {\it We assume that $(\FF, \Delta)$ is lc and we reduce the proof to Case 1}.
\newline
Let $\pi\colon Y \to X$ be an F-dlt modification for the pair $(\FF, \Delta)$, 
\[
K_{\FF_Y}+\Gamma = f^\ast (K_\FF+\Delta).
\]
Hence, also $K_{\FF_Y}+\Gamma$ is not nef and by Case 1 there is  is a rational curve $C \subset Y$ tangent to $\FF_Y$ such that $C \cdot (K_{\FF_Y}+ \Gamma) < 0$;
moreover, the normalization morphism $C^\nu \to Y$ induces either a non-constant morphism $f \colon \mathbb{A}^1 \to Y \setminus \nklt(\FF_Y, \Gamma)$ or a non-constant morphism $f \colon \mathbb{A}^1 \to \bar{S}$, for some stratum $S$ of $\nklt(\FF_Y, \Gamma)$.
The curve $\pi(C)$ is tangent to $\FF$, thus, it is $\FF$-invariant, since $\FF_Y$ has rank 1.
If $C \cap (Y \setminus \nklt(\FF_Y, \Gamma)) \neq \emptyset$, it follows from Theorem \ref{t_sp_dlt_mod} and adjunction that $\pi \circ f \colon \mathbb{A}^1 \to X \setminus \nklt(\FF, \Delta)$ is a well-defined morphism.
\newline
Hence we can assume that $C$ is an lc center of $(\FF_Y, \Gamma)$ and that $\bar{C}$ is a copy of $\mathbb{A}^1$ embedded in $Y$.
But then, again, the adjunction formula and Theorem \ref{t_sp_dlt_mod} imply that $\pi(\bar{C})$ is also a copy of $\mathbb{A}^1$ embedded in $X$, thus proving the proposition.
\end{proof}

\section{Some questions}

The proof of Theorem \ref{t_lc_sep} and its generalizations and possible applications raise
several questions.

\begin{q}
\label{q_klt_lc_sep}
Let $0 \in X$ be a germ of a klt singularity and $\cal F$ a log canonical co-rank one foliation on $X$.
Does $\cal F$ admit a separatrix at $0$?
\end{q}

\begin{q}
\label{q_cy_invariant}
Let $\cal F$ be a co-rank 1 foliation on a klt variety $(X, \Delta)$ with $c_1(K_{\cal F}) = 0$ and $-(K_X+\Delta)$ big.
\begin{enumerate}
\item \label{i_1} Does $\cal F$ admit an invariant divisor? 
\item \label{i_3} Is $\text{sing}(\cal F)$ non-empty?
\item \label{i_2} For $p \in \text{sing}(\cal F)$ is every separatrix at $p$ algebraic?
\end{enumerate}
\end{q}

More generally, one may wonder if log canonical singularities of foliations all dimension admit separatices.
By examples of Gomez-Mont and Luengo \cite{GomezMontLuengo92} it is known that a vector field on $\bb C^3$ does
not always admit a separatrix, however the examples given there are not log canonical.

\begin{q}
Let $\cal F$ be a foliation of any rank on $\bb C^m$.  Let $0$ be an log canonical singularity
of $\cal F$.  Does $\cal F$ admit a separatrix at $0$?
\end{q}

In the proof of existence of flips given in \cite{CS18} the existence of separatrices played a central role, and thus the methods
given there do not immediately imply the existence of log canonical flips.  With Theorem \ref{t_lc_sep} in mind we ask the following.

\begin{q}
Do log canonical foliation flips exist?
\end{q}

This extension seems to be important to apply the methods of the foliated MMP to several classes of folations of interest:
Fano foliations, for instance, have worse than canonical singularities.

\bibliography{math}
\bibliographystyle{amsalpha}


\end{document}